\documentclass[10pt,letterpaper]{article}

\usepackage{latexsym}
\usepackage{amsthm}
\usepackage{amssymb}
\usepackage{amsmath}
\usepackage{amsfonts}
\usepackage{mathtools}
\usepackage{nicefrac}
\usepackage[mathscr]{euscript}

\usepackage{color,url,bm,cite,microtype}
\usepackage{verbatim,listings}
\usepackage{todonotes}

\usepackage[font=small, labelfont=bf]{caption}
\usepackage{setspace}
\usepackage{subcaption}
\usepackage{setspace}

\usepackage[
		pdftex, plainpages = false, pdfpagelabels, 
                 bookmarks = true,
                 bookmarksopen = true,
                 bookmarksnumbered = true,
                 breaklinks = true,
                 linktocpage,
                 pagebackref,
                 hyperindex = true,
                 hyperfigures
                  ]
                  {hyperref}
                     
\usepackage[inline]{enumitem}
\usepackage{dsfont}
\usepackage{tikz}
\usepackage[ruled,linesnumbered]{algorithm2e}

\RequirePackage[T1]{fontenc}
\RequirePackage[utf8]{inputenc}
\usepackage[english]{babel}
\usepackage{alphabeta}

\makeatletter
\input{colordvi}

\usepackage{aliascnt}

\definecolor{MidnightBlack}{rgb}{0.1,0.1,.34}
\definecolor{MidnightBlue}{rgb}{0.1,0.1,0.43}
\definecolor{Black}{rgb}{0,0, 0}
\definecolor{Blue}{rgb}{0, 0 ,1}
\definecolor{Red}{rgb}{1, 0 ,0}
\definecolor{White}{rgb}{1, 1, 1}
\definecolor{grey}{rgb}{.6, .6, .6}
\definecolor{Mygreen}{rgb}{.0, .7, .0}
\definecolor{Yellow}{rgb}{.55,.55,0}
\definecolor{Mustard}{rgb}{1.0, 0.86, 0.35}
\definecolor{applegreen}{rgb}{0.55, 0.71, 0.0}
\definecolor{darkturquoise}{rgb}{0.0, 0.81, 0.82}
\definecolor{celestialblue}{rgb}{0.29, 0.59, 0.82}
\definecolor{green_yellow}{rgb}{0.68, 1.0, 0.18}
\definecolor{crimsonglory}{rgb}{0.75, 0.0, 0.2}
\definecolor{darkmagenta}{rgb}{0.30, 0.0, 0.30}
\definecolor{magenta}{rgb}{0.50, 0.0, 0.50}
\definecolor{internationalorange}{rgb}{1.0, 0.31, 0.0}
\definecolor{darkorange}{rgb}{1.0, 0.55, 0.0}
\definecolor{ao}{rgb}{0.0, 0.5, 0.0}
\definecolor{awesome}{rgb}{1.0, 0.13, 0.32}
\definecolor{darkcyan}{rgb}{0.0, 0.50, 0.50}
\definecolor{violet}{rgb}{0.93, 0.51, 0.93}
\definecolor{brown}{rgb}{0.65, 0.16, 0.16}
\definecolor{orange}{rgb}{1.0, 0.65, 0.0}
\definecolor{cornflowerblue}{rgb}{0.39, 0.58, 0.93}
\definecolor{purpleMadness}{rgb}{0.41, 0.28, 0.69}

\makeatletter
\renewcommand{\todo}[2][]{\tikzexternaldisable\@todo[#1]{#2}\tikzexternalenable}
\makeatother

\newcommand{\remove}[1]{}

\newcounter{func}
\newcommand{\funref}[1]{\hyperref[#1]{f_{\ref*{#1}}}} 

\usetikzlibrary{math}
\tikzset{black node/.style={draw, circle, fill = black, minimum size = 5pt, inner sep = 0pt}}
\tikzset{white node/.style={draw, circlternary_treese, fill = white, minimum size = 5pt, inner sep = 0pt}}
\tikzset{normal/.style = {draw=none, fill = none}}
\tikzset{lean/.style = {draw=none, rectangle, fill = none, minimum size = 0pt, inner sep = 0pt}}
\usetikzlibrary{decorations.pathreplacing}
\usetikzlibrary{arrows.meta}
\usetikzlibrary{calc}

\usetikzlibrary{shapes}
\tikzset{diam/.style={draw, diamond, fill = black, minimum size = 7pt, inner sep = 0pt}}

\tikzset{
	position/.style args={#1:#2 from #3}{
		at=($(#3)+(#1:#2)$)
	}
}

\tikzset{
  v:main/.style = {draw, circle, scale=0.8, thick,fill=black,inner sep=0.7mm},
  v:ghost/.style = {inner sep=0pt,scale=1},
  v:marked/.style = {circle, scale=1.3, fill=DarkGoldenrod,opacity=0.4},
  >={latex},
  e:main/.style = {line width=1pt}
}

\newcommand{\Fcal}{\mathcal{F}}
\newcommand{\Gcal}{\mathcal{G}}
\newcommand{\Hcal}{\mathcal{H}}

\newcommand{\Ocal}{\mathcal{O}}

\newcommand{\Zcal}{\mathcal{Z}}

\newcommand{\Nbbb}{\mathbb{N}}

\RequirePackage{stmaryrd}
\usepackage{textcomp}
\DeclareUnicodeCharacter{2286}{\subseteq}
\DeclareUnicodeCharacter{2192}{\ifmmode\to\else\textrightarrow\fi}
\DeclareUnicodeCharacter{2203}{\ensuremath\exists}
\DeclareUnicodeCharacter{183}{\cdot}
\DeclareUnicodeCharacter{2200}{\forall}
\DeclareUnicodeCharacter{2264}{\leq}
\DeclareUnicodeCharacter{2265}{\geq}
\DeclareUnicodeCharacter{8614}{\mathbin{\mapsto}}
\DeclareUnicodeCharacter{8656}{\Leftarrow}
\DeclareUnicodeCharacter{8657}{\Uparrow}
\DeclareUnicodeCharacter{8658}{\Rightarrow}
\DeclareUnicodeCharacter{8659}{\Downarrow}
\DeclareUnicodeCharacter{8669}{\rightsquigarrow}
\newcommand{\eqdef}{\stackrel{{\scriptsize\rm def}}{=}}
\DeclareUnicodeCharacter{8797}{\eqdef}
\DeclareUnicodeCharacter{8870}{\vdash}
\DeclareUnicodeCharacter{8873}{\Vdash}
\DeclareUnicodeCharacter{22A7}{\models}
\DeclareUnicodeCharacter{9121}{\lceil}
\DeclareUnicodeCharacter{9123}{\lfloor}
\DeclareUnicodeCharacter{9124}{\rceil}
\DeclareUnicodeCharacter{2208}{\in}
\DeclareUnicodeCharacter{9126}{\rfloor}
\DeclareUnicodeCharacter{9655}{\triangleright}
\DeclareUnicodeCharacter{9665}{\triangleleft}
\DeclareUnicodeCharacter{9671}{\diamond}
\DeclareUnicodeCharacter{9675}{\circ}
\DeclareUnicodeCharacter{10178}{\bot}
\DeclareUnicodeCharacter{10214}{} 
\DeclareUnicodeCharacter{10215}{} 
\DeclareUnicodeCharacter{10229}{\longleftarrow}
\DeclareUnicodeCharacter{10230}{\longrightarrow}
\DeclareUnicodeCharacter{10231}{\longleftrightarrow}
\DeclareUnicodeCharacter{10232}{\Longleftarrow}
\DeclareUnicodeCharacter{10233}{\Longrightarrow}
\DeclareUnicodeCharacter{10234}{\Longleftrightarrow}
\DeclareUnicodeCharacter{10236}{\longmapsto}
\DeclareUnicodeCharacter{10238}{\Longmapsto} 
\DeclareUnicodeCharacter{10503}{\Mapsto}    
\DeclareUnicodeCharacter{10971}{\mathrel{\not\hspace{-0.2em}\cap}}
\DeclareUnicodeCharacter{65294}{\ldotp}
\DeclareUnicodeCharacter{65372}{\mid}

\definecolor{Red}{rgb}{1, 0 ,0}
\definecolor{Blue}{rgb}{0, 0 ,1}

\newtheorem{theorem}{Theorem}[section]

\newaliascnt{question}{theorem}

\aliascntresetthe{question}

\newaliascnt{lemma}{theorem}
\newtheorem{lemma}[lemma]{Lemma}
\aliascntresetthe{lemma}

\newaliascnt{claim}{theorem}
\newtheorem{claim}[claim]{Claim}
\aliascntresetthe{claim}

\newaliascnt{invariant}{theorem}

\aliascntresetthe{invariant}

\newaliascnt{proposition}{theorem}
\newtheorem{proposition}[proposition]{Proposition}
\aliascntresetthe{proposition}

\newaliascnt{observation}{theorem}
\newtheorem{observation}[observation]{Observation}
\aliascntresetthe{observation}

\newaliascnt{corollary}{theorem}
\newtheorem{corollary}[corollary]{Corollary}
\aliascntresetthe{corollary}

\newaliascnt{definition}{theorem}
\newtheorem{definition}[definition]{Definition}
\aliascntresetthe{definition}

\newaliascnt{conjecture}{theorem}

\aliascntresetthe{conjecture}

\newaliascnt{counterexample}{theorem}

\aliascntresetthe{counterexample}

\newcommand{\hh}{\end{document}}

\newcommand{\p}{{\sf p}}
\newcommand{\cw}{{\sf cw}\xspace}%
\newcommand{\obs}{{\sf obs}}

\newcommand{\excl}{{\sf excl}}

\newcommand{\w}{{\sf w}}
\newcommand{\core}{\textsf{core}}
\newcommand{\dom}{\textsf{dom}}

\newcommand{\cobs}{\mbox{\rm \textsf{cobs}}}
\newcommand{\pobs}{\mbox{\rm \textsf{pobs}}}
\newcommand{\up}[2]{\mathsf{Up}_{#1}(#2)}
\newcommand{\down}[2]{\mathsf{Down}_{#1}(#2)}
\newcommand{\B}{\mathbb{B}}

\newcommand{\gall}{\mathcal{G}_{{\text{\rm  \textsf{all}}}}}
\newcommand{\gplanar}{\mathcal{G}_{\text{\rm  \textsf{planar}}}}

\newcommand{\gforest}{\mathcal{G}_{\text{\rm  \textsf{forest}}}}
\newcommand{\gtforest}{\mathcal{G}_{\text{\rm  \textsf{subcubic forest}}}}

\newcommand{\gouterplanar}{\mathcal{G}_{\text{\rm \textsf{outerplanar}}}}
\newcommand{\gfapex}{\mathcal{G}_{\text{\rm \textsf{apex forest}}}}

\newcommand{\nton}{\mathbb{N}\to\mathbb{N}\xspace}
\newcommand{\tw}{{\sf tw}\xspace}
\newcommand{\bipw}{{\sf bi\mbox{-}pw}\xspace}
\newcommand{\pw}{{\sf pw}\xspace}

\newcommand{\cupall}{{\pmb{\bigcup}}}

\newenvironment{cproof}{\proof[Proof of claim]}{\endproof}
\newcommand*\samethanks[1][\value{footnote}]{\footnotemark[#1]}

\newcommand{\FPT}{\textsf{FPT}\xspace}

\newcommand{\numen}[1]{\ifthenelse{\not\equal{#1}{1}}{#1}{}}

\definecolor{vagelisColour}{RGB}{0, 65, 130}

\newcommand{\closure}[2]{{\downarrow}_{#1}{#2}}
\newcommand{\upclosure}[2]{{\uparrow}_{#1}{#2}}
\newcommand{\m}[2]{{\sf min}_{#1}(#2)}

\newcommand{\idl}[2]{{\sf Idl}_{#1}(#2)}
\newcommand{\UNB}{\mathbb{U}}
\newcommand{\REC}{{\sf REC}}

\newcommand{\yes}{{\sf yes}}

\usepackage[letterpaper, top=.95in, bottom=1in, left=1.2in, right=1.2in, marginparwidth=2.3cm, marginparsep=1mm, includefoot]{geometry}

\usepackage{marginnote}
\usepackage{titlesec}
\usepackage{cleveref}
\usepackage{chngcntr}

\hypersetup{
    colorlinks,
    linkcolor   = purpleMadness,
    citecolor   = purpleMadness,
    urlcolor    = purpleMadness,
    anchorcolor = purpleMadness
}

\usepackage{tikzscale}

\usepackage{graphicx}

\begin{document}

\title{Graph Parameters, Universal Obstructions, and WQO\thanks{All authors are supported by the French-German Collaboration ANR/DFG Project UTMA (ANR-20-CE92-0027). The third author is also  supported  by the MEAE and the MESR, via the Franco-Norwegian project PHC Aurora project n. 51260WL (2024).}}

\author{ 
Christophe Paul\thanks{LIRMM, Univ Montpellier, CNRS, Montpellier, France.}
\and 
Evangelos Protopapas\samethanks[2]
\and
Dimitrios  M. Thilikos\samethanks[2]
}
\date{}

\maketitle
\thispagestyle{empty}

\begin{abstract} 
\noindent 
We establish a parametric framework for obtaining obstruction characterizations of graph parameters with respect to a quasi-ordering $\leqslant$ on graphs.
At the center of this framework lies the concept of a \textsl{$\leqslant$-parametric graph}: a non $\leqslant$-decreasing sequence $\mathscr{G} = \langle \mathscr{G}_{t} \rangle_{t \in \Nbbb}$ of graphs indexed by non-negative integers.
Parametric graphs allow us to define combinatorial objects that capture the approximate behaviour of graph parameters.
A finite set $\mathfrak{G}$ of $\leqslant$-parametric graphs is a \textsl{$\leqslant$-universal obstruction} for a parameter $\mathsf{p}$ if there exists a function $f \colon \Nbbb \to \Nbbb$ such that, for every $k \in \Nbbb$ and every graph $G$, 1) if $\p(G) \leq k$, then for every $\mathscr{G} \in \mathfrak{G},$ $\mathscr{G}_{f(k)} \not\leqslant G$, and 2) if for every $\mathscr{G} \in \mathfrak{G},$ $\mathscr{G}_{k} \not\leqslant G$, then $\p(G) \leq f(k).$
To solidify our point of view, we identify sufficient order-theoretic conditions that guarantee the existence of universal obstructions and in this case we examine algorithmic implications on the existence of fixed-parameter tractable algorithms.
Our parametric framework has further implications related to finite obstruction characterizations of properties of graph classes.
A \textsl{$\leqslant$-class property} is defined as any set of $\leqslant$-closed graph classes that is closed under set inclusion.
By combining our parametric framework with established results from order theory, we derive a precise order-theoretic characterization that ensures $\leqslant$-class properties can be described in terms of the exclusion of a finite set of $\leqslant$-parametric graphs.
\end{abstract}


\newpage
\thispagestyle{empty}
\tableofcontents

\newpage

\section{Introduction}
\label{intr}

This study is motivated by the fact that many graph-theoretic properties admit finite characterizations that may lead to efficient recognition algorithms. 
A classic example is planarity: a graph is \emph{planar} if it can be embedded in the plane, i.e., it can be drawn in the plane in a way such that no two of its edges cross.
Certainly, there are infinitely many planar graphs, while the definition of graph planarity is ``purely topological''.
Nevertheless, it appears that the class of planar graphs admits a \textsl{finite} characterization due to the theorem of Kuratowski \cite{Kuratowski37surl}:
\begin{eqnarray}
\begin{minipage}{14cm}
\centering
\textsl{A graph is planar if and only if it excludes $K_{5}$ and $K_{3,3}$ as topological minors\footnote{A graph $H$ is a \emph{subdivision} of a graph $G$ if it can be obtained from $G$ by replacing edges with non-trivial paths. A graph $H$ is a topological minor of a graph $G$ if $G$ contains a subgraph that is isomorphic to some subdivision of $H$.}},
\end{minipage}\label{kuratowski_theorem}
\end{eqnarray}
where $K_{5}$ is the complete graph on five vertices, and $K_{3,3}$ is the complete bipartite graphs with three vertices in each part.

The finite characterization of \eqref{kuratowski_theorem} can be used to design polynomial-time planarity testing algorithms and provide an efficient certificate of non-planarity.
Whether it is possible to fully characterize some infinite graph class by the exclusion (under some particular relation on graphs) of a \textsl{finite} set of graphs is an important challenge in graph theory and has significant algorithmic implications.

\paragraph{Well-quasi-ordering and graph algorithms.}
Towards formalizing the above, let $\leqslant$ be a \textsl{quasi-ordering} relation on graphs and let $\mathcal{G}$ be a graph class.
We say that $\mathcal{G}$ is $\leqslant$-\emph{closed} if for every two graphs $H$ and $G$, $H \leqslant G$ and $G \in \mathcal{G}$ imply $H \in \mathcal{G}.$
The $\leqslant$-\emph{obstruction set} of $\mathcal{G}$, denoted by $\obs_{\leqslant}(\mathcal{G})$, is defined as the set of $\leqslant$-minimal graphs not in $\mathcal{G}$.

Now, consider a graph class $\mathcal{G}$ and a quasi-ordering $\leqslant$ on graphs that satisfy the following three conditions:
\begin{itemize}
\item[\textbf{A}.] $\mathcal{G}$ is $\leqslant$-closed,
\item[\textbf{B}.] $\leqslant$ is a well-quasi-ordering on all graphs, and
\item[\textbf{C}.] there exists a constant $c$ and an algorithm that, given a graph $G$ outputs whether $H \leqslant G$ in time $\Ocal(|G|^{c})$, for a fixed graph $H$.
\end{itemize}

The above conditions directly imply that there exists an algorithm able to check whether $G \in \mathcal{G}$ in time $O(|G|^{c}).$
This follows from the fact that, $G \in \mathcal{G}$ if and only if for every $H \in \obs_{\leqslant}(\mathcal{G})$ (because of \textbf{A}) and that the set $\obs_{\leqslant}(\mathcal{G})$ is finite (because of \textbf{B}).
It therefore suffices to apply the algorithm of \textbf{C} a finite number of times in order to check whether the input graph contains any of the graphs in the $\leqslant$-obstruction set of $\mathcal{G}$.

The most celebrated quasi-orderings on graphs that are known to be well-quasi-orderings on all graphs are the \textsl{minor} relation and the \textsl{immersion} relation.
Given two graphs $H$ and $G$, we say that $H$ is a \emph{minor} (resp. \emph{immersion}) of $G$ if a graph isomorphic to $H$ can be obtained from a subgraph of $G$ after contracting (resp. lifting\footnote{See \autoref{gt_conc} for the definition of a contraction and \autoref{immersion_subsection} for the definition of lifting.}) edges and we denote this by $H \leqslant_{\mathsf{m}} G$ (resp. $H \leqslant_{\mathsf{i}} G$).
The fact that $\leqslant_{\mathsf{m}}$ is a wqo on all graphs is known as the Robertson-Seymour theorem and was proved by Robertson and Seymour in the 20th paper of their Graph Minors series \cite{RobertsonS04GraphMinorsXX}.
Also the fact that $\leqslant_{\mathsf{i}}$ is a wqo was proved in the 23th paper of the same series \cite{RobertsonS10GraphminorsXXIII}.
Therefore, condition~\textbf{B} holds for both relations.
Also, algorithms for condition \textbf{C} for $\leqslant_{\mathsf{m}}$ (resp. $\leqslant_{\mathsf{i}}$) were given in \cite{RobertsonS95b,KawarabayashiKR12Thedisjoint} (resp. \cite{GroheKMW11Findingtopological}). 
This directly implies that every minor-closed (resp. immersion-closed) graph class admits a finite characterization and a polynomial recognition algorithm.

Interestingly, the above -- wqo based -- reasoning appears to be inherently non-constructive as the implied algorithm requires the ``knowledge'' of the obstruction set.
According to Friedman, Robertson, and Seymour~\cite{FriedmanRS87them}, the bounded\footnote{In the statement of the ``bounded version'' of the Graph Minor theorem, the graphs are restricted to have bounded treewidth (see \autoref{bicpw} for the definition of treewidth).} version of the Robertson-Seymour theorem is equivalent to the extended version of Kruskal's theorem~\cite{Kruskal60well} that is proof-theoretically stronger than $\Pi^1_1$-CA$_{0}$ (for the meta-mathematics of the Robertson-Seymour theorem, see~\cite{FriedmanRS87them} as well as the more recent work of Krombholz and Rathjen~\cite{krombholz2019upper}).
Hence, we cannot hope for proofs of the above wqo results that also give an indication as to how we can \textsl{construct} the corresponding obstruction set.
From the algorithmic point of view, Fellows and Langston~\cite{FellowsL88nonc} proved by a reduction from the \textsc{Halting} problem, that there is no algorithm that given a finite description of a minor-closed class outputs its obstruction set.
Analogous implications can be made for the immersion relation as well.

The above results, besides (but also because of) their non-constructive nature, strongly motivated the development of algorithmic graph theory. 
Indeed, ``knowing'' that an algorithm ``exists'' certainly makes it meaningful (and somehow easier) to try to design one. 
For this, a long line of research was dedicated to the identification of $\obs_{\leqslant}(\mathcal{G})$ for particular instantiations of $\leqslant$ and $\mathcal{G}.$ 
For a small sample of such results, see \cite{LeivaditisSSTTV20mino,Holst02onth,RobertsonST95sach,BodlaenderThil99,Thilikos00,ArnborgPC90forb,Archdeacon06akur,DinneenX02mino,KinnersleyL94,FioriniHJV17thee,RobertsonST95sach} for the minor relation, \cite{GiannopoulouPRT21LinearKernels,BelmonteHKPT13Characterizing} for the immersion relation, and \cite{AdlerFP14Obstructions,Bouchet94circl,KaminskiPT11Contracting} for other relations.
In a more general direction, several researchers considered additional conditions, mostly related to logic, that can
guarantee the computability of obstruction sets \cite{SauST21apices,FellowsL94onse,LagergrenA91mini,Lagergren91anup,Lagergren98,GuptaI97boun,AbrahamsonF93,AdlerGK08comp,CourcelleDF97,Bienstock95,FellowsL88nonc,ArnborgPS90tree,DowneyF95survey}. 

We should stress that a great part of the theory of parameterized algorithms has been motivated by the results mentioned above.
In many cases, to overcome (or bypass) the non-constructibility obstacle has been a challenge that motivated the introduction of new algorithmic techniques \cite{CyganFKLMPPS15para}.

\medskip
As we mentioned above, if $\leqslant$ is a wqo on all graphs, then every $\leqslant$-closed graph class admits a finite characterization.
The question addressed in this paper is whether (and to what extent) we can finitely characterize ``second-order'' properties, i.e., properties on graph classes rather than on graphs, in terms of a suitable concept of ``second-order'' obstructions.
Formally, we call every inclusion-closed set $\mathbb{CP}$ of $\leqslant$-closed classes a \emph{$\leqslant$-class property}.

Towards this, we introduce a parametric framework based on two central concepts, \textsl{graph parameters} and \textsl{$\leqslant$-parametric graphs}.
A \emph{graph parameter} is a function $\p$ mapping graphs to non-negative integers and a \emph{$\leqslant$-parametric graph} is a $\leqslant$-monotone sequence of graphs $\mathscr{G} \coloneqq \langle \mathscr{G}_{t} \rangle_{t \in \mathbb{N}},$ i.e, for every $t \in \mathbb{N},$ $\mathscr{G}_{t} \leqslant \mathscr{G}_{t+1}.$

\paragraph{Universal obstructions for graph parameters.}

A celebrated result of Robertson and Seymour \cite{RobertsonS86GMV} is the \textsl{Grid Theorem} which certifies the existence of a ``canonical'' type of minor one should expect to find in graphs of large treewidth: \textsl{there exists a function $f \colon \Nbbb \to \Nbbb$ such that for every graph $G$ and every integer $k$, if $G$ has treewidth at least $f(k)$ then $G$ contains the $(k \times k)$-grid\footnote{The $(t \times t)$-grid $\Gamma_{k}$ is defined as the Cartesian product of two paths on $t$ vertices.} $\Gamma_{k}$ as a minor.}
This, paired with the fact that the treewidth of the $(k \times k)$-grid $\Gamma_{k}$ is $k$ (see e.g. \cite{Bodlaender98}), implies that the treewidth of a graph $G$ is ``large'' if and only if $G$ contains a ``large'' grid as a minor.\footnote{We refer the reader to \autoref{bicpw} for a definition of treewidth and an in-depth discussion on the canonical obstruction to treewidth.}
In this sense, the sequence of grids can be viewed as a ``canonical'' obstruction to treewidth, providing the first insight into formalizing the concept of ``second-order'' obstructions.

The previous discussion leads us to the concept of $\leqslant$-parametric graphs.
Certainly the sequence $\Gamma \coloneqq \langle \Gamma_{t} \rangle_{t \in \Nbbb}$ is a minor-parametric graph since we can easily observe that any grid is a minor of a larger grid.
Moreover, every $\leqslant$-parametric graph $\mathscr{H}$ may act as a ``regular'' characterization of a $\leqslant$-closed class defined as the \emph{$\leqslant$-closure} of $\mathscr{H}$ which consists of every graph $H$ such that $H \leqslant \mathscr{H}_{t},$ for some non-negative integer $t$.
Indeed, in the case of planar graphs, every grid is a planar graph and it is known that every planar graph is a minor of a (polynomially) large enough grid.
Therefore, the class of planar graphs corresponds to the minor-closure of $\Gamma.$
We call a set $\mathfrak{G}$ of $\leqslant$-parametric graphs a \emph{$\leqslant$-parametric family} if the classes that correspond to the $\leqslant$-closures of each $\leqslant$-parametric graph in $\mathfrak{G}$ are pairwise incomparable with respect to inclusion.
Moreover, given a \textsl{$\leqslant$-monotone}\footnote{A graph parameter $\mathsf{p}$ is \emph{$\leqslant$-monotone} if for every two graphs $H$ and $G,$ $H \leqslant G$ implies $\mathsf{p}(H) \leq \mathsf{q}(G).$} graph parameter $\p$ we say that a $\leqslant$-parametric family $\mathfrak{G}$ is a \emph{$\leqslant$-universal obstruction} for $\p$ if there exists a function $f \colon \Nbbb \to \Nbbb$ such that for every graph $G$ for every $k \in \Nbbb,$ 1) if $\p(G) \leq k$, then for every $\mathscr{G} \in \mathfrak{G},$ $G$ excludes $\mathscr{G}_{f(k)}$ as a minor and 2) if for every $\mathscr{G} \in \mathfrak{G},$ $\mathscr{G}_{k} \not\leqslant G$, then $\p(G) \leq f(k)$ (see \autoref{def:univ_obs}).
It is well-known that treewidth is minor-monotone.
Therefore, a result of the previous discussion is that in this terminology, the minor-parametric family $\{ \Gamma \}$ is a minor-universal obstruction for treewidth.

The first essential question towards establishing a theory of universal obstructions for graph parameters is the following.
Does every $\leqslant$-monotone parameter admit a finite $\leqslant$-universal obstruction?
In \autoref{chakmrkfinfoerobes} we prove a sufficient order-theoretic condition for this (\autoref{omega2_finite_univ_obs}): ``\textsl{if the set of all graphs is \emph{$\omega^{2}$-well-quasi-ordered}\footnote{We refer the reader to \autoref{prelim_order_theory} for a definition of $\omega^{2}$-well-quasi-ordering.} by $\leqslant$, then every $\leqslant$-monotone parameter admits a finite $\leqslant$-universal obstruction}''.
We shall not define the concept of an $\omega^{2}$-well-quasi-ordering here however we remark that it may be seen as a ``second-order'' lifting of well-quasi-ordering.

\paragraph{Class properties and graph parameters.}

We say that a graph parameter $\p$ is \emph{bounded} in a graph class $\Gcal$ if there exists a non-negative integer $c$ such that for every graph $G \in \Gcal,$ $\p(G) \leq c.$

As a consequence of the grid theorem \cite{RobertsonS86GMV}, Robertson and Seymour established the first link between graph parameters and class properties by providing an equivalent description of the \textsl{Erd\H{o}s-P{\'o}sa property}\footnote{A minor-closed class $\mathcal{H}$ has the Erd\H{o}s-P{\'o}sa property if there exists a function $f \colon \Nbbb \to \Nbbb$ such that for every graph $G$, either $G$ contains $k$ pairwise vertex-disjoint copies of a minor-obstruction to $\Hcal$ as a minor or there exists a set $S \subseteq V(G)$ of size at most $f(k)$ such that $G - S \in \mathcal{H}$.} for minor-closed classes as the minor-class property that consists of all minor-closed classes of bounded treewidth.
This result provides a natural candidate of a ``second-order'' obstruction for the Erd\H{o}s-P{\'o}sa property which is none other that a universal obstruction for treewidth: ``\textsl{a minor-closed class $\Gcal$ has the Erd\H{o}s-P{\'o}sa property if and only if it it does not contain the $(k \times k)$-grid $\Gamma_{k},$ for some non-negative integer $k.$}''

This leads us define a notion of ``representativity'' of $\leqslant$-class properties in terms of $\leqslant$-parametric families: a $\leqslant$-class property $\mathbb{CP}$ is \emph{represented} via a $\leqslant$-parametric family $\mathfrak{G}$ if a $\leqslant$-closed class $\Gcal$ belongs to $\mathbb{CP}$ if and only if for every $\mathscr{G} \in \mathfrak{G}$, $\Gcal$ does not contain $\mathscr{G}_{t},$ for some non-negative integer $t$.
Let us argue why we can see $\leqslant$-parametric families as obstructing notions for $\leqslant$-class properties.
An alternative way to understand the previous definition is that 1) for every $\mathscr{G} \in \mathfrak{G}$ the $\leqslant$-closure of $\mathscr{G}$ does not belong to $\mathbb{CP}$ and 2) there does not exists a $\mathscr{G} \in \mathfrak{G}$ such that the $\leqslant$-closure of $\mathfrak{G}$ is a subclass of any $\leqslant$-closed class $\Gcal$ that belongs to $\mathbb{CP}.$
In this light, the set of $\leqslant$-closed classes defined as the $\leqslant$-closures of each $\leqslant$-parametric graph in $\mathfrak{G}$ acts as a finite ``class obstruction set'' of $\mathbb{CP}$ with respect to set inclusion.

The natural question to ask now is: can every $\leqslant$-class property be represented via $\leqslant$-parametric families?
In \autoref{chakmrkfinfoerobes}, as a consequence of our study of universal obstructions for graph parameters, combined with established results in order theory, we derive a necessary and sufficient condition to this question (\autoref{final_result}): ``Assuming condition \textbf{B}, the set of all graphs is $\omega^{2}$-well-quasi-ordered by $\leqslant$ if and only if every $\leqslant$-class property can be expressed as the union of a finite set of $\leqslant$-class properties, each representable via a finite $\leqslant$-parametric family''.
Hereafter we refer to this condition as: ``a $\leqslant$-classed property is expressed via a finite set of finite $\leqslant$-parametric families''.

\paragraph{Algorithmic consequences of the parametric viewpoint.}

In light of the previous results, and assuming that the $\omega^{2}$-well-quasi-ordering conjecture holds, the parametric viewpoint provides a framework in pursuit of finite obstruction characterizations of class properties.
Essentially, this approach hinges on the proper identification of a set of $\leqslant$-parametric families, consisting of $\leqslant$-parametric graphs that effectively obstruct membership in the given class property.
The next question is to what extent these characterizations lead to efficient algorithms for deciding membership in class properties.

\medskip
To properly formulate this problem we define $\leqslant$-closed classes in terms of the exclusion of a finite set of graphs.
Given a finite set $\Zcal$ of graphs, we define $\excl_{\leqslant}(\mathcal{Z})$ to be the class of \emph{$\Zcal$-$\leqslant$-free} graphs, i.e., the $\leqslant$-closed class that consists of the graphs excluding every graph in $\Zcal$ with respect to $\leqslant$.
We consider the problem \textsc{Membership in class property $\mathbb{CP}$} that asks, given a finite set $\Zcal$ of graphs, whether the class of $\Zcal$-$\leqslant$-free graphs belongs to $\mathbb{CP}.$
It turns out that, if $\mathbb{CP}$ is representable by a finite set $\mathbf{G}$ of finite $\leqslant$-parametric families, then \textsc{Membership in class property $\mathbb{CP}$} again reduces to a finite number of calls to the algorithm of condition $\textbf{C}.$
To see this, note that a $\leqslant$-closed class $\Gcal \in \mathbb{CP}$ if and only if there exists $\mathfrak{G} \in \mathbf{G}$ such that for every $\mathscr{G} \in \mathfrak{G},$ $\Gcal$ does not contain the $\leqslant$-closure of $\mathscr{G},$ which can be phrased in terms of $\leqslant$-containment, given the $\leqslant$-obstruction sets for both $\Gcal$ and the $\leqslant$-closure of $\mathscr{G}.$
These algorithmic consequences are the subject of \autoref{algorithms_section}.
It is an interesting open question to develop techniques that, given the knowledge of $\mathbf{G}$, give constructive algorithms for \textsc{Membership in class property $\mathbb{CP}$} that do not require the knowledge of the $\leqslant$-obstruction sets.

\paragraph{Parameterized algorithms.}

Consider now the problem $Π_{\p}$ that, given a graph $G$ and some non-negative integer $k,$ asks whether $\p(G)≤k.$ 
A general algorithmic problem is to derive a fixed parameter tractable algorithm (in short, an \FPT-algorithm) for $Π_{\p},$ that is an algorithm running in time $O(f(k)\cdot n^{c}),$ for some fixed constant $c.$
The design of an algorithm for $Π_{\p}$ for distinct instantiations of $\p$ is one of the most important running challenges in parameterized algorithm design \cite{Thilikos12GraphMinors,CyganFKLMPPS15para}.

Notice now that the problem $Π_{\p}$ is equivalent to the problem of deciding membership in $\Gcal_{\p,k},$ for each possible $k.$ 
This implies that, if $\p$ is $\leqslant$-monotone, $\leqslant$ satisfies condition \textbf{B}, which implies that $\obs_{\leqslant}(\Gcal_{\p,k})$ is finite for every choice of $k$, and condition \textbf{C}, then $Π_{\p}$ admits a (non-constructive)  \FPT-algorithm.
Clearly, to make this algorithm constructive we need to know $\obs_{\leqslant}(\Gcal_{\p,k}),$ for every $k\in\mathbb{N},$ or at least some bound -- as a computable function of $k$ -- on the size of its graphs.
Typically, such bounds are highly non-trivial to prove and, when they exist, they are huge (see \cite{SauST21apices,LagergrenA91mini,Lagergren98,AdlerGK08comp} for examples of such bounds) and imply \FPT-algorithms with huge parametric dependences.
Given the knowledge of a universal obstruction for $\p$ and given that the ``gap'' of its asymptotic equivalence with $\p$ is reasonable, this may instead lead to an \FPT-approximation algorithm with  better parametric dependence than the exact one.  
The potential of deriving  \FPT-approximation for graph parameters, given some suitable choice of a universal obstruction is examined in \autoref{algorithms_section}.

\paragraph{Examples.}

To exemplify the introduced notions we use as running examples throughout the paper the graph parameter \textsl{pathwidth} and a variant of it, namely the \textsl{biconnected pathwidth}\footnote{See \autoref{gt_conc} for the definitions.}.
Both of these parameters admit universal obstructions.
The choice of biconnected pathwidth is based on the fact that its universal obstructions contain more than one sequence.
The same phenomenon has been observed for many other graph parameters for which similar approximate characterizations are known.
In \autoref{examples_section} we present some indicative examples of graph parameters and the corresponding results using the unifying terminology and notation that we introduce in this paper. 
We present parameters that are minor-monotone and immersion-monotone.
For an extended survey on graph parameters and their obstructions, see \cite{paul2023universal}.

\section{Background concepts}\label{prelim}

In this section we shall present notions from order theory and graph theory that are necessary for the developments in the following sections.
Most of the required order theoretic terminology can be found in~\cite{SchmitzS12Algorithmic} (see also~\cite{Rival85Graphs, Harzheim05Ordered, Marcone01Fine}).
For any undefined terminology on graph theory we refer the reader to \cite{diestel2016graph}.

\subsection{Order theory}\label{prelim_order_theory}

\paragraph{Binary relations.}

A \emph{binary relation} $\leqslant$ on a given set $X$ is a subset of the set $X \times X.$
We write $x \leqslant y$ for any pair $\langle x, y \rangle$ that is in $\leqslant$.
A binary relation $\leqslant$ is called \emph{reflexive} if for every element $x$ in $X$ we have that $x \leqslant x,$ is called \emph{transitive} if for every three elements $x,$ $y,$ and $z$ in $X$ if $x \leqslant y$ and $y \leqslant z$ then $x \leqslant z$, is called \emph{antisymmetric} if for every two elements $x$ and $y$ in $X$ if $x \leqslant y$ and $y \leqslant x$ then $x = y,$ and is called \emph{symmetric} if for every two elements $a$ and $b$ in $X$ if $a \leqslant b$ then $b \leqslant a.$

\paragraph{Quasi-orderings, partial-orderings, and equivalence relations.}

A binary relation $\leqslant$ on a set $X$ is a \emph{quasi-ordering (qo)} if it is reflexive and transitive. 
A quasi-ordering is a \emph{partial-ordering (po)} if it is also antisymmetric and it is an \emph{equivalence relation} if it also symmetric.

Every quasi-ordering $\leqslant$ on a set $X$ naturally induces an equivalence relation on $X$ which we denote by $\equiv_{\leqslant}$: $x \equiv_{\leqslant} y$ if $x \leqslant y$ and $y \leqslant x.$
Moreover, to every quasi-ordering $\leqslant$ on a set $X$ corresponds a canonical partial-ordering on the equivalence classes of $\equiv_{\leqslant}$.

\paragraph{Well-founded quasi-orderings.}

Hereafter, given a quasi-ordering $\leqslant$ on a set $X,$ we use the notation $(X, \leqslant).$
A quasi-ordering $(X, \leqslant)$ is \emph{well-founded} if for every non-empty subset $A$ of $X$ there exists a $\leqslant$-\emph{minimal element} of $A,$ i.e., an element $x$ in $A$ such that for every element $y$ in $A,$ if $y \leqslant x$ then $x \leqslant y$.

A subset $A$ of $X$ is a $\leqslant$-\emph{antichain} of $X$ if it consists of pairwise $\leqslant$-\emph{incomparable} elements, i.e., for every two elements $a$ and $b$ in $A$ we have that $a \not\leqslant b$ and $b \not\leqslant a.$
Observe that in a well-founded quasi-ordering $(X, \leqslant)$ for every subset $A$ of $X$ there exists a $\leqslant$-antichain $A_{0}$ that is a subset of $A$ such that for every element $x$ in $A$ there exists an element $y$ in $A_{0}$ such that $y \leqslant x.$
We call any such $\leqslant$-antichain $A_{0}$ a $\leqslant$-\emph{minimization} of $A.$
Observe that if moreover $(X, \leqslant)$ is a partial-ordering then for every subset $A$ of $X$ its $\leqslant$-minimization is a uniquely defined set that consists of the $\leqslant$-minimal elements of $A.$
In this setting, we denote the $\leqslant$-minimization of $A$ by $\m{\leqslant}{A}.$

\paragraph{Well-quasi-orderings.}

A quasi-ordering $(X, \leqslant)$ is a \emph{well-quasi-ordering (wqo)} if it is well-founded and $X$ contains no infinite $\leqslant$-antichains.
A well-quasi-ordering is a \emph{well-partial-ordering (wpo)} if it is also antisymmetric.

\paragraph{Dilworth's theorem.}

Now, let $(X, \leqslant)$ be a quasi-ordering.
We call the size of the largest $\leqslant$-antichain of $X$ the \emph{width} of $(X, \leqslant).$
Moreover, we call any non-decreasing sequence $x_{1} \leqslant x_{2} \leqslant x_{3} \ldots$ of elements of $X$ a $\leqslant$-\emph{chain} of $X.$
A classic result in order theory by Dilworth relates infinite partial-orderings of fixed finite width with a partition of its universe into as many chains.

\begin{theorem}[Dilworth's Theorem]\label{dilworth} Let $(X, \leqslant)$ be a partial-ordering.
Then $(X, \leqslant)$ has finite width $w$ if and only if it can be partitioned into $w$ many $\leqslant$-chains.
\end{theorem}

\paragraph{Closed sets.}

Next, let $(X, \leqslant)$ be a quasi-ordering.
Given a subset $A$ of $X$ we define the \emph{upward $\leqslant$-closure} (resp. \emph{downward $\leqslant$-closure}) of $A$ in $X$ as the set
$$\text{$\upclosure{\leqslant}{A} \coloneqq \{ x \in X \mid \exists y \in A : y \leqslant x \}$ (resp. $\closure{\leqslant}{A} \coloneqq \{ x \in X \mid \exists y \in A : x \leqslant y \}$).}$$

When $A = \upclosure{\leqslant}{A}$ (resp. $A = \closure{\leqslant}{A}$) we say that $A$ is \emph{upward} $\leqslant$-\emph{closed} (resp. \emph{downward} $\leqslant$-\emph{closed}) in $X$.
We denote by $\up{\leqslant}{X}$ (resp. $\down{\leqslant}{X}$) the set of all upward $\leqslant$-closed (resp. downward $\leqslant$-closed) subsets of $X$.
The following proposition gives an equivalent characterization of a well-quasi-ordering by lifting the required condition to the set of all downward-closed sets being well-founded by the standard inclusion relation on sets.

\begin{proposition}[\!\! \cite{SchmitzS12Algorithmic}]\label{wqo_wellfouned} Let $(X, \leqslant)$ be a quasi-ordering.
Then $(X, \leqslant)$ is a well-quasi-ordering if and only if $(\down{\leqslant}{X}, \subseteq)$ is well-founded.
\end{proposition}

\paragraph{Bases of upward closed sets.}

Let $(X, \leqslant)$ be a quasi-ordering.
Consider an upward $\leqslant$-closed subset $U \in \up{\leqslant}{X}$ of $X.$
A $\leqslant$-antichain $U_{0}$ such that $U = \upclosure{\leqslant}{U_{0}}$ is called a $\leqslant$-\emph{basis} of $U$ in $X$.
In the case that $(X, \leqslant)$ is a partial-ordering, for every upward $\leqslant$-closed subset $U$ of $X$ that admits a $\leqslant$-basis $U_{0}$ we have that $U_{0} = \m{\leqslant}{U}.$
If moreover $(X, \leqslant)$ is well-founded, then every upward $\leqslant$-closed subset $U$ of $X$ admits a $\leqslant$-basis $U_{0}$ and if moreover $(X, \leqslant)$ is well-quasi-ordered, $U_{0}$ is finite.

The finite representation of upward $\leqslant$-closed subsets via their $\leqslant$-bases is called the \emph{finite basis property} of a well-quasi-ordering
$(X, \leqslant).$
This property can also serve as a finite representation of downward $\leqslant$-closed subsets in terms of excluding a finite $\leqslant$-basis of its complement: every downward $\leqslant$-closed subset $D \in \down{\leqslant}{X}$ of $X$ is equal to $X \setminus \upclosure{\leqslant}{U_{0}},$ where $U_{0}$ is a $\leqslant$-basis of the upward $\leqslant$-closed subset $X \setminus D.$

\paragraph{Ideals.}

There is another way to finitely represent downward $\leqslant$-closed sets in a well-quasi-ordering $(X, \leqslant).$
A non-empty subset $A$ of $X$ is \emph{$\leqslant$-directed} if for every two elements $a$ and $b$ in $A$ there exists an element $z$ in $A$ such that $x \leqslant z$ and $y \leqslant z.$
A downward $\leqslant$-closed subset of $X$ that is $\leqslant$-directed is called a $\leqslant$-\emph{ideal} of $X.$
We write $\idl{\leqslant}{X}$ for the set of all $\leqslant$-ideals of $X.$

In well-quasi-ordering theory $\leqslant$-ideals are what we call \emph{irreducible subsets}: for every $\leqslant$-ideal $I,$ if $I \subseteq D_{1} \cup D_{2},$ where both $D_{1}$ and $D_{2}$ are downward $\leqslant$-closed subsets of $X,$ then either $I \subseteq D_{1}$ or $I \subseteq D_{2}.$
Exploiting the irreducibility property of $\leqslant$-ideals one can show that in a well-quasi-ordering every downward $\leqslant$-closed set uniquely corresponds to a finite set of $\leqslant$-ideals as follows.

\begin{proposition}[\!\! \cite{SchmitzS12Algorithmic}]\label{finite_closed_set_decomposition}
Let $(X, \leqslant)$ be a well-quasi-ordering.
Then every downward $\leqslant$-closed subset of $X$ is equal to the union of a unique finite inclusion-antichain of $\leqslant$-ideals of $X.$
\end{proposition}

In fact more is known.

\begin{proposition}[Lemma 5.3, \cite{KabilP92Une}]\label{ideal_decomposition_implies_wqo}
Let $(X, \leqslant)$ be a quasi-ordering.
Then, every downward $\leqslant$-closed subset of $X$ is equal to the union of a finite set of $\leqslant$-ideals of $X$ if and only if $(X, \leqslant)$ contains no infinite $\leqslant$-antichains.
\end{proposition}

As a result the following two propositions hold.

\begin{proposition}[\!\! \cite{SchmitzS12Algorithmic}]\label{countable_wqo}
Let $(X, \leqslant)$ be a countable well-quasi-ordering.
Then $\down{\leqslant}{X},$ $\up{\leqslant}{X}$ and $\idl{\leqslant}{X}$ are countable.
\end{proposition}

\begin{proposition}[\!\! \cite{SchmitzS12Algorithmic}]\label{idl_wqo_down_wqo}
Let $(X, \leqslant)$ be a well-quasi-ordering.
Then $(\idl{\leqslant}{X}, \subseteq)$ is a well-quasi-ordering if and only if $(\down{\leqslant}{X}, \subseteq)$ is a well-quasi-ordering.
\end{proposition}

\paragraph{``Lifting'' relations to powersets.}

Finally we present a general mechanism which allows to ``lift'' a binary relation on $X$ to its powerset, which we denote by $2^{X}.$

\begin{definition}[Smyth extension]\label{smyth}
Let $\leqslant$ be a binary relation on a set $X.$
The \emph{Smyth extension} of $\leqslant$ on $2^{X}$, denoted by $\leqslant^{*}$, is defined as follows.
For every two subsets $A$ and $B$ of $X$:
$$\text{$A \leqslant^{*} B$ if and only if for every $b \in B$ there exists $a \in A$ such that $a \leqslant b$}.$$
\end{definition}

One may easily observe that the Smyth extension of a quasi-ordering maintains being a quasi-ordering: given a quasi-ordering $(X, \leqslant),$ $(2^{X}, \leqslant^{*})$ is a quasi-ordering.
However, in the case that $(X, \leqslant)$ is a partial-ordering, $(2^{X}, \leqslant^{*})$ may fail to be a partial-ordering.

The following proposition in order theory (see e.g. \cite{Marcone01Fine}) shows how the Smyth extension of a quasi-ordering relation translates to ordering downward $\leqslant$-closed subsets by inclusion.

\begin{proposition}[\!\! \cite{Marcone01Fine}]\label{exclminorssmith} Let $(X, \leqslant)$ be a quasi-ordering.
Then for every two subsets $A$ and $B$ of $X$:
$$\text{$A \leqslant^{*} B$ if and only if $X \setminus \upclosure{\leqslant} A \subseteq X \setminus \upclosure{\leqslant} B.$}$$
\end{proposition}

\paragraph{$\omega^{2}$-well-quasi-orderings.}

Finally, one may wonder whether $(\down{\leqslant}{X}, \subseteq)$ is a well-quasi-ordering, given a well-quasi-ordering $(X, \leqslant).$
The answer to this question is negative as $(\down{\leqslant}{X}, \subseteq)$ may indeed fail to be a well-quasi-ordering.
The classic counterexample was given by Rado.

\begin{figure}[htbp]
\centering
\includegraphics[width=0.45\linewidth]{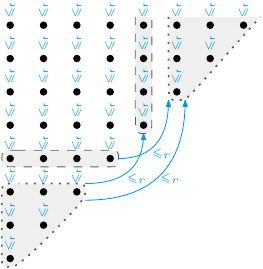}
\caption{\label{radofig}An illustration of the Rado structure.}
\end{figure}

The \emph{Rado structure} (see~\autoref{radofig}) is defined as the pair
$$(\{ (i, j) \in \mathbb{N} \times \mathbb{N} \mid i < j\}, \leqslant_{\mathsf{r}}),$$
where
$$\text{$(i, j) \leqslant_{\mathsf{r}} (i', j')$ if and only if ($i = i'$ and $j ≤ j'$) or ($j < i'$)}.$$

This leads to the definition of a ``second-order lifting'' of a well-quasi-ordering (see \cite{Jancar99ANote, marcone1994foundations}).

\begin{definition}[\!\! \cite{Jancar99ANote}]\label{omega2_wqo}
A quasi-ordering $(X, \leqslant)$ is an $\omega^{2}$-\emph{well-quasi-ordering} if it is a well-quasi-ordering and does not contain an isomorphic copy of the Rado structure.
\end{definition}

The following Theorem due to Jančar in \cite{Jancar99ANote} (see also the work of Marcone in \cite{Marcone01Fine}) gives a characterization of $\omega^{2}$-well-quasi-orderings $(X, \leqslant)$ in terms of the well-quasi-ordering status of subsets of $X$ ordered by the Smyth extension of $\leqslant$.

\begin{theorem}[\!\! \cite{Jancar99ANote}]\label{powerquasi}
Let $(X, \leqslant)$ be a quasi-ordering.
Then $(2^{X}, \leqslant^{*})$ is a well-quasi-ordering if and only if $(X, \leqslant)$ is an $\omega^{2}$-well-quasi-ordering.
\end{theorem}

Notice that \autoref{exclminorssmith} allows for comparing downward $\leqslant$-closed subsets in terms of the Smyth extension on subsets.
This combined with \autoref{powerquasi} gives the following corollary:

\begin{corollary}\label{subsetwqo}
Let $(X, \leqslant)$ be a quasi-ordering.
Then $(\down{\leqslant}{X}, \subseteq)$ is a well-quasi-ordering if and only if $(X, \leqslant)$ is an $\omega^{2}$-well-quasi-ordering.
\end{corollary}

Moreover \autoref{subsetwqo} implies the following variant of \autoref{idl_wqo_down_wqo}:

\begin{corollary}\label{idl_wqo_powset_wqoo}
Let $(X, \leqslant)$ be a well-quasi-ordering.
Then $(\idl{\leqslant}{X}, \subseteq)$ is a well-quasi-ordering if and only if $(X, \leqslant)$ is an $\omega^{2}$-well-quasi-ordering.
\end{corollary}

\subsection{Graph theory}\label{gt_conc}

\paragraph{Sets and integers.} We denote by $\mathbb{N}$ the set of non-negative integers, by $\Nbbb_{\geq n},$ $n > 1,$ to be the set $\Nbbb \setminus \{ m \in \Nbbb \mid m < n \},$ and by $\Nbbb^\mathsf{even}$ the set of even numbers in $\Nbbb.$
Given two integers $p, q,$ where $p \leq q,$ we denote by $[p, q]$ the set $\{p, \dots, q\}.$ For an integer $p \geq 1,$ we set $[p] = [1, p]$ and $\mathbb{N}_{\geq p} = \mathbb{N} \setminus [0, p - 1].$
For a set $S,$ we denote by $2^{S}$ the set of all subsets of $S$ and by $\binom{S}{2}$ the set of all subsets of $S$ of size $2.$
If $\mathcal{S}$ is a collection of objects where the operation $\cup$ is defined, then we denote $\cupall S = \bigcup_{X \in \mathcal{S}} X.$
Also, given a function $f \colon A \to B$ by slightly abusing the notation, we always consider its extension $f \colon 2^A\to B$ such that for every $X \subseteq A,$ $f(X) = \{f(a) \mid a \in X\}.$

\paragraph{Graphs.}

Unless explicitly stated otherwise, throughout this paper we consider only finite, undirected, and simple graphs.
Formally, a \emph{graph} $G$ is a pair $(V, E \subseteq {V \choose 2})$ where $V$ is the finite set of its \emph{vertices} and $E$ is the set of its \emph{edges}.
We use $V(G)$ (resp. $E(G)$) to denote the \emph{vertex set} (resp. the \emph{edge set}) of a graph $G$.
A graph $G$ is not allowed to have \emph{loops}, i.e, for every vertex $u$ of $G$, the edge $\{ u, u \} \notin E(G)$.
The \emph{order} of $G$, denoted by $|G|$ is defined as $|V(G)|$.

If $H$ is a \emph{subgraph} of $G,$ that is $V(H)\subseteq V(G)$ and $E(H)\subseteq E(G)$, we denote this by $H \subseteq G.$
Given a vertex $v \in V(G),$ we denote by $N_{G}(v)$ the set of vertices of $G$ that are adjacent to $v$ in $G.$
Given an edge $e = \{ u, v \} \in E(G)$, the \emph{contraction} of $e$ results in a graph obtained from $G \setminus \{ u, v \}$ by adding a new vertex $w$ adjacent to all vertices incident to $u$ or $v$ in $G$ that is not $u$ or $v$.
A graph $H$ is a \emph{minor} of a graph $G$ if $H$ can be obtained from a subgraph of $G$ after contracting edges.
We denote this relation by $H \leqslant_{\mathsf{m}} G$.

\medskip
Throughout this paper, whenever we use the symbol $\leqslant,$ we shall refer to a well-founded quasi-ordering relation on the set of all graphs, denoted by $\gall$.
For the sake of simplicity we work with the canonical partial-ordering corresponding to $\leqslant$, i.e. for each equivalence class of $\equiv_{\leqslant}$, we arbitrarily choose a representative graph and (re)define $\gall$ to be the set of all chosen representatives. 
Then, $(\gall, \leqslant)$ is a well-founded partial-ordering.

Additionally, as we shall see, our results on universal obstructions apply to any arbitrary well-founded quasi-ordering relation on $\gall.$
However, we shall we use the minor relation throughout the presentation of our results in an attempt to enhance the readability of this work.

\paragraph{Graph classes.}

Respecting the order theoretic terminology we introduced in the previous subsection, we define a $\leqslant$-\emph{closed} class to be any set $\mathcal{G}$ of graphs that belongs to $\down{\leqslant}{\gall}.$

Additionally, for every $\leqslant$-closed class $\mathcal{G}$ we define the $\leqslant$-\emph{obstruction set} of $\mathcal{G},$ denoted by $\obs_{\leqslant}(\mathcal{G})$ to be the set $\m{\leqslant}{\gall \setminus \mathcal{G}}.$
Note that $\gall \setminus \mathcal{G}$ belongs to $\up{\leqslant}{\gall}$ and that $\obs_{\leqslant}(\mathcal{G})$ is the (unique) $\leqslant$-basis of $\gall \setminus \mathcal{G}.$
Moreover, $\obs_{\leqslant}(\mathcal{G})$ is a $\leqslant$-antichain.

For a set of graphs $\mathcal{Z}$, we define the class of $\mathcal{Z}$-$\leqslant$-\emph{free graphs} as $\excl_{\leqslant}(\mathcal{H}) := \gall \setminus \upclosure{\leqslant}{\mathcal{H}}$.
Note that $\excl_{\leqslant}(\mathcal{Z}) = \{ G \in \gall \mid \forall Z \in \mathcal{Z} : Z \nleqslant G\}$ and that $\excl_{\leqslant}(\mathcal{H})$ is $\leqslant$-closed.
Also note that $\mathcal{G} = \excl_{\leqslant}(\obs_{\leqslant}(\mathcal{G}))$.

\medskip
The following proposition is a restatement of \autoref{exclminorssmith} in the terminology introduced in this paragraph.
It gives us a way to test whether a $\leqslant$-closed class class is a subclass of another in terms of the Smyth extension of $\leqslant$ on their $\leqslant$-obstruction sets.

\begin{proposition}\label{subobs}
Let $(\gall, \leqslant)$ be a quasi-ordering.
Then for every two sets of graphs $\mathcal{H}$ and $\mathcal{F}$, $\mathcal{H} \leqslant^{*} \mathcal{F}$ if and only if $\excl_{\leqslant}(\mathcal{H}) \subseteq \excl_{\leqslant}(\mathcal{F}).$
\end{proposition}

As already stated in the introduction, according to the Robertson-Seymour theorem \cite{RobertsonS04GraphMinorsXX}, $\gall$ is well-quasi-ordered by the minor relation.
One of the main consequences of this theorem is that it directly implies that for every $\mathcal{G} \in \mathsf{Down}_{\leqslant_{\mathsf{m}}}(\gall),$ $\obs_{\leqslant_{\mathsf{m}}}(\mathcal{G})$ is a finite set.

\section{The parametric framework}\label{asymp}\label{asympotic_view}

As we discussed in the introduction, our main motivation is the formalization of a theory of finite obstructions for ``second order'' properties of graph classes as well as the introduction of a universal framework for dealing with constructive proofs of obtaining such finite characterizations of individual properties of interest.

\subsection{Class properties}

With this target in mind we commence our story with the introduction of class properties.

\medskip
For the purposes of the definitions that follow, we furthermore assume that $(\gall, \leqslant)$ is a well-quasi-ordering.
As implied by \autoref{wqo_wellfouned}, the former statement is equivalent to $(\down{\leqslant}{\gall}, \subseteq)$ being well-founded.
We first define $\leqslant$-\textsl{class properties}.

\begin{definition}[Class property]
A $\leqslant$-\emph{class property} is any 
subset of $\down{\leqslant}{\gall}$ that is inclusion-closed.
\end{definition}

To illustrate the aforementioned definitions, let us introduce a relatively simple family of $\leqslant$-class properties as a running example.
For every $\leqslant$-closed class $\mathcal{G},$ we define the $\leqslant$-class property $\mathbb{P}(\Gcal)$ as follows.
\begin{align*}
\mathbb{P}(\Gcal) \ \coloneqq \ \{ \mathcal{H} \in \down{\leqslant}{\gall} \mid \mathcal{H} \subseteq \mathcal{G} \}.
\end{align*}

\paragraph{Class obstruction sets.}

Next, we define the $\leqslant$-\textsl{class obstruction set} of a $\leqslant$-class property.

\begin{definition}[Class obstruction set]
For every $\leqslant$-class property $\mathbb{CP},$ the $\leqslant$-\emph{class obstruction set} of $\mathbb{CP}$, denoted by $\cobs_{\leqslant}(\mathbb{CP}),$ is the set $\m{\subseteq}{\down{\leqslant}{\gall} \setminus \mathbb{CP}}.$
\end{definition}

Note that, for every $\leqslant$-class property $\mathbb{CP},$ the set $\down{\leqslant}{\gall} \setminus \mathbb{CP}$ is an upward inclusion-closed subset of $\down{\leqslant}{\gall}$ and that $\cobs_{\leqslant}(\mathbb{CP})$ is the (unique) inclusion-basis of $\down{\leqslant}{\gall} \setminus \mathbb{CP},$ which by the former assumption exists for every $\leqslant$-class property $\mathbb{CP}.$
Additionally, by definition, a $\leq$-closed class $\mathcal{G}$ belongs to $\mathbb{CP}$ if and only if no $\leqslant$-closed class in $\cobs_{\leqslant}(\mathbb{CP})$ is a subclass of $\mathcal{G},$ for every $\leqslant$-class property $\mathbb{CP}$.
Also, by minimality, $\cobs_{\leqslant}(\mathbb{CP})$ is an inclusion-antichain.

\medskip
One may quickly observe that for every $\leqslant$-closed class $\mathcal{G},$ the $\leqslant$-class obstruction set of $\mathbb{P}(\Gcal)$ consists of $|\obs_{\leqslant}(\Gcal)|$ many $\leqslant$-closed classes, each being the $\leqslant$-closure of some graph in $\obs_{\leqslant}(\Gcal).$
Formally:
\begin{align*}
\cobs_{\leqslant}(\mathbb{P}(\Gcal)) \ \coloneqq \ \{ \closure{\leqslant}{\{ Z \}} \mid Z \in \obs_{\leqslant}(\Gcal) \}.
\end{align*}

Moreover, by \autoref{countable_wqo}, we have the following observation on the cardinality of the $\leqslant$-class obstruction sets of $\leqslant$-class properties.

\begin{observation} 
Let $(\gall, \leqslant)$ be a well-quasi-ordering.
Then, the $\leqslant$-class obstruction set of every $\leqslant$-class property is countable.
\end{observation}

Furthermore, by definition of a well-quasi-ordering, we may observe that under the assumption that $(\down{\leqslant}{\gall}, \subseteq)$ is well-founded, $\down{\leqslant}{\gall}$ having no infinite $\subseteq$-antichains is equivalent to the $\leqslant$-class obstruction set of every $\leqslant$-class property being finite.
This observation combined with \autoref{wqo_wellfouned} and \autoref{subsetwqo} implies the following.

\begin{observation}\label{wqo_implies_omega2_equiv_class_obs_finite} Let $(\gall, \leqslant)$ be a well-quasi-ordering. Then, the following statements are equivalent:
\begin{enumerate}
    \item $(\gall, \leqslant)$ is an $\omega^{2}$-well-quasi-ordering;
    \item the $\leqslant$-class obstruction set of every $\leqslant$-class property is finite.
\end{enumerate}
\end{observation}

\paragraph{Directed class properties.}

In this paragraph we are interested in a particular subclass of class properties that will be instrumental towards the characterization of class properties in terms of graph parameters.
We define \textsl{directed} $\leqslant$-class properties as the $\subseteq$-ideals of $\down{\leqslant}{\gall}$.
For completeness:

\begin{definition}[Directed $\leqslant$-class property] A $\leqslant$-class property $\mathbb{CP}$ is \emph{directed} if for every two $\leqslant$-closed classes $\mathcal{G}$ and $\mathcal{H}$ in $\mathbb{CP},$ there exists a $\leqslant$-closed class $\mathcal{F}$ in $\mathbb{CP}$ that is a superclass of both $\mathcal{G}$ and $\mathcal{H}.$
\end{definition}

Over the next few sections it will become apparent that, for every $\leqslant$-closed class $\Gcal,$ the $\leqslant$-class property $\mathbb{P}(\Gcal)$ is directed.

\medskip
Directed $\leqslant$-class properties are particularly interesting as we can show that their $\leqslant$-class obstruction sets precisely correspond to inclusion-antichains of $\leqslant$-ideals of $\gall.$

\begin{theorem}\label{wqo_implies_directed_ideal}
Let $(\gall, \leqslant)$ be a well-quasi-ordering.
Then, a $\leqslant$-class property $\mathbb{CP}$ is directed if and only if $\cobs_{\leqslant}(\mathbb{CP})$ is an inclusion-antichain of $\leqslant$-ideals of $\gall.$
\end{theorem}
\begin{proof} We prove the forward direction by proving the contrapositive.
Assume that there exists $\mathcal{Z} \in \cobs_{\leqslant}(\mathbb{CP})$ that is not a $\leqslant$-ideal of $\gall.$
Then, by \autoref{finite_closed_set_decomposition}, there exists a unique finite set $\{ \mathcal{G}_{1}, \ldots, \mathcal{G}_{z} \},$ $z \in \Nbbb,$ of $\leqslant$-closed classes such that $\mathcal{Z} = \bigcup_{i \in [z]} \mathcal{G}_{i}.$
Moreover, since $\mathcal{Z}$ is not a $\leqslant$-ideal of $\gall,$ every $\mathcal{G}_{i}$ is a proper subclass of $\mathcal{Z}.$
Also, since $\cobs_{\leqslant}(\mathbb{CP})$ is an inclusion-antichain, we have that $\mathcal{Z}' \not\subseteq \mathcal{G}_{i},$ for every $\mathcal{Z}' \in \cobs_{\leqslant}(\mathbb{CP})$ and every $i \in [z].$
Now, since $\mathbb{CP}$ is directed there must exist a $\leqslant$-closed class $\mathcal{H} \in \mathbb{CP}$ such that $\bigcup_{i \in [z]} \mathcal{G}_{i} \subseteq \mathcal{H},$ which leads to a contradiction.

For the reverse direction, let $\mathcal{H}$ and $\mathcal{G}$ be two $\leqslant$-closed classes in $\mathbb{CP}.$
This implies that $\mathcal{Z} \not\subseteq \mathcal{H}$ and $\mathcal{Z} \not\subseteq \mathcal{G},$ for every $\leqslant$-ideal $\mathcal{Z}$ of $\gall$ in $\cobs_{\leqslant}(\mathbb{CP}).$
We claim that $\mathcal{Z} \not\subseteq \mathcal{H} \cup \mathcal{G},$ for every $\mathcal{Z} \in \cobs_{\leqslant}(\mathbb{CP}).$
Indeed, since every such $\mathcal{Z}$ is irreducible, if $\mathcal{Z} \subseteq \mathcal{H} \cup \mathcal{G},$ then either $\mathcal{Z} \subseteq \mathcal{H}$ or $\mathcal{Z} \subseteq \mathcal{G},$ which contradicts our assumptions.
\end{proof}

Finally, combining \autoref{idl_wqo_down_wqo}, \autoref{ideal_decomposition_implies_wqo}, and \autoref{wqo_implies_omega2_equiv_class_obs_finite} we obtain the following.

\begin{theorem}\label{master_template} Let $(\gall, \leqslant)$ be a well-quasi-ordering. Then, the following statements are equivalent:
\begin{enumerate}
\item $(\gall, \leqslant)$ is an $\omega^{2}$-well-quasi-ordering;
\item the $\leqslant$-class obstruction set of every $\leqslant$-class property is finite;
\item the $\leqslant$-class obstruction set of every directed $\leqslant$-class property is finite;
\item every $\leqslant$-class property is equal to the union of a finite inclusion-antichain of directed $\leqslant$-class properties.
\end{enumerate}
\end{theorem}
\begin{proof} The equivalency between \emph{1.} and \emph{2.} follows from \autoref{wqo_implies_omega2_equiv_class_obs_finite}.
The equivalency between \emph{2.} and \emph{3} follows from \autoref{idl_wqo_down_wqo}.
The equivalency between \emph{2.} and \emph{4.} follows from \autoref{ideal_decomposition_implies_wqo}.
\end{proof}

\subsection{Graph parameters}\label{subsec_parameters_parfund}

As we shall see in the coming sections, under certain order theoretic assumptions, graph parameters play a vital role in the quest for a unifying framework for characterizations of class properties.
The goal of this subsection is to introduce graph parameters and set the groundwork for the study of the approximate behaviour.

\medskip
For the purposes of formalizing the theory of universal obstructions for graph parameters, we consider a slight extension of the definition of a graph parameter we gave in the introduction, where we allow for the parameter to be \textsl{infinite} for certain inputs. In this setting:

\begin{definition}[Graph parameter]
A \emph{graph parameter} is a function mapping graphs to non-negative integers and infinity, i.e. $\p  \colon  \gall \to \Nbbb^{\infty}.$
For a graph parameter $\p,$ we define its \emph{domain}, denoted by $\dom(\p)$, to be the set $\{ G \in \gall \mid \p(G) \in \mathbb{N}\}.$
We say that a graph parameter $\p$ is $\leqslant$-\emph{monotone} if for every two graphs $H$ and $G,$ $H \leqslant G$ implies that $\p(H) \leq \p(G).$
For every $k \in \mathbb{N},$ we define the \emph{class of graphs of $\p$ at most $k$}, denoted by $\mathcal{G}_{\p,k}$, as the class $\{ G \in \gall \mid \p(G) \leq k \}.$
\end{definition}

Note that whenever $\p$ is $\leqslant$-monotone, $\dom(\p)$ is a $\leqslant$-closed class and so is $\mathcal{G}_{\p,k},$ for every $k \in \mathbb{N}.$

In order to compare graph parameters and discuss their (approximate) relative behavior, we introduce a quasi-ordering relation $\preceq$ on the set of all graph parameters.
Let $\p$ and $\p'$ be two graph parameters. 
We write $\p \preceq \p'$ if $\dom(\p') \subseteq \dom(\p)$ and there exists a function $f \colon \mathbb{N} \to \mathbb{N}$, called the \emph{gap function} of the relation, such that for every graph $G \in \dom(\p'),$ it holds that $\p(G) \leq f(\p'(G)).$
One can easily verify that $\preceq$ is a quasi-ordering on the set of graph parameters.
Moreover, we use $\sim$ to denote the equivalence relation $\equiv_{\preceq}$ induced by $\preceq$: we say that $\p$ and $\p'$ are \emph{equivalent} with gap function $f$, denoted by $\p \sim \p',$ if there exists a function $f \colon \mathbb{N} \to \mathbb{N}$ such that $\p \preceq \p'$ and $\p' \preceq \p$ both with gap function $f$.
One may notice that even if we allow the gap function $f$ above to differ for the two relations, the definition of equivalence remains the same.

The following lemma gives an equivalent way to approximately compare parameters in terms of their classes of graphs of parameter at most some non-negative integer.

\begin{lemma}\label{parcompobs}
Let $\p$ and $\p'$ be $\leqslant$-monotone parameters. 
Then, $\p \preceq \p'$ if and only if there exists a function $f \colon \Nbbb \to \Nbbb$ such that $\mathcal{G}_{\p', k} \subseteq \mathcal{G}_{\p, f(k)},$ for every $k \in \Nbbb.$
\end{lemma}
\begin{proof}

Assume that $\p \preceq \p'.$ 
Then $\dom(\p') \subseteq \dom(\p)$ and there exists $f  \colon  \mathbb{N} \to \mathbb{N}$ such that for every $G \in \dom(\p'),$ $\p(G) ≤ f(\p'(G)).$
Fix $k \in \mathbb{N}$ and let $G \in \mathcal{G}_{\p', k}.$ 
Equivalently $\p'(G) \leq k$ and thus $\p(G) \leq f(k),$ which implies $G \in \mathcal{G}_{\p, f(k)}.$

For the reverse direction, assume that there exists a function $f \colon \Nbbb \to \Nbbb$ such that $\mathcal{G}_{\p', k} \subseteq \mathcal{G}_{\p, f(k)},$ for every $k \in \Nbbb.$
Clearly $\dom(\p') \subseteq \dom(\p).$ 
Let $G \in \dom(\p')$ such that $\p'(G) = k.$ 
Then $\p(G) \leq f(k).$ 
Equivalently, $\p(G) \leq f(\p'(G)).$
\end{proof}

As a running example of a minor-monotone parameter we choose that of pathwidth, defined by Robertson and Seymour in \cite{RobertsonS83GMI}, which plays an important role in their Graph Minors series.
Out of the many ways to define the pathwidth of a graph, we choose here a definition in terms of vertex orderings.
The \emph{pathwidth} of a graph $G,$ denoted by $\pw(G),$ is the minimum $k$ for which there exists a vertex ordering $\langle v_{1}, \ldots, v_{n} \rangle$ of the vertices of $G$ such that, for each $i \in [n],$ there are at most $k$ vertices in $\{v_{1}, \ldots, v_{i-1} \}$ that are adjacent to vertices in $\{v_{i}, \ldots, v_{n}\}$ (see also \cite{Kinnersley92thev, Moehring90grap, KirousisP85inte, Bienstock89grap, BienstockS91mono, RobertsonS83GMI} for alternative definitions).
We also consider the parameter \emph{biconnected pathwidth}, denoted by $\bipw(G)$, of a graph $G$ that is the maximum pathwidth of its blocks\footnote{A subgraph $B$ of a graph $G$ is a \emph{block} of $G$ if it is either a 2-connected component or an isolated vertex or, a bridge of $G$ (a \emph{bridge} of a graph $G$ is a subgraph  of $G$ on two vertices and one edge whose removal increases the number of connected components of $G$).}.
Biconnected pathwidth was defined in \cite{Thomas96} and was studied in \cite{dang2018minors, HuynhJMW20Seymour}.
One may easily notice that, the biconnected pathwidth of a graph is at most the pathwidth of a graph, i.e., $\bipw \preceq \pw.$

\subsection{Parametric graphs}\label{subsec_parametric_graphs}

Having introduces graph parameters in the previous sections, we move on to the second key notion of our parametric viewpoint.
In this subsection we delve into the realm of parametric graphs.

\medskip
A \emph{graph sequence} $\mathscr{H}$ is a countable sequence of graphs indexed by non-negative integers, denoted by $\mathscr{H} := \langle \mathscr{H}_{i} \rangle_{i \in \mathbb{N}}$. 
In some cases, instead of a sequence, we treat $\mathscr{H}$ as a set of graphs. 
When this is the case we say that $\{ \mathscr{H}_{i} \mid i \in \mathbb{N} \}$ is \emph{the set} of $\mathscr{H}$.
Note that in a sequence $\mathscr{H}$ the same graph is allowed to repeat for distinct indices, but the set of $\mathscr{H}$ is not a \textsl{multiset}, i.e. only a single copy per distinct graph is contained in the set of $\mathscr{H}$.
Moreover, while $\mathscr{H}$ is infinite, the set of $\mathscr{H}$ may be finite.
We define the $\leqslant$-\emph{width} of a graph sequence, denoted by $\w_{\leqslant}(\mathscr{H}),$ as the width of the partial-ordering $(\mathscr{H}, \leqslant)$.
For $k \in \mathbb{N},$ we write $\mathscr{H}_{\geq k}$ for the sequence $\langle \mathscr{H}_{i} \rangle_{i \geq k}$, i.e., for the sequence obtained from $\mathscr{H}$ after removing all graphs with index smaller that $k.$
We call a graph sequence $\leqslant$-\emph{rational} if $\w_{\leqslant}(\mathscr{H}) \in \mathbb{N},$ $\leqslant$-\emph{prime} if $\w_{\leqslant}(\mathscr{H}) = 1$, and $\leqslant$-\emph{increasing} if for $i \leq j,$ $\mathscr{H}_{i} \leqslant \mathscr{H}_{j}.$
With this we are now ready to define parametric graphs.

\begin{definition}[Parametric graph]\label{def:par_graph}
Any $\leqslant$-increasing graph sequence is a $\leqslant$-\emph{parametric graph}.
\end{definition}

The monotonicity of the definition of $\leqslant$-parametric graphs allows for a natural way to compare them.
We define a quasi-ordering relation $\lesssim$ on the set of all $\leqslant$-parametric graphs.
For two $\leqslant$-parametric graphs $\mathscr{H}$ and $\mathscr{F},$ we write $\mathscr{H} \lesssim \mathscr{F}$ if there exists a function $f \colon \mathbb{N} \to \mathbb{N}$ such that $\mathscr{H}_{k} \leqslant \mathscr{F}_{f(k)}$.
If we wish to specify the function $f$ in the definition above we shall denote this by $\mathscr{H} \lesssim_{f} \mathscr{F}.$
It is not hard to see that $\lesssim$ is a quasi-ordering on the set of $\leqslant$-parametric graphs.

Moreover, we say that $\mathscr{H}$ and $\mathscr{F}$ are \emph{equivalent}, denoted by $\mathscr{H} \approx \mathscr{F}$, if $\mathscr{H} \lesssim \mathscr{F}$ and $\mathscr{F} \lesssim \mathscr{H}$, i.e., if they are equivalent under the equivalence relation induced by $\lesssim$.
Once more, if we wish to specify the function $f$ that certifies both directions of the above equivalence, we shall denote this by $\mathscr{H} \approx_{f} \mathscr{F}.$
Note that forcing a single function that bounds both directions in the definition of $\approx_{f}$ gives an equivalent definition to $\approx$.

We continue our discussion with three examples of minor-parametric graphs:
\begin{itemize}
\item The sequence $\mathscr{T}=\langle \mathscr{T}_1,\mathscr{T}_2,\mathscr{T}_3,\mathscr{T}_4,\dots\rangle$ of complete ternary trees (see \autoref{tern_trr}), 
\item the sequence $\mathscr{T}^a=\langle \mathscr{T}^a_2,\mathscr{T}^a_3,\mathscr{T}^a_4,\mathscr{T}^a_5,\dots\rangle,$ where each $\mathscr{T}^a_k$ is obtained from $\mathscr{T}_k$ after adding a new vertex and connecting it to its leaves (see \autoref{teeapex}), and
\item  the sequence $\mathscr{T}^{a*}=\langle \mathscr{T}^{a*}_2,\mathscr{T}^{a*}_3,\mathscr{T}^{a*}_4,\mathscr{T}^{a*}_5,\dots\rangle$, where each $\mathscr{T}^{a*}_k$ is obtained from the dual of $\mathscr{T}_k^{a}$ after subdividing, for each double edge, one of its copies (see \autoref{fig_dual_Tk}).
\end{itemize}

\begin{figure}[htbp]
\centering
\includegraphics[width=0.9\linewidth]{ternary_trees}
\caption{\label{tern_trr}The minor-parametric $\mathscr{T}=\langle \mathscr{T}_1,\mathscr{T}_2,\mathscr{T}_3,\mathscr{T}_4,\ldots\rangle$ of complete ternary trees.}
\end{figure}

\begin{figure}[htbp]
\centering
\includegraphics[width=0.9\linewidth]{apex_complete_ternary_tree}
\caption{\label{fig_bpw}\label{teeapex}The minor-parametric graph $\mathscr{T}^a = \langle \mathscr{T}^a_2,\mathscr{T}^a_3,\mathscr{T}^a_4,\mathscr{T}^a_5,\ldots\rangle.$}
\end{figure}

\begin{figure}[htbp]
\centering
\includegraphics[width=0.9\linewidth]{outerplanar_parametric_obstruction}
\caption{\label{fig_dual_Tk}\label{teeapex2}The minor-parametric graph $\mathscr{T}^{a*}=\langle \mathscr{T}^{a*}_2,\mathscr{T}^{a*}_3,\mathscr{T}^{a*}_4,\mathscr{T}^{a*}_5,\ldots\rangle.$}
\end{figure}

\paragraph{Sequences and their parameters.}

To every graph sequence $\mathscr{H}$ we associate a graph parameter $\p_{\mathscr{H}}$ as follows.

\begin{definition}[Parameter associated to a sequence]\label{thif0lkf}
For every graph sequence $\mathscr{H},$ we define the graph parameter $\p_{\mathscr{H}}$, so that, for a graph $G$:
$$\p_{\mathscr{H}}(G) \coloneqq \inf\{ k \in \mathbb{N} \mid \m{\leqslant}{\mathscr{H}_{\geq k}} \nleqslant^{*} \{G\} \}.$$
Note that $\p_{\mathscr{H}}$ is $\leqslant$-monotone and that $\mathcal{G}_{\p_{\mathscr{H}}, k} = \excl_{\leqslant}(\m{\leqslant}{\mathscr{H}_{\geq k}}),$ for every $k \in \Nbbb.$
\end{definition}

Motivated by the definition above, we may classify the set of all parameters according to the existence of equivalent representative parameters that are associated to sequences.
Let $\p$ be a graph parameter. 
We say that $\p$ is $\leqslant$-\emph{real} if there exists a graph sequence $\mathscr{H}$ such that $\p \sim \p_{\mathscr{H}}$.
If additionally $\mathscr{H}$ is $\leqslant$-rational (resp. $\leqslant$-prime) then we say that $\p$ is \emph{$\leqslant$-rational} (resp. \emph{$\leqslant$-prime}).

Let $\mathscr{H}$ be a $\leqslant$-parametric graph.
We wish to stress that the monotonicity of $\mathscr{H}$ leads to a neater definition for the parameter associated to it.
Indeed, one may notice that for every $k \in \Nbbb,$ $\m{\leqslant}{\mathscr{H}_{\geq k}} = \{ \mathscr{H}_{k} \}.$
Therefore we make the following observation:

\begin{observation}\label{parameter_parametric_graph_simple}
Let $\mathscr{H}$ be a $\leqslant$-parametric graph.
Then, for every graph $G$:
$$\p_{\mathscr{H}}(G) = \inf\{ k \in \Nbbb \mid \mathscr{H}_{k} \not\leqslant G \}.$$
Moreover, $\mathcal{G}_{\p_{\mathscr{H}}, k} = \excl_{\leqslant}(\{\mathscr{H}_{k}\}),$ for every $k \in \Nbbb.$
\end{observation}

We should remark that, given a $\leqslant$-parametric graph $\mathscr{H},$ the set $\{ k \in \Nbbb \mid \mathscr{H}_{k} \not\leqslant G \}$ may be empty for some graph $G$ and therefore for this graph $\p_{\mathscr{H}}(G) = \infty.$
We denote by $\p_{\infty}$ the parameter where for every graph $G \in \gall,$ $\p_{\infty}(G) = \infty.$
It is straightforward to observe that $\p_{\infty}$ corresponds to a maximal element of the quasi-ordering $\preceq$ on the set of all parameters as trivially, for every graph parameter $\p,$ $\p \preceq \p_{\infty}.$
Therefore, $\p_{\infty}$ corresponds to one ``extreme'' of the quasi-ordering $\preceq$ on the set of all graph parameters.

In the context of the minor relation, as every graph contains the empty graph $K_{0}$ as a minor (in fact even as a subgraph), $\p_{\infty} = \p_{\mathscr{H}},$ where $\mathscr{H}$ is the minor-parametric graph whose every instance is the empty graph $K_{0}.$

\paragraph{Rational sequences and finite antichains of parametric graphs.}

In this paragraph we examine the relationship between rational graph sequences and parametric graphs with respect to the approximate behaviour of their corresponding graph parameters.

We first extend \autoref{thif0lkf} to any finite set of $\leqslant$-parametric graphs.

\begin{definition}(parameter associated to a finite set of parametric graphs)\label{ofjfhfdpia}
For every finite set of $\leqslant$-parametric graphs $\mathfrak{H},$ we define the parameter $\p_{\mathfrak{H}}$ so that, for a graph $G$,
$$\p_{\mathfrak{H}}(G) \coloneqq \begin{cases}
                                    \sup\{ \p_{\mathscr{H}}(G) \mid \mathscr{H} \in \mathfrak{H} \}, & \text{if $\mathfrak{H} \neq \emptyset$}\\
                                    0, & \text{otherwise}
                                \end{cases}$$
\end{definition}

Note that in the definition above, if $\mathfrak{H}$ is empty, $\p_{\mathfrak{H}}$ corresponds to the constant zero function.
We denote this graph parameter by $\p_{0}$.
Observe that $\p_{0}$ defines the other ``extreme'' of the quasi-ordering $\preceq$ on the set of all graph parameters, in the sense that for every parameter $\p,$ $\p_{0} \preceq \p.$

The following lemma gives an equivalent definition for \autoref{ofjfhfdpia}.

\begin{lemma}\label{prime_col_new}
Let $\mathfrak{H}$ be a non-empty finite set of $\leqslant$-parametric graphs.
Then for every graph $G,$
$$\p_{\mathfrak{H}}(G) = \inf\{ k \mid \forall \mathscr{H} \in \mathfrak{H} : \mathscr{H}_{k} \nleqslant G \}.$$
\end{lemma}
\begin{proof} Let $G$ be a graph. We have that,
\begin{align*}
	\p_{\mathfrak{H}}(G) \ &= \ \sup\{ \p_{\mathscr{H}}(G) \mid \mathscr{H} \in \mathfrak{H} \}\\
	\text{(\autoref{parameter_parametric_graph_simple})} \ &= \ \sup_{\mathscr{H} \in \mathfrak{H}} \inf\{ k \in \mathbb{N} \mid \mathscr{H}_{k} \nleqslant G \}\\
	\text{($i \leq j \Rightarrow \mathscr{H}_{i} \leqslant \mathscr{H}_{j}$)} \ &= \ \inf\{ k \in \mathbb{N} \mid \forall \mathscr{H} \in \mathfrak{H} : \mathscr{H}_{k} \nleqslant G \}.
\end{align*}
\end{proof}

Note that $\p_{\mathfrak{H}}$ is $\leqslant$-monotone and that from \autoref{prime_col_new} we can observe that $\mathcal{G}_{\p_{\mathfrak{H}}, k} = \excl_{\leqslant}(\{ \mathscr{H}_{k} \mid \mathscr{H} \in \mathfrak{H} \}),$ for every $k \in \Nbbb.$

We continue with the main result of this subsection which shows that any rational sequence corresponds to a finite $\lesssim$-antichain of $\leqslant$-parametric graphs with respect to equivalence of their respective parameters.

\begin{theorem}\label{rationalperiodic}
For every rational sequence $\mathscr{H},$ there exists a finite $\lesssim$-antichain of $\leqslant$-parametric graphs $\mathfrak{H}$ such that $\p_{\mathscr{H}} \sim \p_{\mathfrak{H}}$. 
\end{theorem}
\begin{proof} Let $\mathscr{Η}$ be a rational sequence, $w := \w(\mathscr{H})$, and let $\mathfrak{B} = \{ \mathscr{B}^{(i)} \mid i \in [w] \}$ be a set of $\leqslant$-chains that partition the set of $\mathscr{Η}$, as implied by \autoref{dilworth}.
Observe that $\mathfrak{B}$ may contain (only) finite $\leqslant$-chains.
From $\mathfrak{B}$ we define a set $\mathfrak{C} = \{ \mathscr{C}^{(i)} \mid i \in [w] \}$ of $\leqslant$-chains as follows. 
For every $i \in [w]$, we distinguish the following cases:
\begin{itemize}
\item There exists $k \in \mathbb{N}$ such that $\mathscr{B}^{(i)}_{k}$ is a member of $\mathscr{H}$ an infinite number of times. In this case, let $k \in \mathbb{N}$ be the minimum with this property and define $\mathscr{C}^{(i)} := \langle \mathscr{C}^{(i)}_{j} \rangle_{j \in \mathbb{N}}$ as follows. For $j < k$, $ \mathscr{C}^{(i)}_{j} := \mathscr{B}^{(i)}_{j}$ and for $j \geq k$, $\mathscr{C}^{(i)}_{j} := \mathscr{B}^{(i)}_{k}$.
\item No such $k$ exists. In this case, let $\mathscr{C}^{(i)} := \mathscr{B}^{(i)}$.
\end{itemize}

Clearly, for every $i \in [w]$, $\mathscr{C}^{(i)} \subseteq \mathscr{B}^{(i)}$.
Partition $\mathfrak{C} = \mathfrak{F} \cup \mathfrak{I}$ into the set of finite (the set $\mathfrak{F}$) and the set of infinite (the set $\mathfrak{I}$) $\leqslant$-chains in $\mathfrak{C}$.
Observe that, since $\mathscr{H}$ is a countable sequence, by definition of $\mathfrak{C}$, $\mathfrak{I}$ is non-empty.
Moreover in the case that $\mathscr{H}$ is prime, $\mathfrak{F} = \emptyset$ and $\mathfrak{I}$ is a singleton.

Let $\mathfrak{G} = \{ \mathscr{G}^{(i)} \mid i \in [r] \}$ contain only a single representative from each equivalence class of $\approx$, among the $\lesssim$-minimal $\leqslant$-chains of $\mathfrak{I}$.
Then $\mathfrak{G}$ is a $\lesssim$-antichain of $\leqslant$-parametric graphs.
For convenience, we write $\mathfrak{G}_{\geq k} := \bigcup_{i \in [r]} \mathscr{G}^{(i)}_{\geq k}.$ Note that by definition of $\leqslant$-parametric graphs, $\m{\leqslant}{\mathfrak{G}_{\geq k}} = \{ \mathscr{G}^{(i)}_{k} \mid i \in [r] \}$.
Moreover, by definition of $\p_{\mathscr{H}}$ (resp. $\p_{\mathfrak{G}}$), it is easy to see that for every $k \in \mathbb{N}$, $\mathcal{G}_{\p_{\mathscr{H}}, k} = \excl(\m{\leqslant}{\mathscr{H}_{\geq k})}$ (resp. $\mathcal{G}_{\p_{\mathfrak{G}}, k} = \excl(\m{\leqslant}{\mathfrak{G}_{\geq k}}$)).

\begin{claim}\label{claim_rational_frak} $\p_{\mathscr{H}} \sim \p_{\mathfrak{G}}$.
\end{claim}
\begin{cproof}
Fix $a \in \mathbb{N}$.
Let $b \in \mathbb{N}$ be minimum such that $\{ \mathscr{H}_{l} \mid l \in [a] \} \subseteq \bigcup_{i \in [r]} \{ \mathscr{G}^{(i)}_{l} \mid l < b \}$.
Observe that by the choice of $b$, $\mathscr{H}_{\geq a}$ contains at least one graph in $\{ \mathscr{G}^{(i)}_{l} \mid l \in [b] \}$, for every $i \in [r]$.
Since the sequences in $\mathfrak{G}$ are $\leqslant$-increasing, it is implied that $\m{\leqslant}{\mathscr{H}_{\geq a}} \leqslant^{*} \m{\leqslant}{\mathfrak{G}_{\geq b}}$.
By \autoref{subobs}, $\mathcal{G}_{\p_{\mathscr{H}}, a} \subseteq \mathcal{G}_{\p_{\mathfrak{G}}, b}$. 
Hence, by \autoref{parcompobs}, $\p_{\mathfrak{G}} \preceq \p_{\mathscr{H}}$.

For the other direction, partition $\mathfrak{G}$ into the sets of sequences $\mathfrak{G}_{1} = \{ \mathscr{G}^{(i)} \mid i \in [r_{1}] \}$ and $\mathfrak{G}_{2} = \{ \mathscr{G}^{(i)} \mid i \in (r_{1}, r]\}$ of $\mathfrak{G}$ whose set has a $\leqslant$-maximum element (the set $\mathfrak{G}_{1}$) and those that do not (the set $\mathfrak{G}_{2}$).
In the same way, partition $\mathfrak{I} \setminus \mathfrak{G}$ into the sets of sequences $\mathfrak{I}_{1} = \{ \mathscr{I}^{(i)} \mid i \in [s_{1}] \}$ and $\mathfrak{I}_{2} = \{ \mathscr{I}^{(i)} \mid i \in (s_{1}, s]\}$ of $\mathfrak{I} \setminus \mathfrak{G}$ that have a $\leqslant$-maximum element (the set $\mathfrak{I}_{1}$) and those that do not (the set $\mathfrak{I}_{2}$).

Let $g \in \mathbb{N}$ be minimum such that the set $\{ \mathscr{G}^{(i)}_{g} \mid i \in [r_{1}] \}$ corresponds to the $\leqslant$-maximum elements of the sequences in $\mathfrak{G}_{1}$.
In the same way, let $q \in \mathbb{N}$ be minimum such that the set $\{ \mathscr{I}^{(i)}_{q} \mid i \in [s_{1}] \}$ corresponds to the $\leqslant$-maximum elements of the sequences in $\mathfrak{I}_{1}$.
Observe that, for every $k \geq g$, $\mathfrak{G}_{\geq k} = \{ \mathscr{G}^{(i)}_{g} \mid i \in [r_{1}] \} \cup \{ \mathscr{G}^{(i)}_{k} \mid i \in (r_{1}, r] \}$.
Additionally observe that, there exists a minimum $h \in \mathbb{N}$ such that
$$\mathscr{H}_{\geq h} \cap \Big( \big( \bigcup_{\mathscr{F} \in \mathfrak{F}} \mathscr F \big) \cup \big( \bigcup_{i \in [r_{1}]} \{ \mathscr{G}^{(i)}_{j} \mid j < g \} \big) \cup \big( \bigcup_{i \in [s_{1}]} \{ \mathscr{I}^{(i)}_{j} \mid j < q \} \big) \Big) = \emptyset.$$

By definition of $h$, it holds that for every $i \in [r_{1}]$ and every $k \geq h$, $\mathscr{G}^{(i)} \cap \mathscr{H}_{\geq k} = \{ \mathscr{G}^{(i)}_{g} \}$,
and that for every $i \in [s_{1}]$ and every $k \geq h$, $\mathscr{I}^{(i)} \cap \mathscr{H}_{\geq k} = \{ \mathscr{I}^{(i)}_{q} \}$.
Moreover, observe that any $\leqslant$-parametric graph that has a $\leqslant$-maximum element, is equivalent (w.r.t. $\approx$) to any other $\leqslant$-parametric graph, if and only if it has a $\leqslant$-maximum element as well and their $\leqslant$-maximum elements are the same graph.
This implies that the equivalence classes of $\approx$ containing a sequence from $\mathfrak{I}_{1}$ are singletons.
This, combined with the fact that $\mathfrak{B}$ partitions the set of $\mathscr{H}$, implies that none of the sequences in $\mathfrak{I}_{1}$ are $\lesssim$-minimal in $\mathfrak{C}$.
Then, for every $k \geq h$, $\m{\leqslant}{\mathscr{H}_{\geq k}} \cap \{ \mathscr{I}^{(i)}_{q} \mid i \in [s_{1}] \} = \emptyset$.

Fix $b \geq g$.
For $m \in (s_{1}, s],$ let $I^{(m)} \subseteq \mathbb{N}$ such that $n \in I^{(m)}$ if $n$ is minimum with the property that there exists $l \in [r]$ such that $\mathscr{G}^{(l)}_{b} \leqslant \mathscr{I}^{(m)}_{n}$.
Notice that, since every sequence in $\mathfrak{G}$ is $\lesssim$-minimal, for every $m \in (s_{1}, s]$, there exists $l \in [r]$ such that $\mathscr{G}^{(l)} \lesssim \mathscr{I}^{(m)}$. 
This implies that $I^{(m)} \neq \emptyset$.
Define $c^{(m)} := \max I^{(m)}$ and $c := \max \{ c^{(m)} \mid m \in (s_{1}, s] \}$.
Then, let $a \geq h$ be minimum such that 
$$\mathscr{H}_{\geq a} \cap \Big( \big( \bigcup_{m \in (r_{1}, r]} \{ \mathscr{G}^{(m)}_{n} \mid n \in [b] \} \big) \cup \big(\bigcup_{m \in (s_{1}, s]} \{ \mathscr{I}^{(m)}_{n} \mid n \in [c] \}\big) \Big) = \emptyset.$$ 

Our goal it to show that $\m{\leqslant}{\mathfrak{G}_{\geq b}} \leqslant^{*} \m{\leqslant}{\mathscr{H}_{\geq a}}$, i.e. for every graph $Z \in \m{\leqslant}{\mathscr{H}_{\geq a}}$ there exists $m \in [r]$ such that $\mathscr{G}^{(m)}_{b} \leqslant Z$.
Then, by \autoref{subobs}, this would imply that $\mathcal{G}_{\p_{\mathfrak{G}}, b} \subseteq \mathcal{G}_{\p_{\mathscr{H}}, a}$ and hence $\p_{\mathscr{H}} \preceq \p_{\mathfrak{G}}$ which concludes the proof.
Recall that $\mathfrak{G}_{\geq b} = \{ \mathscr{G}^{(m)}_{g} \mid m \in [r_{1}] \} \cup \{ \mathscr{G}^{(m)}_{b} \mid m \in [r_{1}, r] \}$.
Consider a graph $Z \in \m{\leqslant}{\mathscr{H}_{\geq a}}$. 
In the case that $Z = \mathscr{G}^{(m)}_{g}$, for some $m \in [r_{1}]$, we immediately conclude.
In the case that $Z = \mathscr{G}^{m}_{n}$, for some $m \in (r_{1}, r]$ and $n > b$, we conclude since $\mathscr{G}^{m}$ is $\leqslant$-increasing.
In the case that $Z = \mathscr{I}^{m}_{n}$, for some $m \in (s_{1}, s]$ and $n > c$, we conclude by definition of $c$.
Notice that, by definition of $a$ and the previous observations, these are the only cases to consider.
\end{cproof}

By the previous claim, $\mathfrak{G}$ is the desired set of sequences $\mathfrak{H}$. \end{proof}

Notice that \autoref{rationalperiodic}, applied to a $\leqslant$-prime sequence $\mathscr{H}$, guarantees the existence of a set $\mathfrak{H} = \{ \mathscr{F} \},$ where $\p_{\mathscr{H}} \sim \p_{\mathscr{F}}$ and $\mathscr{F}$ is a $\leqslant$-parametric graph.
Hence, through the lens of equivalence of their respectful parameters, $\leqslant$-prime sequences have representatives that are $\leqslant$-parametric graphs.

\subsection{From graph parameters to class properties and their class obstructions}\label{from_params_to_props}

In this subsection we associate graph parameters to a particular subclass of class properties in a ``natural'' way that respects their approximate behaviour.
Via this correspondance we establish the notion of a class obstruction for a graph parameter and discuss how this set captures the approximate behaviour of the parameter.

\paragraph{Classes where a parameter is bounded.}

Let $\p$ be a graph parameter. 
We say that $\p$ is \emph{bounded} in a graph class $\mathcal{G}$ if there exists $c \in \mathbb{N}$ such that for every $G \in \mathcal{G},$ $\p(G) \leq c.$
Via the previous definition, we may associate every $\leqslant$-monotone parameter to $\leqslant$-class property as follows.

\begin{definition}[Class property associated to monotone parameter]
For every $\leqslant$-monotone parameter, we define the set $\mathbb{B}_{\leqslant}(\p)$ of $\leq$-closed classes where $\p$ is bounded as follows:
\begin{align*}
\B_{\leqslant}(\p) \ \coloneqq \ \{ \mathcal{G} \in \down{\leqslant}{\gall} \mid \textrm{$\p$ is bounded in $\mathcal{G}$}\}.
\end{align*}
Note that, if $\p$ is bounded in a $\leqslant$-closed class $\mathcal{G},$ then $\p$ is also bounded in any subclass of $\mathcal{G}.$
Therefore $\mathbb{B}_{\leqslant}(\p)$ is a $\leqslant$-class property.
\end{definition}

Dually, for every $\leqslant$-monotone parameter, we may also define the set
\begin{align*}
\UNB_{\leqslant}(\p) \ &\coloneqq \ \{ \mathcal{G} \in \down{\leqslant}{\gall} \mid \text{$\p$ is \emph{unbounded} in $\mathcal{G}$} \}\\
&= \down{\leqslant}{\gall} \setminus \B_{\leqslant}(\p).
\end{align*}
Note that $\UNB_{\leqslant}(\p)$ is upward $\subseteq$-closed in $(\down{\leqslant}{\gall}, \subseteq).$

\paragraph{ The set $\B_{\leqslant}(\p)$ and the approximate behaviour of $\p$.}

Having introduced the $\leqslant$-class property $\mathbb{B}_{\leqslant}(\p),$ for every $\leqslant$-monotone parameter $\p,$ we may observe the following.

\begin{observation}\label{gpk_is_bounded}
Let $\p$ be a $\leqslant$-monotone parameter.
Then, $\mathcal{G}_{\p, k} \in \mathbb{B}_{\leqslant}(\p),$ for every $k \in \Nbbb.$
\end{observation}

The previous observation immediately leads to the following second observation:

\begin{observation}\label{obs_bounded_class}
Let $\p$ be a $\leqslant$-monotone parameter.
Then, $\mathcal{G} \in \B_{\leqslant}(\p)$ if and only if there exists $k \in \mathbb{N}$ such that $\mathcal{G} \subseteq \mathcal{G}_{\p, k}$.
\end{observation}

The following lemma show how comparing parameters approximately via the relation $\preceq$ reduces to checking for inclusion of their corresponding $\leqslant$-class properties.

\begin{lemma}\label{parsimbounded} Let $\p$ and $\p'$ be two $\leqslant$-monotone parameters.
Then, $\p \preceq \p'$ if and only if $\B_{\leqslant}(\p') \subseteq \B_{\leqslant}(\p).$
\end{lemma}
\begin{proof} For the proof of the forward direction, assume that $\p \preceq \p'.$
By \autoref{parcompobs}, there exists a function $f \colon \Nbbb \to \Nbbb$ such that $\mathcal{G}_{\p', k} \subseteq \mathcal{G}_{\p, f(k)},$ for every $k \in \Nbbb.$
Then, by an application of \autoref{obs_bounded_class} for $\p$ and $f(k)$ we have that $\mathcal{G}_{\p', k} \in \mathbb{B}_{\leqslant}(\p),$ for every $k \in \Nbbb.$
By a second application of \autoref{obs_bounded_class} for $\p',$ we conclude that $\mathbb{B}_{\leqslant}(\p') \subseteq \mathbb{B}_{\leqslant}(\p).$

For the reverse direction, assume that $\mathbb{B}_{\leqslant}(\p') \subseteq \mathbb{B}_{\leqslant}(\p).$
By an application of \autoref{gpk_is_bounded} to $\p',$ we have that $\mathcal{G}_{\p', k} \in \mathbb{B}_{\leqslant}(\p),$ for every $k \in \Nbbb.$
An application of \autoref{obs_bounded_class} to $\p$ and $\mathcal{G}_{\p', k} \in \mathbb{B}_{\leqslant}(\p),$ for every $k \in \Nbbb,$ implies the existence of a function $f \colon \Nbbb \to \Nbbb$ such that $\mathcal{G}_{\p', k} \subseteq \mathcal{G}_{\p, f(k)},$ for every $k \in \Nbbb.$
Then, by \autoref{parcompobs}, we conclude that $\p \preceq \p'.$
\end{proof}

A consequence of the lemma above and the fact that $(\down{\leqslant}{\gall}, \subseteq)$ is a partial-ordering is the following corollary.

\begin{corollary}\label{parequivsamebounded}
Let $\p$ and $\p'$ be two $\leqslant$-monotone parameters. Then, $\p \sim \p'$ if and only if $\B(\p) = \B(\p')$.
\end{corollary}

We should remark that the previous corollary implies that, $\leqslant$-monotone parameters in the same equivalence class of $\sim$ are precisely characterized by the same $\leqslant$-class property, in the sense that the $\leqslant$-graph classes where any of the parameters in the equivalence class is bounded are precisely the same.

The previous discussion is a further indication that the $\leqslant$-class properties we associate to $\leqslant$-monotone parameters precisely capture their approximate behaviour.

\paragraph{Class obstruction of a parameter.}

Motivated by the observations above, we define our first obstruction notion for a $\leqslant$-parameter to be the $\leqslant$-class obstruction set of its corresponding $\leqslant$-class property $\mathbb{B}_{\leqslant}(\p).$
Formally:

\begin{definition}[Class obstruction]\label{def:cobs}
For every $\leqslant$-monotone parameter $\p,$ we define the $\leqslant$-\emph{class obstruction} of $\p$ as the $\leqslant$-class obstruction set of $\mathbb{B}_{\leqslant}(\p),$ denoted by $\cobs_{\leqslant}(\p),$ when it exists.

Note that, if $(\gall, \leqslant)$ is a well-quasi-ordering, then the set $\cobs_{\leqslant}(\mathsf{p})$ exists, for every $\leqslant$-monotone parameter $\mathsf{p}.$
\end{definition}

We now know that every (equivalence class of a) $\leqslant$-monotone parameter $\p$ precisely corresponds to the $\leqslant$-class property $\mathbb{B}_{\leqslant}(\p).$
A natural question to ask is whether every $\leqslant$-class property corresponds to some (equivalence class of a) $\leqslant$-monotone parameter.
The answer to this question is negative.
We can observe that, for every $\leqslant$-monotone parameter $\p$, $\mathbb{B}_{\leqslant}(\p)$ belongs to a specific subclass of $\leqslant$-class properties.

\begin{observation}\label{bp_is_directed}
For every $\leqslant$-monotone parameter $\mathsf{p}$, $\mathbb{B}_{\leqslant}(\p)$ is a directed $\leqslant$-class property.
\end{observation}
\begin{proof}
It suffices to observe that $\mathbb{B}_{\leqslant}(\p)$ is closed under union.
Indeed, let $\mathcal{G}$ (resp. $\mathcal{H}$) be a $\leqslant$-closed class where $\p$ is bounded by $c \in \mathbb{N}$ (resp. $c' \in \Nbbb$), for every graph in $\mathcal{G}$ (resp. $\mathcal{H}$).
Then, $\p$ is also bounded in $\mathcal{G} \cup \mathcal{H}$ by $\max\{ c, c' \}.$
\end{proof}

Combining the previous observation with \autoref{wqo_implies_directed_ideal} shows, assuming well-quasi-ordering, the $\leqslant$-class obstruction of every $\leqslant$-monotone parameter is an inclusion-antichain of $\leqslant$-ideals of $\gall.$

\begin{corollary}\label{params_obs_ideals} Let $(\gall, \leqslant)$ be a well-quasi-ordering.
Then, for every $\leqslant$-monotone parameter $\mathsf{p},$ the $\leqslant$-class obstruction of $\mathsf{p}$ is an inclusion-antichain of $\leqslant$-ideals of $\gall.$
\end{corollary}

Later we will see that, under certain order theoretic assumptions, every directed $\leqslant$-class property also corresponds to some suitably defined $\leqslant$-monotone parameter.
Moreover, under the same condition, we will also see how every $\leqslant$-class property corresponds to some finite set of $\leqslant$-monotone parameters.

\subsection{Parametric families}

In this subsection we introduce parametric families and slowly work towards our definition of ``canonical representatives'' for graph parameters.

\medskip
We begin with the introduction of $\leqslant$-parametric families, a notion instrumental for our definition of $\leqslant$-universal obstructions in the coming subsections.

\begin{definition}\label{def:param_collection}
Any countable $\lesssim$-antichain of $\leqslant$-parametric graphs is a $\leqslant$-\emph{parametric family}.

Moreover, to every $\leqslant$-parametric family $\mathfrak{H}$ we associate the $\leqslant$-monotone parameter $\p_{\mathfrak{H}}$ as in \autoref{ofjfhfdpia}.
\end{definition}

As we have done for all previously introduced notions, we shall introduce a quasi-ordering relation $\lesssim^{*}$ on the set of all $\leqslant$-parametric families.
Let $\mathfrak{H}$ and $\mathfrak{F}$ be two $\leqslant$-parametric families.
We write $\mathfrak{H} \lesssim^{*} \mathfrak{F}$ if for every $\leqslant$-parametric graph $\mathscr{F} \in \mathfrak{F}$ there exists a $\leqslant$-parametric graph $\mathscr{H} \in \mathfrak{H}$ such that $\mathscr{H} \lesssim \mathfrak{H},$ i.e., $\lesssim^{*}$ corresponds to the Smyth extension of $\lesssim$ on the set of all $\leqslant$-parametric families.
We also introduce a weaker form of $\lesssim^{*}$.
We write $\mathfrak{H} \lesssim^{*}_{f} \mathfrak{F}$ if there exists a function $f \colon \mathbb{N} \to \mathbb{N}$ such that for every $\leqslant$-parametric graph $\mathscr{F} \in \mathfrak{F}$ there exists a $\leqslant$-parametric graph $\mathscr{H} \in \mathfrak{H}$ such that $\mathscr{H} \lesssim_{f} \mathfrak{H}.$

We also define two equivalence relations on the set of all $\leqslant$-parametric families.
The first one corresponds to the equivalence relation induced by $\lesssim^{*}$ and the second to the equivalence relation induced by $\lesssim^{*}_{f}.$
We write $\mathfrak{H} \equiv^{*} \mathfrak{F}$ if $\mathfrak{H} \lesssim^{*} \mathfrak{F}$ and $\mathfrak{F} \lesssim^{*} \mathfrak{H}$.
We write $\mathfrak{H} \equiv^{*}_{f} \mathfrak{F}$ if there exists a function $f \colon \mathbb{N} \to \mathbb{N}$ such that $\mathfrak{H} \lesssim^{*}_{f} \mathfrak{F}$ and $\mathfrak{F} \lesssim^{*}_{f} \mathfrak{H}$.

\paragraph{Comparing parameteric families.}

One may quickly observe that in the case that we compare two $\leqslant$-parametric families $\mathfrak{H}$ and $\mathfrak{F}$ where $\mathfrak{F}$ is \textsl{finite} then $\mathfrak{H} \lesssim^{*} \mathfrak{F}$ is equivalent to $\mathfrak{H} \lesssim^{*}_{f} \mathfrak{F}$ for some function $f \colon \Nbbb \to \Nbbb.$

\begin{lemma}\label{equiv_equal_equiv_finite}
Let $\mathfrak{H}$ and $\mathfrak{F}$ be two $\leqslant$-parametric families such that $\mathfrak{F}$ is finite.
Then, $\mathfrak{H} \lesssim^{*} \mathfrak{F}$ if and only if there exists a function $f \colon \Nbbb \to \Nbbb$ such that $\mathfrak{H} \lesssim^{*}_{f} \mathfrak{F}.$
\end{lemma}
\begin{proof} We only show the forward direction since the reverse trivially holds by definition.
Assume that $\mathfrak{H} \lesssim^{*} \mathfrak{F}$.
By definition, for every $\mathscr{F} \in \mathfrak{F}$ there exists $\mathscr{H} \in \mathfrak{H}$ such that $\mathscr{H} \lesssim_{f_{\mathscr{F}}} \mathscr{F},$ for some function $f_{\mathscr{F}} \colon \mathbb{N} \to \mathbb{N}.$
This implies that, for every $\mathscr{F} \in \mathfrak{F}$ there exists $\mathscr{H} \in \mathscr{H}$ such that for every $k \in \Nbbb,$ $\mathscr{H}_{k} \leqslant \mathscr{F}_{f_{\mathscr{F}}(k)}.$
Let $f \colon \mathbb{N} \to \mathbb{N}$ be the function such that $f(k) \coloneqq \max\{ f_{\mathscr{F}}(k) \mid \mathscr{F} \in \mathfrak{F} \},$ for every $k \in \Nbbb.$
Since $\mathfrak{F}$ is finite, $f$ is well-defined.
Moreover, since $\mathscr{F}$ is $\leqslant$-increasing, for every $\mathscr{F} \in \mathfrak{H},$ it is implied that, for every $\mathscr{F} \in \mathfrak{F}$ there exists $\mathscr{H} \in \mathfrak{H}$ such that for every $k \in \Nbbb,$ $\mathscr{H}_{k} \leqslant \mathscr{F}_{f(k)}.$
Therefore, $\mathfrak{H} \lesssim^{*}_{f} \mathfrak{F}$ and we conclude.
\end{proof}

We should remark that it is unclear whether $\lesssim^{*}$ and $\lesssim^{*}_{f}$ are equivalent when both $\leqslant$-parametric graphs are infinite.
The reason is that it is unclear whether there exists a function $f \colon \mathbb{N} \to \mathbb{N},$ defined as $f(k) \coloneqq \sup\{ f_{\mathscr{F}}(k) \mid \mathscr{F} \in \mathfrak{F} \},$ for every $k \in \Nbbb,$ precisely because the supremum $\sup\{ f_{\mathscr{F}}(k) \mid \mathscr{F} \in \mathfrak{F} \}$ may not exist.

\medskip
An immediate corollary of the previous lemma is the following.

\begin{corollary}
Let $\mathfrak{H}$ and $\mathfrak{F}$ be two finite $\leqslant$-parametric families.
Then, $\mathfrak{H} \equiv^{*} \mathfrak{F}$ if and only if there exists a function $f \colon \Nbbb \to \Nbbb$ such that $\mathfrak{H} \equiv^{*}_{f} \mathfrak{F}.$
\end{corollary}

The necessity of the restricted definition of $\equiv^{*}_{f}$ arises when attempting to compare the parameters corresponding to any two $\leqslant$-parametric families via $\preceq$ by comparing the $\leqslant$-parametric families via $\lesssim^{*}$ instead.

\begin{lemma}
Let $\mathfrak{H}$ and $\mathfrak{F}$ be two $\leqslant$-parametric families.
Then, $\p_{\mathfrak{H}} \preceq \p_{\mathfrak{F}}$ if and only if there exists a function $f \colon \Nbbb \to \Nbbb$ such that $\mathfrak{F} \lesssim^{*}_{f} \mathfrak{H}.$
\end{lemma}
\begin{proof}
We have that $\p_{\mathfrak{H}} \preceq \p_{\mathfrak{F}}$ if and only if, by \autoref{parcompobs}, there exists a function $f \colon \mathbb{N} \to \mathbb{N}$ such that $\mathcal{G}_{\p_{\mathfrak{F}}, k} \subseteq \mathcal{G}_{\p_{\mathfrak{H}}, f(k)}$ if and only if, by \autoref{ofjfhfdpia}, there exists a function $f \colon \mathbb{N} \to \mathbb{N}$ such that for every $k \in \Nbbb,$ $\excl_{\leqslant}(\{ \mathscr{F}_{k} \mid \mathscr{F} \in \mathfrak{F}\}) \subseteq \excl_{\leqslant}(\{ \mathscr{H}_{f(k)} \mid \mathscr{H} \in \mathfrak{H}\})$ if and only if, by \autoref{subobs}, for every $k \in \Nbbb,$  $\{ \mathscr{F}_{k} \mid \mathscr{F} \in \mathfrak{F}\} \leqslant^{*} \{ \mathscr{H}_{f(k)} \mid \mathscr{H} \in \mathfrak{H} \}$ if and only if, by definition of $\leqslant^{*},$ for every $k \in \Nbbb,$ for every $\mathscr{H} \in \mathfrak{H}$ there exists $\mathscr{F} \in \mathfrak{F}$ such that $\mathscr{F}_{k} \leqslant \mathscr{H}_{f(k)}$ if and only if $\mathfrak{F} \lesssim^{*}_{f} \mathfrak{H}.$
\end{proof}

As an immediate corollary we obtain that:

\begin{corollary}\label{univobsprops}
Let $\mathfrak{H}$ and $\mathfrak{F}$ be two $\leqslant$-parametric families.
Then, $\p_{\mathfrak{H}} \sim \p_{\mathfrak{F}}$ if and only if there exists a function $f \colon \Nbbb \to \Nbbb$ such that $\mathfrak{H} \equiv^{*}_{f} \mathfrak{F}.$
\end{corollary}

As implied by \autoref{equiv_equal_equiv_finite}, we may obtain the following variants of the two previous statements.

\begin{corollary}\label{univobsprops_finite}
Let $\mathfrak{H}$ and $\mathfrak{F}$ be two $\leqslant$-parametric families such that $\mathfrak{H}$ is finite.
Then, $\p_{\mathfrak{H}} \preceq \p_{\mathfrak{F}}$ if and only if $\mathfrak{F} \lesssim^{*} \mathfrak{H}.$
\end{corollary}

\begin{corollary}\label{finite_param_families_equiv_params}
Let $\mathfrak{H}$ and $\mathfrak{F}$ be two finite $\leqslant$-parametric families.
Then, $\p_{\mathfrak{H}} \sim \p_{\mathfrak{F}}$ if and only if $\mathfrak{H} \equiv^{*} \mathfrak{F}.$
\end{corollary}

To conclude the discussion on comparing $\leqslant$-parametric families, the following lemma shows that between equivalent parametric families there is a bijection, mapping the parametric graphs of one to equivalent parametric graphs of the other.

\begin{lemma}\label{univobsbijection}
Let $\mathfrak{H}$ and $\mathfrak{F}$ be two $\leqslant$-parametric families.
Then, $\mathfrak{H} \equiv^{*} \mathfrak{F}$ if and only if there exists a bijection $\rho  \colon  \mathfrak{H} \to \mathfrak{F},$ such that for every $\mathscr{H} \in \mathfrak{H},$ $\mathscr{H} \approx \rho(\mathscr{H}).$
\end{lemma}
\begin{proof}
We only show the forward direction as the reverse is immediate.
Assume that $\mathfrak{H} \equiv^{*} \mathfrak{F}.$ 
We define a bijection $\rho \colon \mathfrak{H} \to \mathfrak{F}$ with the desired properties as follows.
Let $\mathscr{H} \in \mathfrak{H}.$
Since $\mathfrak{F} \lesssim^{*} \mathfrak{H},$ there exists $\mathscr{F} \in \mathfrak{F}$ such that $\mathscr{F} \lesssim \mathscr{H}.$
Now, for this $\mathscr{F},$ since $\mathfrak{H} \lesssim^{*} \mathfrak{F},$ there exists $\mathscr{H}' \in \mathfrak{H}$ such that $\mathscr{H}' \lesssim \mathscr{F}.$
However, it must be that $\mathscr{H}' \approx \mathscr{F} \approx \mathscr{H},$ since $\mathfrak{H}$ is an $\lesssim$-antichain.
Then, it suffices to define $f(\mathscr{H}) \coloneqq \mathscr{F}.$
\end{proof}

We can prove the corresponding statement for $\equiv^{*}_{f}$ as well.
The identical proof is ommited.

\begin{lemma}\label{univobsbijection_func}
Let $\mathfrak{H}$ and $\mathfrak{F}$ be two $\leqslant$-parametric families.
Then, there exists a function $f \colon \Nbbb \to \Nbbb$ such that, $\mathfrak{H} \equiv^{*}_{f} \mathfrak{F}$ if and only if there exists a bijection $\rho  \colon  \mathfrak{H} \to \mathfrak{F},$ such that for every $\mathscr{H} \in \mathfrak{H},$ $\mathscr{H} \approx_{f} \rho(\mathscr{H}).$
\end{lemma}

\paragraph{Parametric families and the class obstructions of their parameters.}

We first present a lemma that precisely determines the $\leqslant$-closed classes where $\p_{\mathscr{H}}$ is bounded, for any graph sequence $\mathscr{H}$.

\begin{lemma}\label{simplebounded}
Let $\mathscr{H}$ be a graph sequence. Then
\begin{enumerate}
\item $\B(\p_{\mathscr{H}}) = \closure{\subseteq}{\{ \excl(\m{\leqslant}{\mathscr{H}_{\geq k})} \mid k \in \mathbb{N} \}}$ and
\item  $\B(\p_{\mathscr{H}}) = \{ \mathcal{G} \in \down{\leqslant}{\gall} \mid \exists k \in \mathbb{N} : \mathcal{G} \subseteq \excl(\m{\leqslant}{\mathscr{H}_{\geq k})} \}.$
\end{enumerate}
\end{lemma}
\begin{proof}

We may quickly observe that the fact that $\closure{\subseteq}{\{ \excl(\m{\leqslant}{\mathscr{H}_{\geq k})} \mid k \in \mathbb{N} \}} = \{ \mathcal{G} \in \down{\leqslant}{\gall} \mid \exists k \in \mathbb{N} : \mathcal{G} \subseteq \excl(\m{\leqslant}{\mathscr{H}_{\geq k})} \}$ follows directly from the definition of downward $\subseteq$-closure for the set $\{ \excl(\m{\leqslant}{\mathscr{H}_{\geq k})} \mid k \in \mathbb{N} \}$ in $\down{\leqslant}{\gall}$.
Now, by \autoref{obs_bounded_class}, $\mathcal{G} \in \B(\p_{\mathscr{H}})$ if and only if $\mathcal{G} \subseteq \mathcal{G}_{\p_{\mathscr{H}}, k}$ for some $k \in \mathbb{N}.$
Then, our claim follows, since by \autoref{thif0lkf}, $\mathcal{G}_{\p_{\mathscr{H}}, k} = \excl(\m{\leqslant}{\mathscr{H}_{\geq k})}.$
\end{proof}

We continue with a second lemma on how the comparison of two $\leqslant$-parameteric graphs via the quasi-ordering relation $\lesssim$ we have defined on the set of $\leqslant$-parametric graphs, translates to the comparison of their corresponding parameters via $\preceq,$ and to the comparison of their respective $\leqslant$-closures in terms of inclusion.

\begin{lemma}\label{obsprime} Let $\mathscr{H}$ and $\mathscr{F}$ be two $\leqslant$-parametric graphs.
Then, the following statements are equivalent:
\begin{enumerate}
\item $\p_{\mathscr{H}} \succeq \p_{\mathscr{F}};$
\item $\closure{\leqslant}{\mathscr{H}} \subseteq \closure{\leqslant}{\mathscr{F}};$
\item $\mathscr{H} \lesssim \mathscr{F}.$
\end{enumerate}
\end{lemma}
\begin{proof}

We first show how \emph{1.} implies \emph{2.}.
Suppose that $G$ is a graph of $\closure{\leqslant}{\mathscr{H}}$. 
Then, there exists $k\in\mathbb{N}$ such that $G \leqslant \mathscr{H}_k$. 
It follows that $\{G\}\leqslant^* \{\mathscr{H}_k\}$. 
In turn, by \autoref{subobs}, this implies that $\excl_{\leqslant}(\{G\})\subseteq \excl_{\leqslant}(\{\mathscr{H}_k\})$. 
Since $\mathscr{H}$ is $\leqslant$-increasing, we have that $\excl_{\leqslant}(\{\mathscr{H}_k\}) = \excl_{\leqslant}(\m{\leqslant}{\mathscr{H}_{\geq k})}$. 
By \autoref{simplebounded}, we obtain that $\excl_{\leqslant}(\{G\}) \in \B_{\leqslant}(\p_{\mathscr{H}})$. 
Since $\p_{\mathscr{H}} \succeq \p_{\mathscr{F}}$, by \autoref{parsimbounded}, we have $\B_{\leqslant}(\p_{\mathscr{H}}) \subseteq \B_{\leqslant}(\p_{\mathscr{F}})$ and thereby $\excl_{\leqslant}(\{G\}) \in \B_{\leqslant}(\p_{\mathscr{F}})$.
By \autoref{simplebounded}, there exists $k' \in \mathbb{N}$ such that $\excl_{\leqslant}(\{G\}) \subseteq \excl_{\leqslant}(\m{\leqslant}{\mathscr{F}_{\geq k'})}$.
Again, since $\mathscr{F}$ is $\leqslant$-increasing, we have $\m{\leqslant}{\mathscr{F}_{\geq k'}} = \{\mathscr{F}_{k'}\}$. 
Thereby, by \autoref{subobs}, $\excl_{\leqslant}(\{G\})\subseteq \excl_{\leqslant}(\{\mathscr{F}_{k'}\})$ implies $\{G\}\leqslant^*\{\mathscr{F}_{k'}\}$. 
It follows that there exists $k'\in\mathbb{N}$ such that $G\leqslant \mathscr{F}_{k'}$, in other words, $G \in \closure{\leqslant}{\mathscr{F}}$.

Now, we show how \emph{2.} implies \emph{3.}
For every $i \in \mathbb{N},$ $\mathscr{H}_{i} \in \closure{\leqslant}{\mathscr{H}}.$ 
Since $\mathscr{H}_{i} \in \closure{\leqslant}{\mathscr{F}},$ by definition there exists $j \in \mathbb{N}$ such that $\mathscr{H}_{i} \leqslant \mathscr{F}_{j}.$

To conclude, we show how \emph{3.} implies \emph{1.}
Let $\mathcal{G} \in \B_{\leqslant}(\p_{\mathscr{H}}).$ 
By \autoref{simplebounded}, there exists $i \in \mathbb{N}$ such that $\mathcal{G} \subseteq \excl_{\leqslant}(\{ \mathscr{H}_{i}\}).$ 
By assumption, there exists $j \in \mathbb{N}$ such that $\mathscr{H}_{i} \leqslant \mathscr{F}_{j},$ equivalently $\{ \mathscr{H}_{i} \} \leqslant^{*} \{ \mathscr{F}_{j} \}.$ 
By \autoref{subobs}, we have $\excl_{\leqslant}(\{ \mathscr{H}_{i} \}) \subseteq \excl_{\leqslant}(\{ \mathscr{F}_{j} \}),$ which implies that $\mathcal{G} \subseteq \excl_{\leqslant}(\{ \mathscr{F}_{j} \})$. 
This implies that $\mathcal{G} \in \B_{\leqslant}(\p_{\mathscr{F}}),$ therefore  $\p_{\mathscr{H}} \succeq \p_{\mathscr{F}}$ follows from \autoref{parsimbounded}.
\end{proof}

The following two lemmas precisely characterize the class obstruction of parameters associated to parametric graphs and, in turn, to finite parametric families as the set of $\leqslant$-closed classes that correspond to the $\leqslant$-closures of the involved $\leqslant$-parametric graphs.

\begin{lemma}\label{primeunbbasis} Let $\mathscr{H}$ be a $\leqslant$-parametric graph.
Then,
$$\cobs_{\leqslant}(\p_{\mathscr{H}}) = \{ \closure{\leqslant}{\mathscr{H}} \}.$$
\end{lemma}
\begin{proof}
By definition of $\p_{\mathscr{H}}$ and the fact that $\mathscr{H}$ is a $\leqslant$-parametric graph, we have that $\closure{\leqslant}{\mathscr{H}} \in \UNB_{\leqslant}(\p_{\mathscr{H}})$.
Now, consider any $\mathcal{G} \in \UNB_{\leqslant}(\p_{\mathscr{H}})$.
Then for every $k \in \mathbb{N}$, there exists a graph $G \in \mathcal{G}$ such that $\p_{\mathscr{H}}(G) > k.$
This implies that $\mathscr{H}_{k} \leqslant G.$
Since $\mathcal{G} \in \down{\leqslant}{\gall}$, $\mathscr{H}_k\in\mathcal{G}$ and thereby $ \closure{\leqslant}{\mathscr{H}} \subseteq \mathcal{G}$. 
Hence $\UNB_{\leqslant}(\p_{\mathscr{H}}) = \upclosure{\subseteq}{(\closure{\leqslant}{\mathscr{H}})}.$
\end{proof}

\begin{lemma}\label{primcollunbbasis} Let $\mathfrak{H}$ be a finite $\leqslant$-parametric family.
Then,
$$\cobs_{\leqslant}(\p_{\mathfrak{H}}) = \bigcup_{\mathscr{H} \in \mathfrak{H}} \{\{ \closure{\leqslant}{\mathscr{H}} \}\}.$$
\end{lemma}
\begin{proof}
One may observe that $\B_{\leqslant}(\p_{\mathfrak{H}}) = \bigcap_{\mathscr{H} \in \mathfrak{H}} \B_{\leqslant}(\p_{\mathscr{H}})$.
Let $\mathcal{G} \in \B_{\leqslant}(\p_{\mathfrak{H}}).$
By definition of $\p_{\mathfrak{H}},$ there exists $k \in \Nbbb$ such that $\mathcal{G} \subseteq \excl_{\leqslant}(\{ \mathscr{H}_{k} \mid \mathscr{H} \in \mathfrak{H} \})$.
This in turn implies that, $\mathcal{G} \subseteq \excl_{\leqslant}({\{ \mathscr{H}_{k} \}}),$ for every $\mathscr{H} \in \mathfrak{H}.$
Then, by definition of $\p_{\mathscr{H}},$ $\mathcal{G} \in \mathbb{B}_{\leqslant}(\p_{\mathscr{H}}),$ for every $\mathscr{H} \in \mathfrak{H}$ and therefore $\mathcal{G} \in \bigcap_{\mathscr{H} \in \mathfrak{H}} \B(\p_{\mathscr{H}}).$
For the reverse, let $\mathcal{G} \in \bigcap_{\mathscr{H} \in \mathfrak{H}} \B_{\leqslant}(\p_{\mathscr{H}}).$
Then, for every $\mathscr{H} \in \mathfrak{H},$ by definition of $\p_{\mathscr{H}},$ there exists $k_{\mathscr{H}} \in \Nbbb$ such that $\mathcal{G} \subseteq \excl_{\leqslant}(\{ \mathscr{H}_{k_{\mathscr{H}}} \}).$
Since $\mathfrak{H}$ is finite, we may define $k \coloneqq \max\{ k_{\mathscr{H}} \mid \mathscr{H} \in \mathfrak{H} \}$.
Then, since every $\mathscr{H} \in \mathfrak{H}$ is $\leqslant$-increasing, we have that $\mathcal{G} \subseteq \excl_{\leqslant}(\{ \mathscr{H}_{k} \mid \mathscr{H} \in \mathfrak{H} \}).$
This in turn shows that, by definition of $\p_{\mathfrak{H}},$ $\mathcal{G} \in \B_{\leqslant}(\p_{\mathfrak{H}}).$

Then, we have that $\UNB_{\leqslant}(\p_{\mathfrak{H}}) = \bigcup_{\mathscr{H} \in \mathfrak{H}} \UNB_{\leqslant}(\p_{\mathscr{H}}).$
Since $\mathfrak{H}$ is a $\lesssim$-antichain, by  \autoref{obsprime}, we have that $\bigcup_{\mathscr{H} \in \mathfrak{H}} \{\{ \closure{\leqslant}{\mathscr{H}} \}\}$ is an $\subseteq$-antichain. 
Hence by \autoref{primeunbbasis}, it is the inclusion-basis of $\UNB_{\leqslant}(\p_{\mathfrak{H}}).$
\end{proof}

The situation with infinite $\leqslant$-parametric families is more tricky.
We will see more on infinite $\leqslant$-parametric families later.

\paragraph{Two examples of minor-parametric families.}

Recall the definition of the complete ternary trees $\mathscr{T}=\langle \mathscr{T}_1,\mathscr{T}_2,\mathscr{T}_3,\mathscr{T}_4,\dots\rangle$.
The singleton $\mathfrak{T}=\{\mathscr{T}\}$, is trivially a minor-parametric family.
Moreover, it is easy to see that every forest is a minor of a large enough complete ternary tree.

\begin{observation}\label{prop_tree_omnivore}
For every forest $F$, there exists a constant $c_{F} \in \Nbbb$, such that $F$ is a minor of the 
complete ternary tree $\mathscr{T}_{c_{F}}$.
\end{observation}

Since every complete ternary tree is itself a forest, i.e., an acyclic graph, \autoref{prop_tree_omnivore} implies that the class $\gforest$ of all forests is exactly the set of all minors of complete ternary trees, i.e. $\closure{\leqslant_{\mathsf{m}}}{\mathscr{T}} = \gforest$.
Hence, \autoref{primcollunbbasis} implies that the $\leqslant$-class obstruction $\p_{\mathfrak{T}}$ is exactly the set containing the class of forests, i.e. $\{ \gforest \}$.

An \emph{apex forest} is a graph containing a vertex whose removal yields a forest. We use $\gfapex$ to denote the class of all apex forests.
An \emph{outerplanar} graph is a graph that has a plane embedding where all its vertices are incident to the same face.
We use $\gouterplanar$ to denote the class of outerplanar graphs.
 
\begin{observation}\label{prop_outer_omnivore} For every apex forest $D$ (resp. outerplanar  graph $Q$), there exists a constant $c_{D} \in \Nbbb$ (resp. $c_{Q} \in \Nbbb$), such that $D$ (resp. $Q$) is a minor of $\mathscr{T}^{a}_{c_{D}}$ (resp. $\mathscr{T}^{a*}_{c_{Q}}$).
\end{observation}

Notice that all graphs in $\mathscr{T}^{a}$ are apex forests, while all graphs in $\mathscr{T}^{a*}$ are outerplanar. 
Moreover, notice that $\mathfrak{B} = \{\mathscr{T}^a,\mathscr{T}^{a*}\}$ is a minor-parametric family.
As in the case of forests, combining this fact with \autoref{prop_outer_omnivore}, we have that the $\leqslant$-class obstruction of $\p_{\mathfrak{B}}$ is exactly the set containing the class of outerplanar graphs and the class of apex forests, i.e., the set $\{\gfapex, \gouterplanar\}$.

\subsection{Universal obstructions}
\label{univ_on4td}

We are now in the position to describe the main obstruction notion for graph parameters of our parametric viewpoint that will allow us to express graph parameters in a ``canonical obstructing'' form, via the exclusion of a set of parametric graphs.
We furthermore examine sufficient order theoretic conditions that ensure the existence of universal obstructions.
We also introduce the notion of \textsl{parametric obstruction} for a graph parameter, which is a variant of its class obstruction.

\medskip
We first introduce the definition of a $\leqslant$-universal obstruction for a graph parameter.

\begin{definition}[Universal obstruction]\label{def:univ_obs}
A $\leqslant$-parametric family $\mathfrak{H}$ is a $\leqslant$-universal obstruction for a graph parameter $\p$ if $\p$ and $\p_{\mathfrak{H}}$ are equivalent.
\end{definition}

Building on the results of the previous subsection, we show that, for any $\leqslant$-monotone parameter that admits a finite $\leqslant$-universal obstruction, its $\leqslant$-class obstruction exists and is closely related to the $\leqslant$-parametric graphs contained in its $\leqslant$-universal obstruction.

\begin{theorem}\label{uobs_to_cobs} Let $\p$ be a $\leqslant$-monotone parameter.
If $\p$ has a finite $\leqslant$-universal obstruction $\mathfrak{H}$ then
$$\cobs_{\leqslant}(\p) = \bigcup_{\mathscr{H} \in \mathfrak{H}} \{\{ \closure{\leqslant}{\mathscr{H}} \}\}.$$
\end{theorem}
\begin{proof}
By definition of $\leqslant$-universal obstructions, $\p \sim \p_{\mathfrak{H}}.$
By \autoref{parequivsamebounded}, $\UNB_{\leqslant}(\p) = \UNB_{\leqslant}(\p_{\mathfrak{H}}).$
Then, by \autoref{primcollunbbasis}, our claim holds.
\end{proof}

The next result shows that the class of $\leqslant$-rational parameters corresponds precisely to the $\leqslant$-monotone parameters that have finite $\leqslant$-universal obstructions.

\begin{theorem}\label{finiteuobs_rational}
A graph parameter $\p$ is $\leqslant$-rational if and only if it has a finite $\leqslant$-universal obstruction.
\end{theorem}
\begin{proof}
Let $\mathscr{H}$ be a $\leqslant$-rational sequence such that $\p \sim \p_{\mathscr{H}}.$ 
By \autoref{rationalperiodic}, there exists a finite $\leqslant$-parametric family $\mathfrak{H}$ such that $\p_{\mathscr{H}} \sim \p_{\mathfrak{H}}.$ 
Then by definition $\mathfrak{H}$ is a $\leqslant$-universal obstruction for $\p.$ 

For the reverse direction, let $\mathfrak{H} = \{ \mathscr{H}^{(i)} \mid i \in [r]\}$ be a $\leqslant$-universal obstruction for $\p$.
We define a graph sequence $\mathscr{H} = \langle \mathscr{H}_{k} \rangle_{k \in \mathbb{N}}$ such that for every $k \in \mathbb{N}$, $\mathscr{H}_{k} = \mathscr{H}^{(i)}_{j}$, where $k = i \cdot r + j$, for some choice of $i \in \mathbb{N}$, $j \in [r]$.
Clearly $\mathscr{H}$ is a $\leqslant$-rational sequence.
Define the $\leqslant$-parametric family $\mathfrak{H}'$ to be the set obtained starting from $\mathfrak{H}$ after removing from each $\leqslant$-parametric graph in $\mathfrak{H}$ all duplicates of elements that appear only a finite number of times in the $\leqslant$-parametric graph, as well as all elements that appear in any other $\leqslant$-parametric graph of $\mathfrak{H}$.
Since $\mathfrak{H}$ is a finite $\lesssim$-antichain, it is not hard to see that $\mathfrak{H} \equiv^{*} \mathfrak{H}'$.
Then, by \autoref{univobsprops}, $\p_{\mathfrak{H}} \sim \p_{\mathfrak{H}'}$.
Observe that $\mathfrak{H'}$ satisfies all necessary properties in order to apply \autoref{claim_rational_frak} of \autoref{rationalperiodic} for $\mathscr{H}$ and $\mathfrak{G}$ being $\mathfrak{H}'$, which shows that $\p_{\mathscr{H}} \sim \p_{\mathfrak{H}}$.
We conclude that $\p$ is $\leqslant$-rational.
\end{proof}

\paragraph{Two examples of minor-universal obstructions.}

Bienstock, Robertson, Seymour, and Thomas proved in \cite{Bienstock89grap} that  for every forest $F$, every graph with no minor isomorphic to $F$ has pathwidth at most $|V(F)|-2$.
Thereofore, every graph excluding $\mathscr{T}_k$ as a minor has pathwidth $O(k)$.
Moroever, it is easy to verify that $\pw(\mathscr{T}_k)=\Omega(\log(k))$ (see  e.g., \cite{EllisST94,BarriereFST03sear}).
This implies that the singleton $\mathfrak{T}=\{\mathscr{T}\}$ is a minor-universal obstruction for $\pw$ (recall \autoref{def:univ_obs}).
Additionally, \autoref{uobs_to_cobs} implies that the minor-class obstruction of $\pw$ is the set that contains the class of forests, i.e. $\cobs_{\leqslant_{\mathsf{m}}}(\pw) = \{ \gforest \}$.



Similar implications can be argued for biconnected pathwidth.
It was proven in \cite{dang2018minors, HuynhJMW20Seymour} that there exists a function $f \colon \Nbbb\to\Nbbb$ such that every graph with biconnected pathwidth at least $f(k)$ contains as a minor either an apex forest or an outerplanar graph on $k$ vertices.
Moreover, all graphs in $\mathscr{T}^{a}$ and in $\mathscr{T}^{a*}$ are biconnected, therefore their pathwidh and their biconnected pathwidth are identical.
Also, it is easy to observe that both $\mathscr{T}^{a}_{k}$  and $\mathscr{T}^{a*}$ contains the tree $\mathscr{T}_{\Omega(k)}$ as a minor. 
This implies that $\bipw(\mathscr{T}^{a})=\Omega(\log(k))$ and  $\bipw(\mathscr{T}^{a*})=\Omega(\log(k))$.
We conclude that $\mathfrak{B}=\{\mathscr{T}^{a},\mathscr{T}^{a*}\}$ is a minor-universal obstruction for $\bipw$ (recall \autoref{def:univ_obs}).
Finally, \autoref{uobs_to_cobs} implies that the minor-class obstruction of $\bipw$ is the set that contains the class of  apex forests and the class of outerplanar graphs, i.e. $\cobs_{\leqslant_{\mathsf{m}}}(\bipw) = \{ \gfapex,\gouterplanar \}$.



\subsection{Parametric obstructions}

In the previous subsections we introduced the notions of $\leqslant$-class obstructions and $\leqslant$-universal obstructions as a means of approximately characterizing graph parameters by an equivalent representative parameter in a ``canonical'' form.
In this manner a $\leqslant$-monotone graph parameter that admits a finite $\leqslant$-universal obstruction $\mathfrak{H}$, also admits a finite representation in the sense that the both its $\leqslant$-class obstruction and the set $\mathfrak{H}$ are finite and very closely related.
However, we should stress that these two notions do not yet offer a truly finite characterization as each $\leqslant$-parametric graph in $\mathfrak{H}$ is an infinite object and the same holds for each $\leqslant$-closed class in its $\leqslant$-class obstruction.

In this subsection, to circumvent the issue above, we introduce a variant of the $\leqslant$-class obstruction which we shall call the $\leqslant$-parametric obstruction, that under certain order theoretic assumptions will serve as a fully finite obstruction characterization for graph parameters.

\medskip
We proceed with the definition of the $\leqslant$-parametric obstruction of a $\leqslant$-monotone parameter.

\begin{definition}[Parametric obstruction]\label{def:pobs}
The $\leqslant$-\emph{parametric obstruction}
of a $\leqslant$-monotone parameter $\p$ is the set of $\leqslant$-obstruction sets $\{ \obs_{\leqslant}(\mathcal{G}) \mid \mathcal{G} \in \cobs_{\leqslant}(\p) \}$, which we denote by $\pobs_{\leqslant}(\p),$ when $\cobs_{\leqslant}(\p)$ exists.

Note that, if $(\gall, \leqslant)$ is a well-quasi-ordering, then the set $\pobs_{\leqslant}(\mathsf{p})$ exists, for every $\leqslant$-monotone parameter $\mathsf{p},$ and moreover every $\leqslant$-obstruction set in $\pobs_{\leqslant}(\mathsf{p})$ is finite.
\end{definition}

Having introduced all of our obstruction notions for graph parameters, we are now in the position to state a theorem that shows how graph parameters may be equivalently compared through all the different notions we have seen so far.

\begin{theorem}\label{relationequiv} Let $\p$ and $\p'$ be two $\leqslant$-monotone parameters. If $\p$ and $\p'$ both admit $\leqslant$-universal obstructions and their $\leqslant$-class obstructions exist, then the following statements are equivalent:
\begin{enumerate}
\item $\p \succeq \p'$;
\item $\cobs_{\leqslant}(\p) \subseteq^{*} \cobs_{\leqslant}(\p')$;
\item $\pobs_{\leqslant}(\p) \leqslant^{**} \pobs_{\leqslant}(\p')$;\footnote{$\leqslant^{**}$ denotes the Smyth extension of $\leqslant^{*}$.}
\item for every $\leqslant$-universal obstruction $\mathfrak{H}$ of $\p$ and every $\leqslant$-universal obstruction $\mathfrak{F}$ of $\p'$, there exists a function $f \colon \mathbb{N} \to \mathbb{N}$ such that $\mathfrak{H} \lesssim^{*}_{f} \mathfrak{F}$;
\item for every $\leqslant$-closed class $\mathcal{G},$ if $\p$ is bounded in $\mathcal{G}$ then $\p'$ is bounded in $\mathcal{G}.$
\end{enumerate}
\end{theorem}
\begin{proof}
By \autoref{parsimbounded}, $\p \succeq \p'$ if and only if  $\B_{\leqslant}(\p) \subseteq \B_{\leqslant}(\p').$
By \autoref{def:cobs} and \autoref{exclminorssmith}, $\B_{\leqslant}(\p) \subseteq \B_{\leqslant}(\p')$ if and only if $\cobs(\p) \subseteq^{*} \cobs(\p').$
Then, by \autoref{def:pobs} and \autoref{subobs}, $\cobs(\p) \subseteq^{*} \cobs(\p')$ if and only if $\pobs(\p) \leqslant^{**} \pobs(\p').$
Also, by \autoref{def:univ_obs} and \autoref{univobsprops_finite}, $\p \succeq \p'$ if and only if there exists a function $f \colon \mathbb{N} \to \mathbb{N}$ such that $\mathfrak{H} \lesssim^{*}_{f} \mathfrak{F},$ for any $\leqslant$-universal obstruction $\mathfrak{H}$ of $\p$ (resp. $\mathfrak{F}$ of $\p'$).
\end{proof}

We remark that $(\gall, \leqslant)$ is assumed to be a partial-ordering.
Moreover, $(\down{\leqslant}{\gall}, \subseteq)$ is also a partial-ordering.
These two facts, combined with the fact that the $\leqslant$-class obstructions and the $\leqslant$-parametric obstructions are inclusion-antichains and $\leqslant^{*}$-antichains respectively, and \autoref{relationequiv}, imply the following corollary.
 
\begin{corollary}\label{relationequal} Let $\p$ and $\p'$ be two $\leqslant$-monotone parameters. If $\p$ and $\p'$ both admit $\leqslant$-universal obstructions and their $\leqslant$-class obstructions exist, then the following statements are equivalent:
\begin{enumerate}
\item $\p \sim \p'$;
\item $\cobs_{\leqslant}(\p) = \cobs_{\leqslant}(\p')$;
\item $\pobs_{\leqslant}(\p) = \pobs_{\leqslant}(\p')$;
\item for every $\leqslant$-universal obstruction $\mathfrak{H}$ of $\p$ and every $\leqslant$-universal obstruction $\mathfrak{F}$ of $\p'$, there exists a function $f \colon \mathbb{N} \to \mathbb{N}$ such that $\mathfrak{H} \equiv^{*}_{f} \mathfrak{F}$;
\item for every $\leqslant$-closed class $\mathcal{G},$ $\p$ is bounded in $\mathcal{G}$ if and only if $\p'$ is bounded in $\mathcal{G}.$
\end{enumerate}
\end{corollary}

We stress that $\leqslant$-class obstructions and $\leqslant$-parametric obstructions, when they exist, are unique.
Certainly, this is not the case with universal obstructions as the equivalence classes of $\equiv^{*}_{f}$ may contain (infinitely) many $\leqslant$-parametric families.

\paragraph{Two examples of minor-parametric obstructions.}

Finally, having introduced $\leqslant$-parametric obstructions, by the discussion at the end of \autoref{univ_on4td}, we can conclude that the minor-parametric obstruction of pathwidth is the set containing the set containing the triangle $K_{3}$, i.e. $\pobs_{\leqslant_{\mathsf{m}}}(\pw) = \big\{  \{ K_{3}\}  \big\}$.
Therefore, the singleton containing the singleton containing the triangle can be seen as a finite characterization of the parameter of pathwidth.

For the case of biconnected pathwidth, it is well known that $\obs_{\leqslant_{\mathsf{m}}}(\gouterplanar) = \{K_{4},K_{2,3}\}$.
Moreover, according to \cite{DinneenCF0fForbidden}, $\obs_{\leqslant_{\mathsf{m}}}(\gfapex) = \{K_{4},S_{3},2\cdot K_{3}\}$, depicted in \autoref{two_obs_bpw}.
Therefore, $\pobs_{\leqslant_{\mathsf{m}}}(\bipw) = \big\{\{K_{4},S_{3},2\cdot K_{3}\}, \{K_{4},K_{2,3}\}  \big\}$.

\begin{figure}[htbp]
\centering
\includegraphics[width=0.55\linewidth]{biconnected_pathwidth_pobs}
\caption{\label{biconnected_pathwidth_pobs}The graphs in the two minor-obstruction sets in $\pobs_{\leqslant_{\mathsf{m}}}(\bipw).$}
\label{two_obs_bpw}
\end{figure}

\subsection{Omnivores}\label{subsec_omnivores}

In this subsection we introduce the notion of a $\leqslant$-omnivore of a $\leqslant$-closed class in an attempt to understand which $\leqslant$-closed classes admit a regular description in terms of the $\leqslant$-closure of some $\leqslant$-parametric graph.
We provide an exact order theoretic characterization for when $\leqslant$-closed classes admit $\leqslant$-omnivores.
This results will have a tremendous impact in our quest for characterizing class properties in terms of graph parameters in a ``canonical'' way.

\medskip
We proceed with the definition of a $\leqslant$-omnivore of a $\leqslant$-closed class.

\begin{definition}[Omnivore]\label{def:omnivore}
A $\leqslant$-parametric graph $\mathscr{H}$ is a $\leqslant$-\emph{omnivore} of a $\leqslant$-closed graph class $\mathcal{G}$ if $\closure{\leqslant}{\mathscr{H}} = \mathcal{G}.$
\end{definition}

It turns out that the $\leqslant$-closed classes that admit $\leqslant$-omnivores are precisely the $\leqslant$-ideals of $\gall.$

\begin{lemma}\label{omnivore}
A $\leqslant$-closed class $\mathcal{G}$ has a $\leqslant$-omnivore if and only if $\mathcal{G}$ is a $\leqslant$-ideal of $\gall.$
\end{lemma}
\begin{proof}
Let $\mathscr{H}$ be a $\leqslant$-omnivore of $\mathcal{G}.$ 
We show that $\mathcal{G}$ is $\leqslant$-directed. 
For every $G \in \mathcal{G}$ we have that there exists $k \in \mathbb{N}$ such that $G \leqslant \mathscr{H}_{k}.$ 
For $G, G' \in \mathcal{G},$ let $k, k' \in \mathbb{N}$ be such that $G \leqslant \mathscr{H}_{k}$ and $G' \leqslant \mathscr{H}_{k'}.$ 
Let $m := \max\{k, k'\}.$ 
By definition of $\mathscr{H},$ $G \leqslant \mathscr{H}_{m}$ and $G' \leqslant \mathscr{H}_{m}.$ 
As $\mathscr{H}_{m} \in \mathcal{G},$ it follows that $\mathcal{G}$ is $\leqslant$-directed, therefore it is a $\leqslant$-ideal of $\gall.$

For the reverse direction, we assume that $\mathcal{G}$ is $\leqslant$-directed and we show how to define a $\leqslant$-omnivore $\mathscr{H}$ of $\mathcal{G}.$ 
Let $\mathcal{G}_{r} \coloneqq \{ G \in \mathcal{G} \mid |V(G)| ≤ r \}$, for every $r \in \mathbb{N}.$ 
We define $\mathscr{H}_{1}$ as a graph in $\mathcal{G}$ such that for every graph $G \in \mathcal{G}_{1}$, $G \leqslant \mathscr{H}_{1}$.
For $k > 1,$ we define $\mathscr{H}_{k}$ as a graph in $\mathcal{G}$ such that for every graph $G \in \mathcal{G}_{k}$, $G \leqslant \mathscr{H}_{k}$ and $\mathscr{H}_{k-1} \leqslant \mathscr{H}_{k}$.
$\mathscr{H}_{k}$ is well-defined, since $\mathcal{G}_{k} \cup \{ \mathscr{H}_{k-1}\}$ is finite and $\mathcal{G}$ is $\leqslant$-directed. 
By definition $\mathscr{H}_{k-1} \leqslant \mathscr{H}_{k}.$ 
Also, for any graph $G \in \mathcal{G}$ we have that $G \leqslant \mathscr{H}_{|V(G)|}.$ 
Therefore $\mathscr{H} = \langle \mathscr{H}_{k} \rangle_{k \in \mathbb{N}}$ is a $\leqslant$-omnivore of $\mathcal{G}.$
\end{proof}

\paragraph{Two examples of minor-omnivores.}

Recall that the graph class of forests is exactly the set of minors of complete ternary trees (\autoref{prop_tree_omnivore}).
In the terminology introduced in this subsection this translates to the fact that the parametric graph $\mathscr{T}$ is a minor-omnivore of the class of forests $\gforest$. 
Similarly, \autoref{prop_outer_omnivore} implies that the parametric graph $\mathscr{T}^{a}$ (resp. $\mathscr{T}^{a*}$) is a minor-omnivore of the class of apex forests $\gfapex$ (resp. the class of outerplanar graphs $\gouterplanar$).

\medskip
Beforehand, with \autoref{bp_is_directed}, we saw that, every $\leqslant$-monotone parameter corresponds to a directed $\leqslant$-class property.
With $\leqslant$-omnivores we now have a powerful tool at hand that will allow us, under suitable order theoretic assumptions, to prove the reverse \autoref{bp_is_directed}.

Starting from the $\leqslant$-class obstruction set of a directed $\leqslant$-class property we shall construct a $\leqslant$-parametric family, whose corresponding $\leqslant$-monotone parameter, precisely captures this property.

\paragraph{Parametric families and antichains of ideals.}

In fact we can prove that finite $\leqslant$-parametric families and finite inclusion-antichains of $\leqslant$-ideals are closely related.

Given a $\leqslant$-parametric family $\mathfrak{H}$, we use $[\mathfrak{H}]_{f}$ to denote the equivalence class of $\equiv^{*}_{f}$ that contains $\mathfrak{H}.$
Let $\mathbf{F}_{\leqslant}^{\mathsf{fin}}$ denote the set of all equivalence classes of $\equiv^{*}_{f},$ restricted only to finite $\leqslant$-parametric families, namely,
$$\mathbf{F}_{\leqslant}^{\mathsf{fin}} \coloneqq \{ [\mathfrak{H}]_{f} \mid \text{$\mathfrak{H}$ is a finite $\leqslant$-parametric family}\},$$
and
$\mathbf{I}_{\leqslant}^{\mathsf{fin}}$ denote the set of all finite inclusion-antichains of $\leqslant$-ideals of $\gall,$ namely,
$$\mathbf{I}_{\leqslant}^{\mathsf{fin}} \coloneqq \{ \mathbb{I} \in 2^{\idl{\leqslant}{\gall}} \mid \text{$
\mathbb{I}$ is a finite inclusion-antichain} \}.$$

The next theorem shows that the two sets $\mathbf{F}_{\leqslant}^{\mathsf{fin}}$ and $\mathbf{I}_{\leqslant}^{\mathsf{fin}}$ are in bijection in the following sense:

\begin{theorem}\label{bijection_coll_ideals}
There exists a unique bijection $\rho_{\ref{bijection_coll_ideals}}  \colon  \mathbf{F}_{\leqslant}^{\mathsf{fin}} \to \mathbf{I}_{\leqslant}^{\mathsf{fin}}$ such that, for every $[\mathfrak{G}]_{f} \in \mathbf{F}_{\leqslant}^{\mathsf{fin}}$, for every $\mathfrak{H} \in [\mathfrak{G}]_{f},$ $\cobs(\p_{\mathfrak{H}}) = \rho_{\ref{bijection_coll_ideals}}([\mathfrak{G}]_{f})$.
\end{theorem}
\begin{proof}
We define the desired bijection $\rho  \colon  \mathbf{F}_{\leqslant}^{\mathsf{fin}} \to \mathbf{I}_{\leqslant}^{\mathsf{fin}}$ as follows.
Consider a finite $\leqslant$-parametric family $\mathfrak{H}$.
\autoref{primcollunbbasis} tells us that $\cobs_{\leqslant}(\p_{\mathfrak{H}}) = \bigcup_{\mathscr{H} \in \mathfrak{H}} \{\{ \closure{\leqslant}{\mathscr{H}} \}\}$.
Observe that by definition of $\leqslant$-omnivores each $\mathscr{H} \in \mathfrak{H}$ is a $\leqslant$-omnivore of $\closure{\leqslant}{\mathscr{H}}$ and as such, \autoref{omnivore} implies that $\cobs_{\leqslant}(\p_{\mathfrak{H}})$ is a finite inclusion-antichain of $\leqslant$-ideals of $\gall$, i.e. $\cobs_{\leqslant}(\p_{\mathfrak{H}}) \in \mathbf{I}_{\leqslant}^{\mathsf{fin}}$.
We define $\rho([\mathfrak{H}]_{f}) \coloneqq \cobs_{\leqslant}(\p_{\mathfrak{H}})$.
The fact that $\rho$ is injective follows from \autoref{relationequal}, since for any two non-equivalent (with respect to $\equiv^{*}_{f}$) $\leqslant$-parametric families their corresponding parameter must have non-equal $\leqslant$-class obstructions.

To show that $\rho$ is also surjective consider $\mathbb{I} \in \mathbf{I}_{\leqslant}^{\mathsf{fin}}$.
\autoref{omnivore} implies the existence of a set $\mathfrak{F}$ of $\leqslant$-parametric graphs, each one being a $\leqslant$-omnivore of a different $\leqslant$-ideal of $\gall$ in $\mathbb{I}$.
Moreover, since $\mathbb{I}$ is a $\subseteq$-antichain, \autoref{obsprime} implies that $\mathfrak{F}$ is a $\lesssim$-antichain and thus a $\leqslant$-parametric family.
Observe that \autoref{primcollunbbasis} implies $\cobs_{\leqslant}(\p_{\mathfrak{F}}) = \mathbb{I}$.
Then, by definition of $\rho$, $\rho([\mathfrak{F}]_{f}) = \mathbb{I}$.

To conclude, uniqueness is implied by the fact that the $\leqslant$-class obstruction of any $\leqslant$-monotone parameter is unique.
\end{proof}

The previous theorem combined with \autoref{finiteuobs_rational} implies the same relationship between equivalence classes of $\leqslant$-rational parameters and finite inclusion-antichains of $\leqslant$-ideals of $\gall.$

For a graph parameter $\p$ we use $[\p]_{\sim}$ in order to denote the equivalence class of $\p$ with respect to $\sim,$ namely $[\p]_{\sim} = \{ \p' \mid \p \sim \p' \}.$ 
We write $\mathbf{P}_{\leqslant}^{\mathsf{rational}}$ for the set of equivalence classes of $\leqslant$-rational parameters, namely 
$$\mathbf{P}_{\leqslant}^{\mathsf{rational}} := \{ [\p]_{\sim} \mid \textrm{$\p$ is a $\leqslant$-rational parameter}\}.$$

The following variant of \autoref{bijection_coll_ideals} follows.

\begin{theorem}\label{bijection_rational_ideals} There exists a unique bijection $\rho_{\ref{bijection_rational_ideals}} \colon \mathbf{P}_{\leqslant}^{\mathsf{rational}} \to \mathbf{I}_{\leqslant}^{\mathsf{fin}}$ such that, for every $[\p]_{\sim} \in \mathbf{P}_{\leqslant}^{\mathsf{rational}}$, for every $\p \in [\p]_{\sim}$, $\cobs_{\leqslant}(\p) = \rho_{\ref{bijection_rational_ideals}}([\p]_{\sim})$. 
\end{theorem}
\begin{proof} We define the desired bijection $\rho \colon \mathbf{P}_{\leqslant}^{\mathsf{rational}} \to \mathbf{I}_{\leqslant}^{\mathsf{fin}}$ as follows.
Consider an equivalence class $[\p]_{\sim} \in \mathbf{P}_{\leqslant}^{\mathsf{rational}}$ and any $\leqslant$-rational parameter $\p \in [\p]_{\sim}$.
By \autoref{finiteuobs_rational} there exists a finite $\leqslant$-parametric family $\mathfrak{H}$ that is a $\leqslant$-universal obstruction for $\p$ and therefore for every parameter in $[\p]_{\sim}$.
We define $\rho([\p]_{\sim}) \coloneqq \rho_{\ref{bijection_coll_ideals}}([\mathfrak{H}]_{f})$.
The fact that $\rho$ is an injection follows from \autoref{relationequal}, i.e., from the fact that for two non-equivalent parameters we have non-equivalent $\leqslant$-universal obstructions.
To conclude that $\rho$ is a surjection as desired, recall that every finite $\leqslant$-parametric family $\p_{\mathfrak{H}}$ is trivially $\leqslant$-rational and that $\rho_{\ref{bijection_coll_ideals}}$ is a surjection.
\end{proof}

As a consequence of \autoref{bijection_coll_ideals} and \autoref{bijection_rational_ideals} we can prove that the implication of \autoref{params_obs_ideals} holds without the well-quasi-ordering assumption, when restricting to the class of $\leqslant$-rational parameters.

\begin{corollary}\label{corollary_uobs_ideal}
Let $\mathfrak{H}$ be a finite $\leqslant$-parametric family.
Then $\cobs_{\leqslant}(\p_{\mathfrak{H}})$ is a finite inclusion-antichain of $\leqslant$-ideals of $\gall.$
\end{corollary}

\begin{corollary}\label{cor:par_uobs_cobs_ideals}
Let $\p$ be a $\leqslant$-monotone parameter.
If $\p$ has a finite $\leqslant$-universal obstruction then $\cobs_{\leqslant}(\p)$ is a finite inclusion-antichain of $\leqslant$-ideals of $\gall.$
\end{corollary}

\paragraph{Directed properties with finite class obstructions are representable.}

As a consequence of \autoref{bijection_coll_ideals} and \autoref{bijection_rational_ideals} we may also obtain a partial result to the question we asked at the end of \autoref{from_params_to_props}.

Towards this, we define when a $\leqslant$-class property is \textsl{representable} by a set of $\leqslant$-parametric families.

\begin{definition}[Representable class property]
A $\leqslant$-class property $\mathbb{CP}$ is \emph{representable} by a set of $\leqslant$-parametric families $\mathbf{H}$ if $\mathbb{CP} = \bigcup_{\mathfrak{H} \in \mathbf{H}} \mathbb{B}_{\leqslant}(\p_{\mathfrak{H}}).$

If moreover $\mathbf{H}$ is a singleton $\{ \mathfrak{H} \}$, we say that $\mathbb{CP}$ is representable by the $\leqslant$-parametric family $\mathfrak{H}.$
\end{definition}

Under the well-quasi-ordering assumption, \autoref{bijection_coll_ideals} and \autoref{wqo_implies_directed_ideal} imply that $\leqslant$-class properties with a finite $\leqslant$-class obstruction set are representable by a single fingle $\leqslant$-parametric family if and only if they are directed.

\begin{theorem}\label{finite_directed_properties_representable}
Let $(\gall, \leqslant)$ be a well-quasi-ordering and $\mathbb{CP}$ be a $\leqslant$-class property with a finite $\leqslant$-class obstruction set.
Then, $\mathbb{CP}$ is representable by a finite $\leqslant$-parametric family if and only $\mathbb{CP}$ is directed.
\end{theorem}
\begin{proof}
We may observe that, by \autoref{wqo_implies_directed_ideal} and our assumption, we have that $\mathbb{CP}$ is directed if and only if $\cobs_{\leqslant}(\mathbb{CP})$ is a finite inclusion-antichain of $\leqslant$-ideals of $\gall.$
Then, our result follows directly by \autoref{bijection_coll_ideals}.
\end{proof}

At this point, given a $\leqslant$-closed class $\mathcal{G},$ we may observe that the $\leqslant$-class property $\mathbb{P}(\Gcal)$ is representable by the finite $\leqslant$-parametric family $\mathfrak{H}(\Gcal)$ that consists of $|\obs_{\leqslant}(\Gcal)|$ many $\leqslant$-parametric graphs defined as follows.
\begin{align*}
\mathfrak{H}(\Gcal) \ \coloneqq \ \{ \langle \mathscr{H}_{t} \rangle_{t \in \Nbbb} \mid Z \in \obs_{\leqslant}(\Gcal)\text{ and }\mathscr{H}_{t} = Z\text{, for every }t\in\Nbbb\}.
\end{align*}

As a result, by calling upon \autoref{finite_directed_properties_representable}, we observe that $\mathbb{P}(\Gcal)$ is directed, for every $\leqslant$-closed class $\Gcal.$

\medskip
Later, in \autoref{chakmrkfinfoerobes} we shall see that, under an even stronger order theoretic assumption, we may generalize \autoref{finite_directed_properties_representable} to capture {\sl any} directed $\leqslant$-class property. 

\section{Algorithmic consequences}\label{algorithms_section}

In this section, we explore the algorithmic consequences of having finite $\leqslant$-class obstructions sets for deciding the membership problem in a given $\leqslant$-class property.
Moreover, we complement these algorithmic results by examining the algorithmic consequences of having finite $\leqslant$-universal obstructions for approximating graph parameters.

\subsection{Parameterized problems and  algorithms} 

For algorithmic purposes we previously discussed, we introduce a suitable notion of ``efficient'' $\leqslant$-containment testing.
Towards this, we need a short introduction to parameterized algorithms and complexity.

\medskip
Let $Σ$ be a finite alphabet.
A \emph{parameterized problem} is a pair $\mathcal{P} = (Π,κ)$ where $Π$ is a subset of $Σ^*$ (encoding the \yes-instances of the problem) and $κ \colon Σ^*\to\mathbb{N}$ is a function mapping instances of $Π$ to non-negative integers.
A \emph{parameterized algorithm} (in short \FPT-algorithm) for $\mathcal{P}=(Π,κ)$ is an algorithm that decides if an instance $x\in Σ^*$ belongs in $Π$ or not, in time $f(κ(x))\cdot |x|^{O(1)}$ for some function $f \colon \mathbb{N}\to\mathbb{N}.$
If $f$ is a computable function, then we have an \emph{effective \FPT-algorithm}.
We consider parameterized problems on graphs, i.e., $x$ encodes some collection of graphs.
We say that $\mathcal{P}=(Π,κ)$ is \emph{(effectively) Fixed Parameter Tractable}, in short is (\emph{effectively}) \FPT, if it admits an (effective) \FPT-algorithm.

Not all parameterized problems are \FPT.
The design of \FPT-algorithms for parameterized problems or the proof that no such algorithm exists (subject to certain computational complexity assumptions), is the object of the field Parameterized Computation (see~\cite{CyganFKLMPPS15para, FlumG06Parameterized} for related textbooks).

\medskip
Let $\p$ be a graph parameter. 
We define $\mathcal{P}_{\p}=(Π_\p,κ_\p)$ as the parameterized problem where $x\in Σ^*$ encodes a pair $(G,k)$ where $G$ is a graph, $k\in\mathbb{N},$  $κ_\p(G,k)=k,$ and finally, $Π_\p=\{(G,k)\mid \p(G)≤k\}.$
We say that $\p$ is (effectively) \emph{\FPT-decidable} if $\mathcal{P}_{\p}$ is (effectively) \FPT.

Let $g \colon \mathbb{N}\to\mathbb{N}.$
We say that $\p$ is (effectively) \emph{\FPT-approximable with gap $g$} if there is an algorithm that either outputs that $(G,g(k))\in Π_{\p}$ or outputs that $(G,k)\not\in Π_{\p},$ in time $f(\p(G))\cdot |G|^{O(1)},$ for some (computable) function $f \colon \mathbb{N}\to\mathbb{N}.$

\paragraph{\FPT-decidable relations.}
Let $\leqslant$ be a quasi-ordering relation on graphs.
Consider the parameterized problem $\Pi_{\leqslant}$ where each $x\in Σ^*$ encodes two graphs $H$ and $G,$ i.e, $x=(H,G),$ $κ(H,G)=|H|,$ and $(H,G)\in \Pi_{\leqslant}$ if $H \leqslant G.$
We say that $\leqslant$ is \FPT-decidable if $Π_{\leqslant}$ is \FPT.

\paragraph{Equivalence with gap function.}
We next refine the definition of equivalence of two graph parameters.
Let $\mathbb{N}^{\mathbb{N}}$ denote the set of all functions mapping non-negative integers to non-negative integers.
Let $\p$ and $\p'$ be two equivalent graph parameters with gap function $f$.
We say that $\p$ and $\p'$ are equivalent with $F$ \emph{gap} if $f \in F$, where $F \subseteq \mathbb{N}^{\mathbb{N}}.$
We denote this as $\p \sim_{F} \p'.$

\subsection{Testing membership in class properties}

In this subsection we consider the problem of deciding membership in a given $\leqslant$-class property.
\medskip

Certainly, for the purpose of designing algorithms that decide membership in class properties, we need to consider a finite description of the input $\leqslant$-closed class.
For this, given a finite set of graphs $\mathcal{Z},$ we define the class of \emph{$\mathcal{Z}$-$\leqslant$-free} graphs, denoted by $\excl_{\leqslant}(\mathcal{Z})$, as the $\leqslant$-closed class that consists of the graphs that exclude all graphs in $\mathcal{Z}$ with respect to $\leqslant$.

With this notion at hand, we may now formally define the membership problem in a $\leqslant$-class property $\mathbb{CP}$.
\begin{eqnarray}
\begin{minipage}{11.2cm}
\noindent \textsc{Membership in $\leqslant$-class property $\mathbb{CP}$}
\begin{description}
\vspace{-0.7em}
  \item[Input:] A finite set $\mathcal{Z}$ of graphs.
\vspace{-0.7em}
  \item[Question:] Does the class of $\mathcal{Z}$-$\leqslant$-free graphs belong to $\mathbb{CP}$?
\end{description}
\end{minipage}\label{class_property_membership}
\end{eqnarray}

In what follows, we show that we can always decide \textsc{Membership in $\leqslant$-class property $\mathbb{CP}$} in polynomial time, given a $\leqslant$-class property whose $\leqslant$-class obstruction set is finite and assuming that $\leqslant$ is \FPT-decidable.

To achieve this, we need an algorithmic way to determine when one $\leqslant$-closed class is a subclass of another.
This follows immediately from \autoref{subobs}, assuming that $\leqslant$ is \FPT-decidable.

\begin{observation}\label{testing_subclass} Assume that $\leqslant$ is \FPT-decidable.
Then for every two finite sets of graphs $\Hcal$ and $\Fcal$ we can decide whether $\excl_{\leqslant}(\Hcal) \subseteq \excl_{\leqslant}(\Fcal)$ in polynomial time.
\end{observation}
\begin{proof} By definition of $\leqslant^{*}$ and \autoref{subobs} it suffices to test whether there exists a graph in $\Fcal$ that excludes every graph in $\Hcal$ with respect to $\leqslant$, which by the assumption that $\leqslant$ is \FPT-decidable can be done in polynomial time.
\end{proof}

With this tool at hand we can show that \textsc{Membership in $\leqslant$-class property $\mathbb{CP}$} is decidable in polynomial time, assuming that $\cobs_{\leqslant}(\mathbb{CP})$ is finite and that $\leqslant$ is \FPT-decidable.

\begin{theorem}\label{membership_is_decidable} Assume that $\leqslant$ is \FPT-decidable.
Then, for every $\leqslant$-class property $\mathbb{CP}$ for which $\cobs_{\leqslant}(\mathbb{CP})$ is finite, \textsc{Membership in $\leqslant$-class property $\mathbb{CP}$} is decidable in polynomial time.
\end{theorem}
\begin{proof} Given a finite set $\Zcal$ of graphs, using \autoref{testing_subclass}, it suffices to check whether $\Hcal \not\subseteq \excl_{\leqslant}(\Zcal),$ for every $\Hcal \in \cobs_{\leqslant}(\mathbb{CP}).$
Since $\cobs_{\leqslant}(\mathbb{CP})$ is finite, we conclude with a finite number of calls to the $\leqslant$-containment algorithm implied by the assumption that $\leqslant$ is \FPT-decidable.
\end{proof}

Let us note that the theorem above, polynomial time means that there exists an algorithm with running time of the form $f(\mathsf{a}(\cobs_{\leqslant}(\mathbb{CP}))) \cdot |\mathsf{b}(\mathcal{Z})|^{O(1)},$ where $f \colon \Nbbb \to \Nbbb$ is the function implied by the \FPT algorithm that decides $\Pi_{\leqslant}$ and $\mathsf{a}(\cobs_{\leqslant}(\mathbb{CP})) \coloneqq |\cobs_{\leqslant}(\mathbb{CP})| \cdot \max \{ \mathsf{b}(\obs_{\leqslant}(\Hcal)) \mid \Hcal \in \cobs_{\leqslant}(\mathbb{CP}) \}$ and $\mathsf{b}(\Zcal) \coloneqq |\Zcal| \cdot \max\{ |Z| \mid Z \in \Zcal \}.$

To conclude, an immediate corollary of \autoref{membership_is_decidable} and \autoref{wqo_implies_omega2_equiv_class_obs_finite} is the following.

\begin{corollary} Let $(\gall, \leqslant)$ be an $\omega^{2}$-well-quasi-ordering and $\leqslant$ be \FPT-decidable.
Then, for every $\leqslant$-class property $\mathbb{CP}$, \textsc{Membership in $\leqslant$-class property $\mathbb{CP}$} is decidable in polynomial time.
\end{corollary}

We wish to stress however that the algorithm we present for \eqref{class_property_membership} is not constructive, in the sense that it requires the knowledge of the $\leqslant$-obstruction sets, for every $\leqslant$-closed class in the $\leqslant$-class obstruction set of $\mathbb{CP}$.

\subsection{Effective parametric families}

Given a $\leqslant$-antichain $\mathcal{O}$ we say that $\mathcal{O}$ is a $\leqslant$-\emph{ideal antichain} if $\excl_{\leqslant}(\mathcal{O})$ is a $\leqslant$-ideal of $\gall.$
Let $\mathbf{O}_{\leqslant}^{\mathsf{fin}}$ denote the set of all finite $\leqslant^{*}$-antichains of $2^{\gall}$ of (possibly infinite) $\leqslant$-ideal antichains.

Observe that, by \autoref{subobs}, $\mathbf{O}_{\leqslant}^{\mathsf{fin}}$ bijectively corresponds to the set of $\leqslant$-obstruction sets of the inclusion-antichains of $\leqslant$-ideals of $\gall$ in $\mathbf{I}^{\mathsf{fin}}_{\leqslant}$.
Therefore, the following variant of \autoref{bijection_rational_ideals} immediately follows.

\begin{corollary}\label{bijection_rational_obs}
There exists a unique bijection $\rho_{\ref{bijection_rational_obs}}  \colon  \mathbf{P}_\leqslant^{\mathsf{rational}} \to \mathbf{O}_\leqslant^{\mathsf{fin}}$ such that, for every $[\p]_{\sim} \in \mathbf{P}_\leqslant^{\mathsf{rational}}$, for every $\p \in [\p]_{\sim}$, $\pobs_{\leqslant}(\p) = \rho_{\ref{bijection_rational_obs}}([\p]_{\sim}).$
\end{corollary}

Moving on, we give a definition of \textsl{effective} $\leqslant$-parametric families.

\begin{definition}
A $\leqslant$-parametric family $\mathfrak{H}$ is \emph{effective} if $\mathfrak{H}$ is finite and there exists an algorithm that, given $k \in \mathbb{N}$ outputs $\{ \mathscr{H}_{k} \mid \mathscr{H} \in \mathfrak{H}\}.$
\end{definition}

There is no guarantee that all finite $\leqslant$-parametric families are effective.
Nevertheless, we can guarantee this given some additional conditions.

\begin{theorem}\label{effective_existence}
Let $[\p]_{\sim} \in \mathbf{P}_\leqslant^{\mathsf{rational}}.$
If $\leqslant$ is \FPT-decidable and all sets in $\rho_{\ref{bijection_rational_obs}}([\p]_{\sim})$ are finite, then there exists an effective $\leqslant$-parametric family $\mathfrak{H}$ such that $\p_{\mathfrak{H}} \in [\p]_{\sim}.$
\end{theorem}
\begin{proof}
Let $<_{\mathsf{enum}}$ be a total-ordering\footnote{A total-ordering is a partial-ordering in which any two elements are comparable.} of $\gall$ computed by some deterministic enumeration algorithm such that for two graphs $G$ and $H,$ if $|V(G)| < |V(H)|,$ then $G <_{\mathsf{enum}} H.$

Let $\mathbb{I} = \{ \excl_{\leqslant}(\mathcal{O}) \mid \mathcal{O} \in \rho_{\ref{bijection_rational_obs}}([\p]_{\sim})\}.$
For a $\leqslant$-ideal $\mathcal{G} \in \mathbb{I}$ of $\gall,$ we define a $\leqslant$-omnivore $\mathscr{H}^{\mathcal{G}}$ of $\mathcal{G}$ as follows.
Let $\mathcal{G}_{r} = \{ G \in \mathcal{G} \mid |V(G)| ≤ r\}.$ 
Let $\mathscr{H}^{\mathcal{G}}_{1}$ be the $<_{\mathsf{enum}}$-minimum graph in $\mathcal{G},$ such that for every graph $G \in \mathcal{G}_{1}$, $G \leqslant \mathscr{H}^{\mathcal{G}}_{1}$.
For $k > 1,$ let $\mathscr{H}^{\mathcal{G}}_{k}$ be the $<_{\mathsf{enum}}$-minimum graph in $\mathcal{G}$, such that for every graph $G \in \mathcal{G}_{k}$, $G \leqslant \mathscr{H}^{\mathcal{G}}_{k}$ and $\mathscr{H}^{\mathcal{G}}_{k-1} \leqslant \mathscr{H}^{\mathcal{G}}_{k}$. 
Then, as in the proof of \autoref{omnivore}, $\mathscr{H}^{\mathcal{G}}$ is well-defined and a $\leqslant$-omnivore of $\mathcal{G}.$
 
\begin{claim}\label{compomnivore} Given a $\leqslant$-ideal $\mathcal{G} \in \mathbb{I},$ there exists an algorithm that, given $k \in \mathbb{N}$ outputs $\mathscr{H}^{\mathcal{G}}_{k}.$
\end{claim}
\begin{cproof}
The algorithm proceeds as follows.
Start by enumerating graphs in the order they are present in $<_{\mathsf{enum}}.$
For every such graph $G,$ first check whether $G \in \Gcal,$ by checking, using the $\leqslant$-containment algorithm given by the hypothesis, whether for every graph $Z \in \obs_{\leqslant}(\Gcal),$ $Z \not\leqslant G.$
If this check fails, we move on to the next graph.
Otherwise, we then check whether $\mathscr{H}^{\mathcal{G}}_{k-1} \leqslant G.$
If this check fails, we skip this graph as before, otherwise we continue.
Now, we compute the set $\{ H \in \gall \mid \text{$|V(H)| \leq k$ and $H \in \Gcal$}\}$ by a again enumerating graphs in the order of $<_{\mathsf{enum}}$ and using the $\leqslant$-containment algorithm.
Finally, check whether $H \leqslant G,$ for every graph $H$ in the set above.
The algorithm will terminate with output the first graph for which the previous checks succeed.
By definition of $\mathscr{H}^{\mathcal{G}},$ the algorithm will in fact terminate.
\end{cproof}

We conclude by defining $\mathfrak{H} \coloneqq \{ \mathscr{H}^{\mathcal{G}} \mid \mathcal{G} \in \mathbb{I} \}.$
Note that, by \autoref{primcollunbbasis} and definition of $\leqslant$-parametric obstructions, $\pobs_{\leqslant}(\p_{\mathfrak{H}}) = \rho_{\ref{bijection_rational_obs}}([\p]_{\sim}).$
Therefore, $\p_{\mathfrak{H}} \in [\p]_{\sim}$ and \autoref{compomnivore} completes our proof.
\end{proof}

\subsection{Cores and parameterized algorithms}

Consider the class of \textsl{computable} functions which we denote as $\REC.$
We can observe that given two equivalent parameters with computable gap, if one is \FPT-decidable the other is \FPT-approximable.

\begin{lemma}\label{thmfpt} Let $\p$ and $\p'$ be graph parameters such that $\p \sim_{\REC} \p'$ and $\p'$ be \FPT-decidable. 
Then $\p$ is \FPT-approximable.
\end{lemma} 
\begin{proof} Let $A_{\p'}$ be the \FPT-algorithm for $\mathcal{P}_{\p'}$ and $f \in \REC$ be such that $\p \sim_{f} \p'.$ 
Notice that, without loss of generality, we can assume that $f$ is monotone. 
The algorithm proceeds as follows: 
given a graph $G$ and $k \in \mathbb{N}$ execute $A_{\p'}$ on $G$ and $f(k).$ 
If $\p'(G) > f(k)$ then, clearly $\p(G) > k.$ 
Otherwise, if $\p'(G) ≤ f(k),$ then $\p(G) ≤ f(f(k)).$
\end{proof}

Now, let $\mathbf{F}^{\mathsf{eff}}_{\leqslant} \subseteq \mathbf{F}_{\leqslant}^{\mathsf{fin}}$ denote the set of all equivalence classes of effective $\leqslant$-parametric families, namely,
$$\mathbf{F}^{\mathsf{eff}}_{\leqslant} \coloneqq \{ [\mathfrak{G}]_{f} \mid \text{$\mathfrak{G}$ is an effective $\leqslant$-parametric family} \}$$
and let $F \subseteq \mathbb{N}^{\mathbb{N}}$ be a set of functions.
Given a $\leqslant$-rational parameter $\p,$ we define its \emph{$F$ $\leqslant$-core} as the set of all effective $\leqslant$-parametric families that are $\leqslant$-universal obstructions for $\p$ whose corresponding $\leqslant$-monotone parameters are equivalent with $F$ gap, namely,
$$F\textrm{-}\core_{\leqslant}(\p) := \{ \mathfrak{H} \in [\mathfrak{G}]_{f} \mid \text{$[\mathfrak{G}]_{f} \in \mathbf{F}^{\mathsf{eff}}_{\leqslant}$ and $\p \sim_{F} \p_{\mathfrak{H}}$} \}.$$

For simplicity, we write $\core_{\leqslant}(\p)$ instead of $\mathbb{N}^{\mathbb{N}}\textrm{-}\core_{\leqslant}(\p)$ and we call $\core_{\leqslant}(\p)$ the $\leqslant$-\emph{core} of $\p.$
We may now observe that, for $\leqslant$-rational parameters \autoref{effective_existence} implies the following corollary.

\begin{corollary}\label{corollary_fullcore} Let $[\p]_{\sim} \in \mathbf{P}_\leqslant^{\mathsf{rational}}.$ 
If $\leqslant$ is \FPT-decidable and all sets in $\rho_{\ref{bijection_rational_obs}}([\p]_{\sim})$ are finite, then all graph parameters in $[\p]_{\sim}$ have a non-empty $\leqslant$-core.
\end{corollary}

We call $\mathsf{REC}\textrm{-}\core_{\leqslant}(\p)$ the \emph{recursive} $\leqslant$-\emph{core} of $\p.$ 
We can now rephrase \autoref{thmfpt} in terms of $\leqslant$-universal obstructions.

\begin{theorem}\label{thmuobsfpt} Assume that $\leqslant$ is \FPT-decidable. 
Then  every $\leqslant$-rational parameter $\p$ with a non-empty recursive $\leqslant$-core is \FPT-approximable.
\end{theorem}
\begin{proof} Let $\mathfrak{H} \in \REC\textrm{-}\core_{\leqslant}(\p).$ 
Given a graph $G$ and $k \in \mathbb{N},$ by definition, $\p_{\mathfrak{H}}(G) \leq k$ if and only if $\{ \mathscr{H}_{k} \mid \mathscr{H} \in \mathfrak{H} \} \nleqslant^{*} \{ G \}.$
Note that, since $\mathfrak{H}$ is effective, the set $\{ \mathscr{H}_{k} \mid \mathscr{H} \in \mathfrak{H} \}$ can be computed in time depending on $k.$
Then, using the \FPT $\leqslant$-containment algorithm, we can check whether $\{ \mathscr{H}_{k} \mid \mathscr{H} \in \mathfrak{H} \} \nleqslant^{*} \{ G \}.$ 
This implies that $\p_{\mathfrak{H}}$ is \FPT-decidable. 
Then our claim holds by \autoref{thmfpt}.
\end{proof}

In general, we have no way to additionally guarantee that some particular $\p$ in $[\p]_{\sim}$ has a non-empty recursive $\leqslant$-core. 
Hence, it is important to determine for which graph parameters this is the case, which would imply its \FPT-approximability, by \autoref{thmuobsfpt}.

\medskip
The algorithms in the proofs of \autoref{thmfpt} and \autoref{thmuobsfpt} provide the following conditions in order to construct the implied approximation algorithm.
\begin{corollary} If $\p$ is a $\leqslant$-monotone parameter that has a finite $\leqslant$-universal obstruction $\mathfrak{H}$, and if we are given
\begin{enumerate}
\item an algorithm that, given $k \in \mathbb{N}$, outputs $\{ \mathscr{H}_{k} \mid \mathscr{H} \in \mathfrak{H}\},$
\item an algorithm computing the gap $g$ in the relation $\p\sim_{g}\p_{\mathfrak{H}}$,
\item and an effective \FPT-algorithm for $Π_{\leqslant}$,
\end{enumerate}
then $\p$ is effectively \FPT-approximable in time $f(k)\cdot |G|^{O(1)}$, for some computable function $f$.
\end{corollary}

Notice that the algorithms of conditions \emph{1.}~and \emph{2.} exist if and only if $\p$ has a non-empty recursive $\leqslant$-core.
Therefore Condition \emph{3.} is the only additional requirement for the effectivity of the \FPT-approximability of $\p$ in \autoref{thmuobsfpt}.

\section{Towards finite universal obstructions}\label{general_section}

In \autoref{asympotic_view} we saw the central role that finite $\leqslant$-parametric families play, in obtaining finite $\leqslant$-universal obstructions for $\leqslant$-rational parameters as well as for representing directed $\leqslant$-class properties whose $\leqslant$-class obstruction sets are finite.
In this section, our goal is to identify appropriate order theoretic assumptions on $(\gall, \leqslant)$ implying that finite $\leqslant$-parametric families suffice to characterize \textsl{every} $\leqslant$-monotone parameter in terms of a finite $\leqslant$-universal obstruction.
Moreover, we identify a necessary and sufficient order theoretic condition that implies a representation of \textsl{every} $\leqslant$-class property in terms of a finite collection of $\leqslant$-parametric families.

\subsection{More implications of well-quasi-ordering}

In this first subsection we consider the implications on the existence of $\leqslant$-universal obstructions when $(\gall, \leqslant)$ is a well-quasi-ordering.

\paragraph{Obstruction chains.}

Consider a $\leqslant$-monotone parameter $\p$ and let $I \in \{ [n] \mid n \in \mathbb{N} \} \cup \{ \mathbb{N} \}$ be a countable index set. 
Let $\mathscr{C} = \langle \mathscr{C}_{i} \rangle_{i \in I}$ be a $\leqslant$-chain such that for every $i \in I,$ $\mathscr{C}_{i} \in \obs_{\leqslant}(\mathcal{G}_{\p, i}).$ 
We call $\mathscr{C}$ a \emph{$\leqslant$-obstruction chain} of $\p.$ 
In the case that $I = [n]$ for some $n \in \mathbb{N}$ we say that $\mathscr{C}$ is finite and of length $|I|,$ while in the case that $I = \mathbb{N}$ we say it is infinite. 
Given a $\leqslant$-monotone parameter $\p$, we define
$$\mathfrak{C}_{\leqslant}(\p) \coloneqq \{ \mathscr{C} \mid \textrm{$\mathscr{C}$ is a $\leqslant$-obstruction chain of $\p$} \}$$
and from this point forward let $\mathfrak{F}_{\leqslant}(\p) \cup \mathfrak{I}_{\leqslant}(\p)$ be a partition of $\mathfrak{C}_{\leqslant}(\p)$ into the set of finite ($\mathfrak{F}_{\leqslant}(\p)$) and the set of infinite ($\mathfrak{I}_{\leqslant}(\p)$) $\leqslant$-obstruction chains of $\p.$

\medskip
Recall that, by \autoref{bijection_coll_ideals}, so far we have only been able to fully characterize the $\leqslant$-class obstructions and obtain $\leqslant$-universal obstructions, only for $\leqslant$-monotone parameters whose corresponding directed $\leqslant$-class property has a finite class $\leqslant$-obstruction set.
However, one could attempt to repeat the arguments in the proof of \autoref{bijection_coll_ideals} and try to show that, by taking a $\leqslant$-omnivore for each $\leqslant$-closed class in an infinite inclusion-antichain of $\leqslant$-ideals $\mathbb{I},$ one could define a $\leqslant$-parametric family, where the $\leqslant$-class obstruction of its corresponding parameter is precisely $\mathbb{I}.$
In this way, we could argue that we find a certificate of a $\leqslant$-monotone parameter that has an infinite $\leqslant$-universal obstruction and that captures a directed $\leqslant$-class property with an infinite $\leqslant$-class obstruction set.
However, this approach seems to fail, and the intuitive reason why it does seems to lie in our inability to certify that, the two equivalence relations $\equiv^{*}$ and $\equiv^{*}_{f}$ on the set of $\leqslant$-parametric families, are equivalent in the presence of infinite $\leqslant$-parametric families.

Therefore, should we wish to prove the existence of $\leqslant$-universal obstructions for arbitrary $\leqslant$-monotone parameters we need to be more careful in how we define our $\leqslant$-parametric graphs.
In the theorem that follows, we demonstrate that, for every $\leqslant$-monotone parameter, its $\leqslant$-class obstruction consists of the inclusion-minimal classes of graphs that corresponds to the $\leqslant$-closures of the infinite $\leqslant$-obstructions chains of $\p$.
Moreover, this set of minimal infinite $\leqslant$-obstruction chains of $\p$ corresponds to a $\leqslant$-parametric family that is a $\leqslant$-universal obstruction for $\p.$

\begin{theorem}\label{idealbasis}\label{cobs_to_uobs}
Let $(\gall, \leqslant)$ be a well-quasi-ordering.
Then, for every $\leqslant$-monotone parameter $\p$, any $\lesssim$-minimization $\mathfrak{Z}$ of the set $\mathfrak{I}_{\leqslant}(\p)$ is a $\leqslant$-universal obstruction for $\p.$
Moreover,
$$\cobs_{\leqslant}(\p) = \bigcup_{\mathscr{Z} \in \mathfrak{Z}} \{ \{ \closure{\leqslant}{\mathscr{Z}}\} \}.$$ 
\end{theorem}
\begin{proof}
Since $(\gall, \leqslant)$ is a well-quasi-ordering, by \autoref{wqo_wellfouned}, $(\down{\leqslant}{\gall}, \subseteq)$ is well-founded.
Then, by \autoref{obsprime}, $\lesssim$-minimizations are well-defined.
Let $\mathfrak{Z}$ be any $\lesssim$-minimization of the set $\mathfrak{I}_{\leqslant}(\p).$
Again, by \autoref{obsprime}, we have that the set $\mathbb{Z} \coloneqq \{ \closure{\leqslant}{\mathscr{Z}} \mid \mathscr{Z} \in \mathfrak{Z} \}$ is an inclusion-antichain.
We first demonstrate that $\cobs_{\leqslant}(\p) = \mathbb{Z}.$

Let $\mathcal{G} \in \UNB_{\leqslant}(\p).$ We show that there exists $\mathscr{I} \in \mathfrak{I}_{\leqslant}(\p)$ such that $\closure{\leqslant}{\mathscr{I}} \subseteq \mathcal{G}.$
For every $k \in \mathbb{N},$ since $\p$ is unbounded in $\mathcal{G},$ there exists $G_{k} \in \obs_{\leqslant}(\mathcal{G}_{\p, k})$ such that $G_{k} \in \mathcal{G}.$ For $i \leqslant j$ we have that $\mathcal{G}_{\p,i} \subseteq \mathcal{G}_{\p,j}.$ By \autoref{subobs}, we have that $\obs_{\leqslant}(\mathcal{G}_{\p,i}) \leqslant^{*} \obs_{\leqslant}(\mathcal{G}_{\p,j}).$ Then, by definition of $\leqslant^{*},$ there exists a finite $\leqslant$-chain $\mathscr{F}^{(k)} \in \mathfrak{F}(\p)$ of length $k$ whose $\leqslant$-maximum element is $G_{k}.$ We define an infinite $\leqslant$-obstruction chain $\mathscr{I}$ of $\p$ as follows:

Let $\mathfrak{F}^{0} = \{ \mathscr{F}^{(k)} \mid k \in \mathbb{N}\}$ and $\mathcal{B}^{1} = \{ H \in \obs_{\leqslant}(\mathcal{G}_{\p,1}) \mid \exists \mathscr{F} \in \mathfrak{F}^{0} : H = \mathscr{F}_{1}\}.$ 
Choose $\mathscr{I}_{1}$ from $\mathcal{B}^{1}$ such that $\{ \mathscr{F} \in \mathfrak{F}^{0} \mid \mathscr{F}_{1} = \mathscr{I}_{1} \}$ is infinite. 
This choice is valid, since $\mathcal{B}^{1}$ is finite (by the assumption that $(\gall, \leqslant)$ is a well-quasi-ordering) and $\mathfrak{F}^{0}$ is infinite. 

Now, assume that we have defined $\mathfrak{F}^{0}, \dots, \mathfrak{F}^{i-1},$ $\mathcal{B}^{1}, \dots, \mathcal{B}^{i}$ and $\mathscr{I}_{1}, \dots, \mathscr{I}_{i},$ for some $i > 0,$ such that $\{ \mathscr{F} \in \mathfrak{F}^{i-1} \mid \mathscr{F}_{i} = \mathscr{I}_{i} \}$ is infinite. 
Define $\mathfrak{F}^{i} = \{ \mathscr{F} \in \mathfrak{F}^{i-1} \mid \mathscr{F}_{i} = \mathscr{I}_{i} \}$ and $\mathcal{B}^{i+1} = \{ H \in \obs_{\leqslant}(\mathcal{G}_{\p,i+1}) \mid \exists \mathscr{F} \in \mathfrak{F}^{i} : H = \mathscr{F}_{i+1} \}.$
Choose $\mathscr{I}_{i+1}$ from $\mathcal{B}^{i+1}$ such that $\{ \mathscr{F} \in \mathfrak{F}^{i} \mid \mathscr{F}_{i+1} = \mathscr{I}_{i+1} \}$ is infinite.
As before, the choice is valid, since $\mathcal{B}^{i+1}$ is finite (by the assumption that $(\gall, \leqslant)$ is a well-quasi-ordering) and $\mathfrak{F}^{i}$ is infinite (by the assumption that $\{ \mathscr{F} \in \mathfrak{F}^{i-1} \mid \mathscr{F}_{i} = \mathscr{I}_{i} \}$ is infinite).
By the axiom of  choice, $\mathscr{I}$ is well-defined and by its definition, $\mathscr{I} \in \mathfrak{I}_{\leqslant}(\p)$ and $\closure{\leqslant}{\mathscr{I}} \subseteq \mathcal{G}.$

Let $\mathbb{I}(\p) = \{ \closure{\leqslant}{\mathscr{I}} \mid \mathscr{I} \in \mathfrak{I}(\p) \}.$ Recall that by definition, every $\mathcal{I} \in \mathbb{I}(\p)$ contains a graph in $\obs_{\leqslant}(\mathcal{G}_{\p, k})$ for every $k \in \mathbb{N}.$ Hence, $\p$ is unbounded for any $\mathcal{I} \in \mathbb{I}(\p),$ i.e. $\mathbb{I}(\p) \subseteq \UNB_{\leqslant}(\p).$ Since for every $\mathcal{G} \in \UNB_{\leqslant}(\p),$ we have that there exists $\mathcal{I} \in \mathbb{I}(\p)$ such that $\mathcal{I} \subseteq \mathcal{G},$ it must be that $\cobs_{\leqslant}(\p) = \mathbb{Z},$ since $\mathbb{Z} = \m{\subseteq}{\mathbb{I}(\p)}$.

We now argue that $\mathfrak{Z}$ is indeed a $\leqslant$-universal obstruction for $\mathsf{p}.$
By \autoref{def:univ_obs} and \autoref{parequivsamebounded}, it suffices to show that $\B_{\leqslant}(\p) = \B_{\leqslant}(\p_{\mathfrak{Z}}).$
Let $\mathcal{G}$ be a $\leqslant$-closed class in $\B_{\leqslant}(\p_{\mathfrak{Z}}).$
This implies that there exists $c \in \Nbbb$ such that for every graph $G \in \mathcal{G},$ $\p_{\mathfrak{Z}}(G) \leq c.$
By definition of $\p_{\mathfrak{Z}},$ this implies that for every graph $G \in \mathcal{G},$ $\{ \mathscr{Z}_{c} \mid \mathscr{Z} \in \mathfrak{Z} \} \nleqslant^{*} \{ G \}.$
By definition of $\leqslant^{*},$ this implies that for every graph $G \in \mathscr{G},$ $\mathscr{Z}_{c} \not\leqslant G,$ and since $\mathcal{G}$ is $\leqslant$-closed, this implies that $\mathscr{Z}_{c} \not\in \mathcal{G},$ for every $\mathscr{Z} \in \mathfrak{Z}.$
Therefore, $\mathcal{G} \in \B_{\leqslant}(\p).$
For the reverse direction, assume that $\mathcal{G} \in \B_{\leqslant}(\p).$
This implies that there exists $c \in \Nbbb$ such that for every graph $G \in \mathcal{G},$ $\p(G) \leq c.$
In other words, for every graph $G \in \mathcal{G},$ $G \in \mathcal{G}_{\p,c}.$
This implies that for every graph $G \in \mathcal{G},$ $Z \not\leqslant G,$ for every $Z \in \obs_{\leqslant}(\mathcal{G}_{\p, c}).$
Now, it suffices to observe that by definition of $\leqslant$-obstruction chains, $\{ \mathscr{Z}_{c} \mid \mathscr{Z} \in \mathfrak{Z} \} \subseteq \obs_{\leqslant}(\mathcal{G}_{\p, c}).$
Combining the two previous remarks we have that, for every graph $G \in \mathcal{G},$ $\mathscr{Z}_{c} \not\leqslant G,$ for every $\mathscr{Z} \in \mathfrak{Z}.$
With this we conclude that $\mathcal{G} \in \B_{\leqslant}(\p_{\mathfrak{Z}}).$
\end{proof}



From the algorithmic point of view, under the well-quasi-ordering assumption, we get the following stronger version of \autoref{finiteuobs_rational}.

\begin{corollary}\label{corollary_wqo_fullcore}
Let $(\gall, \leqslant)$ be a well-quasi-ordering and $\leqslant$ be \FPT-decidable.
Then, a $\leqslant$-monotone parameter $\p$ is $\leqslant$-rational if and only if it has an effective $\leqslant$-universal obstruction.
\end{corollary}

\begin{proof}
Since $(\gall, \leqslant)$ is a well-quasi-ordering, for every $\mathcal{G} \in \down{\leqslant}{\gall},$ $\obs_{\leqslant}(\mathcal{G})$ is finite.
Then, the forward direction is implied by \autoref{corollary_fullcore}, since $\core_{\leqslant}(\p) \neq \emptyset.$
The reverse direction is trivially implied by \autoref{finiteuobs_rational}.
\end{proof}

Furthermore, under the well-quasi-ordering assumption we can additionally show that, every $\leqslant$-monotone parameter is $\leqslant$-real, i.e., corresponds to a graph sequence, in the sense that it is equivalent to the parameter corresponding to said sequence.

\begin{theorem}
Let $(\gall, \leqslant)$ be a well-quasi-ordering.
Then, every $\leqslant$-monotone parameter is $\leqslant$-real.
\end{theorem}
\begin{proof}
Let $\p$ be a $\leqslant$-monotone parameter.
We define the graph sequence $\mathscr{H} := \langle \obs(\mathcal{G}_{\p, k}) \rangle_{k \in \mathbb{N}}.$
We first prove the following claim.

\begin{claim}\label{claimreal}
There exists a function $f \colon \Nbbb \to \Nbbb$ such that $\m{\leqslant}{\mathscr{H}_{\geq f(k)}} = \obs_{\leqslant}(\mathcal{G}_{\p, k}).$
\end{claim}
\begin{cproof}
We define the desired function $f \colon \Nbbb \to \Nbbb$ as follows.
Fix $k \in \mathbb{N}.$
We define $f(k)$ so that $\mathscr{H}_{\geq f(k)} \cap \bigcup_{j < k} \obs(\mathcal{G}_{\p,j}) = \emptyset.$
Let $i > k.$
Since $\p$ is $\leqslant$-monotone, we have that $\mathcal{G}_{\p, k} \subseteq \mathcal{G}_{\p, i}.$
Then, by \autoref{subobs}, we have that $\obs(\mathcal{G}_{\p,k}) \leqslant^{*} \obs(\mathcal{G}_{\p, i}).$
Now, observe that $\mathscr{H}_{\geq f(k)} = \bigcup_{j \geq k} \obs(\mathcal{G}_{\p, j}).$
Thus we have that for every $H \in \mathscr{H}_{\geq f(k)}$ there exists $H' \in \obs(\mathcal{G}_{\p, k})$ such that $H' \leqslant H$ and moreover, since $\obs(\mathcal{G}_{\p, k})$ is an $\leqslant$-antichain, $\m{\leqslant}{\mathscr{H}_{\geq f(k)}} = \obs(\mathcal{G}_{\p,k}).$
\end{cproof}

We show that $\p \sim \p_{\mathscr{H}}.$ Let $G$ be a graph such that $\p(G) \leq k.$
Then for every $H \in \obs_{\leqslant}(\mathcal{G}_{\p, k}),$ $H \nleqslant G.$
By \autoref{claimreal}, $\m{\leqslant}{\mathscr{H}_{\geq f(k)}} = \obs(\mathcal{G}_{\p,k}).$
Then, by definition of $\p_{\mathfrak{H}},$ it is implied that $\p_{\mathscr{H}}(G) \leq f(k).$
Now, let $G$ be a graph such that $\p_{\mathscr{H}}(G) \leq k.$ 
Let $i \in \mathbb{N}$ be maximum such that $\obs(\mathcal{G}_{\p, i}) \cap \{ \mathscr{H}_{j} \mid j \in [k] \} \neq \emptyset.$
By \autoref{claimreal}, $\m{\leqslant}{\mathscr{H}_{\geq f(i+1)}} = \obs(\mathcal{G}_{\p, i+1}).$ 
By definition of $\mathscr{H},$ $f(i+1) \geq k.$ Then, clearly $\m{\leqslant}{\mathscr{H}_{\geq k}} \leqslant^{*} \m{\leqslant}{\mathscr{H}_{\geq f(i+1)}}$ and thus $\obs(\mathcal{G}_{\p, i+1}) \nleqslant^{*} \{ G \},$ which implies $\p(G) \leq i+1.$
\end{proof}

Moreover, we show that, under the well-quasi-ordering assumption, every graph sequence (not necessarily rational as in \autoref{rationalperiodic}) also corresponds to a $\leqslant$-parametric family with respect to equivalence of their corresponding parameters.

\begin{theorem}
Let $(\gall, \leqslant)$ be a well-quasi-ordering.
Then, for every graph sequence $\mathscr{H}$ there exists a $\leqslant$-parametric family $\mathfrak{H}$ such that $\p_{\mathscr{H}} \sim \p_{\mathfrak{H}}.$
\end{theorem}
\begin{proof}
Let $\mathfrak{H}$ be a $\leqslant$-parametric family defined as any $\lesssim$-minimization of the set of infinite obstruction $\leqslant$-chains in $\mathfrak{I}(\p_{\mathscr{H}})$.
Note that, since $(\gall, \leqslant)$ is a well-quasi-ordering, by \autoref{wqo_wellfouned}, $(\down{\leqslant}{\gall}, \subseteq)$ is well-founded.
Then, by \autoref{obsprime}, $\lesssim$-minimizations are well-defined.
Then, our claim follows by \autoref{idealbasis}, since it implies that $\mathfrak{H}$ is a $\leqslant$-universal obstruction of $\p_{\mathscr{H}}$ and therefore $\p_{\mathscr{H}} \sim \p_{\mathfrak{H}}.$
\end{proof}

\subsection{Finite representation of class properties}
\label{chakmrkfinfoerobes}

In the previous subsection we saw that under the well-quasi-ordering assumption, every $\leqslant$-monotone parameter admits a $\leqslant$-universal obstruction.
In this subsection, we establish sufficient order theoretic conditions that moreover certify the existence of finite $\leqslant$-universal obstructions for all $\leqslant$-monotone parameters.
Moreover, we study the potential of these order-theoretic assumptions, paired with our viewpoint of $\leqslant$-parametric families, in obtaining a representation of every $\leqslant$-class property via a finite set of finite $\leqslant$-parametric families.

\medskip
As we saw in \autoref{idealbasis}, under the assumption that $(\gall, \leqslant)$ is a well-quasi-ordering, we can show that every $\leqslant$-monotone parameter admits a $\leqslant$-universal obstruction.
However, we would like to identify a sufficient order-theoretic condition which furthermore implies the existence of finite $\leqslant$-universal obstructions.
Such a condition of course has to be ``stronger'' that the well-quasi-ordering assumption.

Therefore it is natural to consider whether the ``second order'' well-quasi-ordering notion, i.e., an $\omega^{2}$-well-quasi-ordering suffices for our purposes.
Towards this direction, \autoref{wqo_implies_omega2_equiv_class_obs_finite} tells us that assuming that $(\gall, \leqslant)$ is an $\omega^{2}$-well-quasi-ordering, suffices to show that the $\leqslant$-class obstruction set of every $\leqslant$-class property is finite.

The observation above directly allows for a generalization of \autoref{finite_directed_properties_representable}.

\begin{theorem}\label{directed_properties_representable}
Let $(\gall, \leqslant)$ be an $\omega^{2}$-well-quasi-ordering and $\mathbb{CP}$ be any $\leqslant$-class property.
Then, $\mathbb{CP}$ is representable by a finite $\leqslant$-parametric family if and only $\mathbb{CP}$ is directed.
\end{theorem}

Now, let $\mathbf{F}_{\leqslant}$ denote the set of all equivalence classes of $\equiv^{*}_{f}$, namely,
$$\mathbf{F}_{\leqslant} \coloneqq \{ [\mathfrak{H}]_{f} \mid \text{$\mathfrak{H}$ is a $\leqslant$-parametric family}\}$$
and let $\mathbf{I}_{\leqslant}$ denote the set of all inclusion-antichains of $\leqslant$-ideals of $\gall,$ namely,
$$\mathbf{I}_{\leqslant} \coloneqq \{ \mathbb{I} \in 2^{\idl{\leqslant}{\gall}} \mid \text{$\mathbb{I}$ is an inclusion-antichain}\}.$$

\autoref{wqo_implies_omega2_equiv_class_obs_finite} has as a consequence that $\mathbf{I}_{\leqslant}$ is equal to its restricted counterpart $\mathbf{I}^{\mathsf{fin}}_{\leqslant},$ which we introduced in \autoref{subsec_omnivores}, and via \autoref{obsprime}, that $\mathbf{F}_{\leqslant}$ is also equal to its restricted counterpart $\mathbf{F}^{\mathsf{fin}}_{\leqslant}.$
This implies that the bijection $\rho_{\ref{bijection_coll_ideals}} \colon \mathbf{F}^{\mathsf{fin}}_{\leqslant} \to \mathbf{I}^{\mathsf{fin}}_{\leqslant}$ extends to a bijection between $\mathbf{F}_{\leqslant}$ and $\mathbf{I}_{\leqslant}.$

\begin{observation}\label{parametric_families_correspond_ideals}
Let $(\gall, \leqslant)$ be an $\omega^{2}$-well-quasi-ordering.
Then, $\rho_{\ref{bijection_coll_ideals}} \colon \mathbf{F}_{\leqslant} \to \mathbf{I}_{\leqslant}$ is a unique bijection such that, for every $[\mathfrak{G}]_{f} \in \mathbf{F}_{\leqslant},$ for every $\mathfrak{H} \in [\mathfrak{G}]_{f},$ $\cobs_{\leqslant}(\p_{\mathfrak{H}}) = \rho_{\ref{bijection_coll_ideals}}([\mathfrak{G}]_{f}).$
\end{observation}

Also, let $\mathbf{P}_{\leqslant}$ denote the set of all equivalence classes (with respect to $\sim$) of $\leqslant$-monotone parameters, namely,
$$\mathbf{P}_{\leqslant} \coloneqq \{ [\p]_{\sim} \mid \textrm{$\p$ is a $\leqslant$-monotone parameter}\}.$$

Now, there are a number of ways to see that, the $\omega^{2}$-well-quasi-ordering assumption, also implies that $\mathbf{P}_{\leqslant}$ is equal to its restricted counterpart $\mathbf{P}_{\leqslant}^{\mathsf{rational}},$ i.e., that every $\leqslant$-monotone parameter is $\leqslant$-rational.
One way is to observe that, by \autoref{params_obs_ideals}, the $\leqslant$-class obstruction of every $\leqslant$-monotone parameter is an inclusion-antichain of $\leqslant$-ideals of $\gall.$
Since, by \autoref{wqo_implies_omega2_equiv_class_obs_finite}, every such inclusion-antichain is finite, it is implied that the bijection $\rho_{\ref{bijection_rational_ideals}} \colon \mathbf{P}_{\leqslant}^{\mathsf{rational}} \to \mathbf{I}_{\leqslant}^{\mathsf{fin}}$ extends to a bijection between $\mathbf{P}_{\leqslant}$ and $\mathbf{I}_{\leqslant}.$

\begin{observation}
Let $(\gall, \leqslant)$ be an $\omega^{2}$-well-quasi-ordering.
Then, $\rho_{\ref{bijection_rational_ideals}} \colon \mathbf{P}_{\leqslant} \to \mathbf{I}_{\leqslant}$ is a unique bijection such that, for every $[\p]_{\sim} \in \mathbf{P}_{\leqslant}$, for every $\p \in [\p]_{\sim}$, $\cobs_{\leqslant}(\p) = \rho_{\ref{bijection_rational_ideals}}([\p]_{\sim})$.
Moreover, $\mathbf{P}_{\leqslant}$ is countable.
\end{observation}

\paragraph{Finite universal obstructions.}

As a corollary we obtain that the $\omega^{2}$-well-quasi-ordering assumption suffices to prove that every $\leqslant$-monotone parameter has a finite $\leqslant$-universal obstruction.

\begin{theorem}\label{omega2_finite_univ_obs} Let $(\gall, \leqslant)$ be an $\omega^{2}$-well-quasi-ordering.
Then, every $\leqslant$-monotone parameter has a finite $\leqslant$-universal obstruction.
\end{theorem}

From the algorithmic point of view, under the $\omega^{2}$-well-quasi-ordering, we get the following variant of \autoref{omega2_finite_univ_obs}, which is directly implied by \autoref{omega2_finite_univ_obs} and the fact that $\leqslant$-obstruction sets are finite for every $\leqslant$-closed class.

\begin{corollary}\label{corollary_wqo2_fullcore}
Let $(\gall, \leqslant)$ be an $\omega^{2}$-well-quasi-ordering and $\leqslant$ be \FPT-decidable.
Then, every $\leqslant$-monotone parameter $\p$ has an effective $\leqslant$-universal obstruction.
\end{corollary}

It is worthwhile to stress that even under the $\omega^{2}$-well-quasi-ordering assumption, \autoref{corollary_wqo2_fullcore} still does not guarantee the \FPT-approximability of $\p$ implied by \autoref{thmuobsfpt}.
For this, we morever need to find an efficient universal obstruction with a gap that is recursive.
We have no guarantees that such a universal obstruction exists. 
In practice, for explicit proofs where we find universal obstructions for parameters, these are always effective and the gap is typically constructive, so they can be used for deriving effective \FPT-approximation algorithms.

\paragraph{Representation of class properties.}

To conclude, we saw in the form of \autoref{directed_properties_representable} that, under the $\omega^{2}$-well-quasi-ordering assumption, every directed $\leqslant$-class property is representable by a finite $\leqslant$-parametric family.
We can in fact show that we can represent every $\leqslant$-class property by a finite set of finite $\leqslant$-parametric families.

\begin{theorem}\label{class_properties_rep_parametric_families} Let $(\gall, \leqslant)$ be an $\omega^{2}$-well-quasi-ordering.
Then, every $\leqslant$-class property is representable by a finite $\lesssim^{*}$-antichain of finite $\leqslant$-parametric families.
\end{theorem}
\begin{proof}
By \autoref{master_template} (\emph{1.} $\to$ \emph{4.}), every $\leqslant$-class property $\mathbb{CP}$ is equal to the union of a finite inclusion-antichain $\{ \mathbb{D}^{1}, \ldots, \mathbb{D}^{c} \},$ $c \in \Nbbb,$ of directed $\leqslant$-class properties.
By \autoref{wqo_implies_directed_ideal}, for every $i \in [c],$ $\cobs_{\leqslant}(\mathbb{D}^{i})$ is an inclusion-antichain of $\leqslant$-ideals of $\gall,$ which by \autoref{wqo_implies_omega2_equiv_class_obs_finite}, is moreover finite.
Then, by \autoref{parametric_families_correspond_ideals}, $\rho_{\ref{bijection_coll_ideals}}^{-1}(\{ \mathbb{D}^{1}, \ldots, \mathbb{D}^{c} \}) = \{ [\mathfrak{G}^{1}]_{f}, \ldots, [\mathfrak{G}^{c}]_{f} \}$ is a set of distinct equivalence classes of $\equiv^{*}_{f}$, such that, for every $\mathfrak{H} \in [\mathfrak{G}^{i}],$ $i \in [c],$ $\cobs_{\leqslant}(\p_{\mathfrak{H}}) = \rho_{\ref{bijection_coll_ideals}}([\mathfrak{G}^{i}]) = \cobs_{\leqslant}(\mathbb{D}^{i}).$
Now, we define the set $\mathbb{H} \coloneqq \{ \mathfrak{H}^{i} \mid i \in [c] \},$ by arbitrarily choosing some $\mathfrak{H}^{i} \in [\mathfrak{G}^{i}],$ $i \in [c].$
Clearly, $\mathbf{H}$ is a $\lesssim^{*}$-antichain.
Moreover, $\bigcup_{i \in [c]} \B_{\leqslant}(\p_{\mathfrak{H}^{i}}) = \bigcup_{i \in [c]} \mathbb{D}^{i} = \mathbb{CP}$ and therefore we conclude that $\mathbb{CP}$ is representable by the finite $\lesssim^{*}$-antichain of finite $\leqslant$-parametric families $\mathbf{H}$.
\end{proof}

In fact, under the well-quasi-ordering assumption, we can prove that both \autoref{directed_properties_representable} and \autoref{class_properties_rep_parametric_families} are equivalent to $(\gall, \leqslant)$ being an $\omega^{2}$-well-quasi-ordering.

\begin{theorem}\label{final_result}
Let $(\gall, \leqslant)$ be a well-quasi-ordering.
Then, the following statements are equivalent:
\begin{enumerate}
\item $(\gall, \leqslant)$ is an $\omega^{2}$-well-quasi-ordering;
\item every directed $\leqslant$-class property is representable by a finite $\leqslant$-parametric family;
\item every $\leqslant$-class property is representable by a finite $\lesssim^{*}$-antichain of finite $\leqslant$-parametric families.
\end{enumerate}
\end{theorem}
\begin{proof}
We first show the equivalence between \emph{1.} and \emph{2.}.
By \autoref{directed_properties_representable}, \emph{1.} directly implies \emph{2.}
For the reverse direction, let $\mathbb{CP}$ be a directed $\leqslant$-class property representable by a finite $\leqslant$-parametric family $\mathfrak{H}.$
Since $\mathbb{CP} = \B_{\leqslant}(\p_{\mathfrak{H}}),$ we have that $\cobs_{\leqslant}(\mathbb{CP}) = \cobs_{\leqslant}(\p_{\mathfrak{H}})$ and therefore $\cobs_{\leqslant}(\mathbb{CP})$ is finite.
Therefore, by \autoref{master_template} (\emph{3.} $\to$ \emph{1.}) $(\gall, \leqslant)$ is an $\omega^{2}$-well-quasi-ordering.

Now, we show the equivalence between \emph{1.} and \emph{3.}.
By \autoref{class_properties_rep_parametric_families}, \emph{1.} directly implies \emph{3}.
For the reverse direction, let $\mathbb{CP}$ be a $\leqslant$-class property representable by a finite $\lesssim^{*}$-antichain $\{ \mathfrak{H}^{1}, \ldots, \mathfrak{H}^{c} \},$ $c \in \Nbbb$ of finite $\leqslant$-parametric families.
By definition of representability, we have that $\mathbb{CP} = \bigcup_{i \in [c]} \B_{\leqslant}(\p_{\mathfrak{H}^{i}}).$
Now, note that by \autoref{bp_is_directed}, $\B_{\leqslant}(\p_{\mathfrak{H}}^{i})$ is a directed $\leqslant$-class property, for every $i \in [c].$
Next, we argue that $\B_{\leqslant}(\p_{\mathfrak{H}^{i}})$ is inclusion-incomparable with $\B_{\leqslant}(\p_{\mathfrak{H}^{j}}),$ for every $i \neq j \in [c].$
By \autoref{finite_param_families_equiv_params}, we have that $\p_{\mathfrak{H}^{i}} \not\preceq \p_{\mathfrak{H}^{j}}$ and
$\p_{\mathfrak{H}^{j}} \not\preceq \p_{\mathfrak{H}^{i}}.$
Then, by \autoref{parsimbounded}, we conclude that $\B_{\leqslant}(\p_{\mathfrak{H}^{i}}) \not\subseteq \B_{\leqslant}(\p_{\mathfrak{H}^{j}})$ and $\B_{\leqslant}(\p_{\mathfrak{H}^{j}}) \not\subseteq \B_{\leqslant}(\p_{\mathfrak{H}^{i}}).$
Now, we may conclude by applying \autoref{master_template} (\emph{4.} $\to$ \emph{1.}).
\end{proof}

\section{Some examples of obstructions for graph parameters}\label{examples_section}
 
In this section we present obstructions for a small indicative sample of graph parameters that are monotone under certain partial-ordering relations on graphs.
All results of this section are consequences of known results for the graph parameters that we present.

\subsection{One more example of a minor-monotone parameter} 
\label{bicpw}

Throughout the previous sections we presented the minor-monotone parameters \pw and \bipw.
In the next paragraph we present one more (perhaps the most famous one) example of such a graph parameter.

\paragraph{Treewidth.}

The graph parameter treewidth was defined by Robertson and Seymour in \cite{RobertsonS84GraphminorsIII} and used as a cornerstone parameter of their Graph Minors series (see also \cite{BerteleB73onno}).
Among the (many) equivalent definitions of treewidth we pick the following that is similar to the definition of pathwidth that we gave in \autoref{gt_conc}.
The \emph{treewidth} of a graph $G,$ denoted by $\tw(G),$ is defined as the minimum $k$ for which there is a vertex ordering $\langle v_{1},\ldots,v_{n} \rangle$ of the vertices of $G$ such that, for each $i \in [n],$ there are at most $k$ vertices in $\{ v_{1}, \ldots, v_{i-1} \}$ that are adjacent with vertices in the connected component of $G[\{v_{i}, \ldots, v_{n} \}]$ that contains $v_{i}$ (see \cite{DendrisKT97FugitiveSearch}).

According to the \textsl{grid theorem}, proved by Robertson and Seymour in~\cite{RobertsonS86GMV} (see also ~\cite{RobertsonST94Quickly,DiestelJGT99high,ChekuriC16Polynomial,ChuzhoyT21Towards}), there exists a function $f \colon \nton$ such that every graph of treewidth at least $f(k)$ contains as a minor the $(k\times k)$-grid $Γ_{k}.$\footnote{The $(k\times k)$-grid is defined as the Cartesian product of two paths on $k$ vertices.}
As $\tw(Γ_k) \geq k$ (see e.g., \cite[Lemma 88]{Bodlaender98}), it follows that the singleton $\{Γ\},$ where  $Γ = \langle Γ_{k} \rangle_{k \in \mathbb{N}},$ is a $\leqslant_{\mathsf{m}}$-universal obstruction of $\tw.$
Using our terminology, it follows that $\obs_{\leqslant_{\mathsf{m}}}(\Gcal_{\tw, k})$ contains a planar graph, for every $k.$
This means that by considering one planar graph from each $\obs_{\leqslant_{\mathsf{m}}}(\Gcal_{\tw, k}), k\in\mathbb{N},$ we define the obstruction set of the classes $\Gcal_{\p', k}, k \in \Nbbb,$ of a graph parameter $\p'$ that is equivalent to $\tw.$
In other words, among the (exponentially many, according to \cite{Ramachandramurthi97thes}) obstructions in $\obs_{\leqslant_{\mathsf{m}}}(\Gcal_{\tw, k}),$ \textsl{just one of them} is sufficient to determine the approximate behaviour of treewidth and this holds for \textsl{every} other parameter that is equivalent to treewidth (see \cite{HarveyW17ParametersTied}).
The same phenomenon has been observed for many other parameters for which similar approximate characterizations are known.
It is also easy to observe that the class $\gplanar$ of all planar graphs is exactly the class of all minors of grids (see e.g., \cite[(1.5)]{RobertsonST94Quickly}). 
We conclude with the following proposition.

\begin{proposition}
The set $\{\Gamma\}$ is a minor-universal obstruction for $\tw.$
Moreover, $\cobs_{\leqslant_{\mathsf{m}}}(\tw)=\{\gplanar\}$ and $\pobs_{\leqslant_{\mathsf{m}}}(\tw)=\big\{\{K_5,K_{3,3}\}\big\}.$
\end{proposition}

In the spirit of \autoref{relationequiv}, observe that
$$\tw \preceq \bipw \preceq \pw$$
if and only if
$$\{ \gforest \} \subseteq^{*} \{ \gfapex, \gouterplanar \} \subseteq^{*} \{ \gplanar \}$$
if and only if
$$\big\{ \{ K_{3} \} \big\} \leqslant^{**}_{\mathsf{m}} \big\{ \{ K_4, S_3, 2\cdot K_3 \}, \{ K_4, K_{2,3} \} \big\} \leqslant^{**}_{\mathsf{m}} \big\{ \{ K_{5}, K_{3, 3} \} \big\}.$$

\subsection{Two examples of immersion-monotone parameters}\label{immersion_subsection}

In this subsection we consider graphs that may contain multiple edges but no loops. 
For this, given a graph $G,$ we see its edge set as a multiset. 
We use $\gall^{\mathsf{e}}$ to denote the class of all graphs with multiple edges. 

Let $G$ be a graph and let $e_1 = \{x,y\}$ and $e_{2} = \{y,z\}$ be two edges of $G$ where $y$ is a common endpoint of both.
The result of the operation of \emph{lifting} the edges $e_{1}$ and $e_{2}$ in $G$ is the graph obtained from $G$ if we first remove $e_{1}$ and $e_{2}$ from $G$ and then add the edge $\{ x, z \}.$
As we deal with multigraphs, we agree that if any of $e_{1}$ and $e_{2}$ has multiplicity bigger than two, then its removal reduces its multiplicity. 
Also if the edge $\{x, z\}$ already exists, then we just increase its multiplicity by one.
Moreover, if $e_{1}$ and $e_{2}$ are parallel edges then the lifting operation simply removes both of them.
That is because we consider multigraphs with only multiple edges and therefore we suppress all loops that may arise from the lifting operation.

We define the \textsl{immersion} relation on the set of graphs with multiple edges as follows.
We say that a graph $H$ is an \emph{immersion} of a graph $G,$ denoted by $H \leqslant_{\mathsf{i}} G$ if $H$ can be obtained by a subgraph of $G$ after a (possibly empty) sequence of liftings of pairs of edges with common endpoints.
Observe that if $H$ is a topological minor of $G$ then $H$ is also an immersion of $G$ (but not vice versa).
Recall that Robertson and Seymour proved in \cite{RobertsonS10GraphminorsXXIII} that $\leqslant_{\mathsf{i}}$ is a wqo in $\gall^{\mathsf{e}}.$
This directly implies that for every $\mathcal{G} \in \mathsf{Down}_{\leqslant_{\mathsf{i}}}(\gall),$ $\obs_{\leqslant_{\mathsf{i}}}(\mathcal{G})$ is a finite set.

\paragraph{Degree.}

A (simple) example of an immersion-monotone parameter is the \emph{degree} $\Delta,$ 
where $\Delta(G)$ is the maximum number of edges incident to a vertex of $G.$ 
Let $\mathscr{K}^{\mathsf{s}}=\langle K_{1,k} \rangle_{k\in\mathbb{N}}$ be the graph sequence of \emph{star graphs}.
Note that star graphs are simple graphs.
We define $K_{1, 0}$ to be $K_{1}$, i.e., a graph with a single vertex.
Moreover, let $Θ = \langle θ_{k }\rangle_{k \in \mathbb{N}},$ where $θ_{k}$ is the graph on two vertices with $k$ parallel edges. 
Let also $\mathcal{C}^{S}$ be the set of all star forests with at most one component on at least three vertices and $\mathcal{C}^{θ} := \{ θ_{k} \mid k \in \mathbb{N}_{≥1}\} \cup \{K_{1}, K_{0} \}.$
It is not difficult to see that using the terminology introduced in this paper, we get the following:

\begin{proposition} The set $\mathfrak{D} = \{Θ,\mathscr{K}^{\mathsf{s}}\}$ is an immersion-universal obstruction for $\Delta.$ 
Moreover, it holds that $\cobs_{\leqslant_{\mathsf{i}}}(Δ) = \{\mathcal{C}^{θ},\mathcal{C}^{S}\}$ and $\pobs_{\leqslant_{\mathsf{i}}}(Δ) = \big\{ \{3 \cdot K_{1}\}, \{ θ_{2}, 2\cdot P_{3}, P_{4} \} \big\}$
\end{proposition}

Observe that $Θ$ is a $\leqslant_{\mathsf{i}}$-omnivore of $\mathcal{C}^{θ}$ and $\mathscr{K}^{\mathsf{s}}$ is $\leqslant_{\mathsf{i}}$-omnivore of $\mathcal{C}^{S}.$
Also $\mathcal{C}^{θ},\mathcal{C}^{S}$ are immersion-ideals, $\p_{\mathfrak{D}}$ and $\Delta$ are equivalent and $\{\mathcal{C}^{θ},\mathcal{C}^{S}\}$ is a $\lesssim$-antichain. 
Finally, it is easy to observe that $\big\{ \{3 \cdot K_{1}\}, \{ θ_{2}, 2\cdot P_{3}, P_{4} \} \big\}$ is a $\leqslant_{\mathsf{i}}^{*}$-antichain.

\paragraph{Cutwidth.}
The graph parameter \emph{cutwidth}, denoted by $\cw \colon \gall^{\mathsf{e}} \to \mathbb{N},$ is defined so that $\cw(G)$ is the minimum $k$ such that there exists a vertex ordering $\langle v_1, \ldots,v_{n}\rangle$ of $V(G)$ such that for every $i ∈ [n]$ there are at most $k$ edges in $G$ with one endpoint in $\langle v_1, \ldots,v_{i-1}\rangle$  and the other in $\langle v_1, \ldots,v_{i-1}\rangle.$ 

Let $\gtforest$ be the class of subcubic forests, that is the class of all forests of maximum degree at most three. 
Additionally, let $3\cdot K_{1}, θ_{2}, 2\cdot P_{3}, P_{4}$ and $K_{1,4}$ be the graphs depicted in \autoref{cutwidth_pobs}.

\begin{figure}[ht]
\begin{center}
\scalebox{1}{\includegraphics{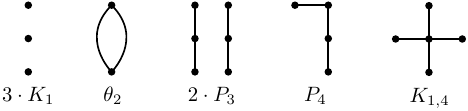}}
\end{center}
\caption{\label{cutwidth_pobs}The graphs in the three immersion-obstruction sets in $\pobs_{\leqslant_{\mathsf{i}}}(\cw).$}
\end{figure}

The results of \cite{ChungS89Graphswith,BienstockRST91Quickly,KirousisP86Searching,RobertsonS83GraphminorsI}, using the terminology introduced in this paper, can be rewritten as follows.

\begin{proposition}
\label{ishmssdwsaal}
The set  $\mathfrak{C} = \{Θ,\mathscr{K}^{\mathsf{s}},\mathscr{T}\}$ is an immersion-universal obstruction for $\cw.$ 
Moreover, it holds that $\cobs_{\leqslant_{\mathsf{i}}}(\cw) = \{\mathcal{C}^{θ}, \mathcal{C}^{S}, \gtforest\}$ and $\pobs_{\leqslant_{\mathsf{i}}}(\cw) = \big\{ \{3 \cdot K_{1}\}, \{θ_{2}, 2 \cdot P_{3}, P_{4}\}, \{θ_2, K_{1,4}\} \big\}.$
\end{proposition}

Observe that also $\mathscr{T}$ is an immersion-omnivore of the class $\gtforest$ of subcubic forests.
Also $\gtforest$ is an immersion-ideal, $\p_{\mathfrak{C}}$ and $\cw$ are equivalent, and $\{\mathcal{C}^{θ},\mathcal{C}^{S},\gtforest\}$ is a $\lesssim$-antichain. 
Finally, it is easy to observe that $\big\{\{3\cdot K_{1}\},\{θ_{2},2\cdot P_{3}, P_{4} \},\{θ_2,K_{1,4}\}\big\}$ is a $\leqslant_{\mathsf{i}}^{*}$-antichain.

\medskip
In the spirit of \autoref{relationequiv}, observe that
$$\Delta \preceq \cw$$
if and only if
$$\{\mathcal{C}^{θ},\mathcal{C}^{S},\gtforest\} \subseteq^{*} \{\mathcal{C}^{θ},\mathcal{C}^{S}\}$$
if and only if
$$\big\{ \{ 3\cdot K_{1} \}, \{ θ_{2}, 2\cdot P_{3}, P_{4} \}, \{ θ_2, K_{1,4} \} \big\}  \leqslant^{**}_{\mathsf{i}} \big\{ \{ 3\cdot K_{1} \}, \{ θ_{2}, 2\cdot P_{3}, P_{4} \} \big\}.$$

\section{Conclusion}

In this paper we introduced a parametric framework in an attempt to bridge the gap between finite obstruction characterizations for $\leqslant$-class properties and the study of graph parameters.
We remark that all concepts and results of  \autoref{prelim}, \autoref{asympotic_view}, and \autoref{general_section}
can be directly transferred to any kind of (finite) structures where some notion of ``size'' can be defined. This may certainly include directed graphs, hypergraphs, matroids, and other combinatorial structures. 

A first remark is that the definition of a $\leqslant$-universal obstruction for a parameter $\p$ does not require that $\p$ is monotone under some quasi-ordering $\leqslant$ on graphs.
The choice of such a ``$\leqslant$'' for a graph parameter so that $\leqslant$ is a well-quasi-ordering on graphs is not always easy (or even possible).
An indicative example is the graph parameter \emph{tree-partition treewidth}, denoted by $\mathsf{tpw}$ where, given a graph $G,$ $\mathsf{tpw}(G)$  is the minimum $k$ such that there is a partition $\{X_{t} :t \in V(T)\}$ of $V(G),$ indexed by the vertices of a tree $T$ such that (a) $|X_{t}| ≤ k,$ for every $t\in V(T)$ and (b) for every edge $e=\{x,y\}$ of $G,$ if $x\in X_{t}$ and $y\in X_{t'},$ then either $\{t,t'\}\in E(T)$ or $t=t'.$
In \cite{DingO96ontre}, Ding and Oporowski gave a finite $\leqslant_{\mathsf{tp}}$-universal obstruction for $\mathsf{tpw}$ where $\leqslant_{\mathsf{tp}}$ is the topological minor relation.
Interestingly, neither $\mathsf{tpw}(G)$ is topological minor-monotone (see~\cite{DingO96ontre}) nor $\leqslant_{\mathsf{tp}}$ is a wqo on all graphs.
This indicates that, for particular parameters, (finite) universal obstructions may exist even without the general requirements of \autoref{final_result} about $\p$ and $\leqslant$.

As we have seen in \autoref{chakmrkfinfoerobes}, to have a general theory that guarantees the existence of \textsl{finite} $\leqslant$-universal obstructions we need that $\leqslant$ is an $\omega^2$-well-quasi-ordering on graphs.
For the minor ($\leqslant_{\mathsf{m}}$) and the immersion ($\leqslant_{\mathsf{i}}$) relation, this is an important and wide open question in order theory \cite{Requinot2017TowardsBetter, Robertson1995BQOisHard}.
Especially for the minor relation, the greatest known advance in this direction is the celebrated result of Thomas in \cite{Thomas1989wellquasi}, stating that, for every $k \in \mathbb{N},$ $\leqslant_{\mathsf{m}}$ is $\omega^2$-wqo when we restrict to graphs of treewidth at most $k,$ i.e., $(\mathcal{G}_{\tw,k}, \leqslant_{\mathsf{m}})$ is an $\omega^2$-well-quasi-ordering.
In fact, Thomas proved that $\leqslant_{\mathsf{m}}$ is a \textsl{better-quasi-ordering} in the more general case of infinite graphs that exclude as a minor some finite planar graph.
We avoid here the technical definition of better-quasi-ordering; we only mention that it implies $\omega^2$-well-quasi-ordering.
A consequence of this is that when we restrict our universe of graphs to bounded treewidth graphs, it is indeed the case that all minor-universal-obstructions are finite.
We hope that the present work may motivate towards resolving the $\omega^2$-well-quasi-ordering conjecture for the minor relation (and perhaps also for other relations that are known to be well-quasi-orderings on all graphs).

We conclude with the remark that independently of the status of the $\omega^2$-well-quasi-ordering conjecture for the minor or immersion relation, it is still important to develop theories that can automatically produce finite obstructions sets for wide families of parameters.
An example of such a theory is the Erdős–Pósa property for minors: given a (finite) set of graphs $\mathcal{Z}$ we define the parameter $\mathsf{apex}_{\Zcal} \colon \gall \to \mathbb{N}$ such that  $\mathsf{apex}_{\Zcal}(G)$ is the minimum size of a set $S$ of vertices such that none of the graphs in ${\mathcal{Z}}$ is a minor of $G-S$ (that is the graph obtained from $G$ if we remove the vertices of $S$).
According to the results of \cite{RobertsonS86GraphminorsV} (see also \cite{van2019tight} for an optimized version), if $\mathcal{Z}$ contains a planar graph then the set $\{\langle k \cdot Z \rangle_{k\in\mathbb{N}}\mid Z\in \mathcal{Z}\}$ is a (finite) minor-universal obstruction for $\mathsf{apex}_{\Zcal}.$
Employing the parametric viewpoint introduced in this paper, \cite{PaulPTW24Delineating, PaulPTW24Obstructions} extend these results to arbitrary $\Zcal$ as a result of providing a finite minor-universal obstruction for $\Hcal$-treewidth\footnote{$\Hcal$-treewidth is a natural extension of treewidth that measures the tree decomposability of a graph into small pieces and pieces belonging to a fixed class $\Hcal$.} \cite{EibenGHK21}, for every minor-closed class $\Hcal.$
Is it possible to prove similar results about the existence of finite universal obstructions for other families of parameters and obtain finite class obstruction sets for families of class properties?
Such kind of general conditions have been investigated for obstructions of minor-closed graph classes (see e.g., \cite{SauST21apices, LagergrenA91mini, Lagergren98, AdlerGK08comp}).
It is an open challenge whether analogous combinatorial or algorithmic results may be developed for parametric obstructions, even without the a priori knowledge that they are finite.

\paragraph{Acknowledgements:} We wish to thank Sebastian Wiederrecht for all the endless discussions related to the topics and ideas of this paper.

\bibliographystyle{plainurl}


%

\end{document}